\documentclass[10pt]{amsart}

\usepackage{amssymb}

\usepackage{graphicx}

\usepackage{xcolor}
\usepackage{enumerate}
\usepackage{tikz-cd}

\newcommand{\DD}{\Delta\!\!\!\!\Delta}

\newcommand{\res}{\upharpoonright}

\newcommand{\Fraisse}{Fra\"iss\'e}

\theoremstyle{plain}
\newtheorem{theorem}{Theorem}[section]
\newtheorem{corollary}[theorem]{Corollary}
\newtheorem{lemma}[theorem]{Lemma}
\newtheorem{proposition}[theorem]{Proposition}
\newtheorem{definition}[theorem]{Definition}

\newtheorem*{theorem*}{Theorem}
\theoremstyle{remark}

\newtheorem*{claim}{Claim}

\numberwithin{equation}{section}

\begin{document}

\title[The combinatorial simplex]{The generic combinatorial simplex}

\author{A. Panagiotopoulos}
\address{Institut f\"ur Mathematische Logik und Grundlagenforschung, Westfalische Wil\-helms-Universit\"at M\"unster, Einsteinstr. 62, 48149 M\"unster, Germany 
}
\email{aristotelis.panagiotopoulos@gmail.com}
\urladdr{http://apanagiotopoulos.org}

\author{S. Solecki}
\address{Department of Mathematics, Cornell University, Ithaca, NY 14853, USA}
\email{ssolecki@cornell.edu}
\urladdr{https://e.math.cornell.edu/people/ssolecki/}

\thanks{Research of Solecki was supported by NSF grants DMS-1800680 and 1954069.} 

\subjclass[2010]{03C30, 05E45, 55U10, 57N60, 57Q05}
\keywords{Projective \Fraisse{} limit, domination closure, simplex, simplicial complex, stellar move, cellular map, near-homeomorphism}

\begin{abstract}
We employ projective \Fraisse{} theory to define the ``generic combinatorial $n$-simplex'' as the pro-finite, simplicial complex that is canonically associated with a family of 
simply defined selection maps between finite triangulations of the simplex. The generic combinatorial $n$-simplex is a combinatorial  object that can be used to define the geometric realization of a simplicial complex without any reference to the Euclidean space. 
It also reflects dynamical properties of its homeomorphism group down to finite combinatorics.

As part of our study of the generic combinatorial simplex,
we define and prove results on domination closure for \Fraisse{} classes, and 
we develop further the theories of stellar moves and cellular maps. 
We prove that the domination closure of selection maps contains the class of face-preserving simplicial maps that are cellular  
on each face of the $n$-simplex and is contained in the class of simplicial, face-preserving near-homeomorphisms. 
Under the PL-Poincar{\'e} conjecture, this gives a characterization of the domination closure of selections.  
\end{abstract}

\maketitle
\addtocontents{toc}{\setcounter{tocdepth}{1}}

\tableofcontents{}

\section{Introduction}

\Fraisse{} theory  has been extensively used in recent years  to establish connections between combinatorics of 
finite structures and properties of automorphism groups of countable structures; see \cite{Ke} for a survey. Projective \Fraisse{} theory 
can similarly be used to provide natural combinatorial models for compact metrizable spaces and to investigate symmetries of these spaces.
Projective \Fraisse{} theory was introduced in \cite{IrSo}, where it was used to give a combinatorial description of the \emph{pseudo-arc}, 
which in turn was applied  
to study the homeomorphism group of that space. Since then, projective \Fraisse{} theory  has been used in analyzing the dynamics of various homeomorphism groups \cite{BaKw,BaKw1,Kw1,Kw2,PaSo,BaCa}. For example: in \cite{Kw1}, it was applied to show that the homeomorphism group of the \emph{Cantor space} has ample generics; in \cite{BaKw1}, it played a role in computing the universal minimal flow of the homeomorphism group of the \emph{Lelek fan};  in \cite{PaSo}, it was the main ingredient in a new combinatorial proof of the homogeneity of the \emph{Menger sponge}. While  these  examples demonstrate the usefulness of the 
projective \Fraisse{}-theoretic approach, so far the reach of this theory has been almost entirely confined to the study of compacta whose  topological dimension does not exceed 1. Of course,  given the subtleties and complications that arise in higher-dimensional combinatorial topology, this may not come as a surprise.

The main goal of this paper is to produce the projective \Fraisse{} theory of the $n$-dimensional topological simplex. In short, 
we introduce the \emph{generic combinatorial $n$-simplex} $\DD$ as the projective \Fraisse{} limit of a simply defined category  
$\mathcal{S}(\Delta)$ of selection maps, and we prove that the canonical quotient $|\DD|$ of $\DD$ is homeomorphic to the topological 
$n$-dimensional simplex.
This proof makes it necessary to investigate the extent to which the symmetries of $\DD$ approximate arbitrary face-preserving homeomorphisms of $|\DD|$.
The approximation problem, in turn, leads to two further developments that appear to be of independent interest. 

First, we formulate a new general notion of \emph{domination closure} $[{\mathcal F}]$ of 
an arbitrary projective \Fraisse{} category $\mathcal F$, which makes it possible to trade a rigid projective \Fraisse{} category 
for a more flexible one. 
 While $\mathcal{F}$ and $[\mathcal{F}]$ have the  same projective \Fraisse{} limit,  $[\mathcal{F}]$ reflects  more faithfully the intrinsic symmetries of this limit.

Second, we introduce two categories of simplicial maps   $\mathcal{C}(\Delta)\subseteq \mathcal{H}(\Delta)$,  which  reflect  
the piecewise linear and topological natures of  $|\DD|$, respectively. We show that the two categories 
constitute lower and upper bounds for  the domination closure $[\mathcal{S}(\Delta)]$ of  $\mathcal{S}(\Delta)$. 
To prove this, we expand the theory of stellar moves. (Stellar moves provide a purely combinatorial alternative to PL-topology.) In particular, we introduce a stronger notion of \emph{starring} for induced complexes and give a new variant  of the main technical theorem from~\cite{Al,Ne}.

We finish this part of the introduction by making two additional points. 
First, the generic combinatorial $n$-simplex $\DD$ is a purely combinatorial object that gives, via $|{\DD}|$, 
an intrinsic definition of the geometric realization of the $n$-dimensional simplex without any reference to the Euclidean space. 
Second, $\DD$ seems to be the right notion of simplex for  homology theories  appropriate for projective \Fraisse{} limits. We do not explore this theme here but we  expect it to have applications in the study of several compacta; see \cite[Section 6]{PaSo}.

\subsection{An outline of results}\label{Su:outline}

For every finite simplicial complex $A$, we consider the category $\mathcal{S}(A)$ of  {\bf selection maps} whose objects are all finite 
barycentric subdivisions $\beta^n A$ of $A$ and whose morphisms are defined by closing the collection of elementary selections 
$s\colon \beta^{n+1}A\to \beta^n A$ under composition and barycentric subdivision $f\to \beta f$. 
Informally speaking, an elementary selection $s$, as above, is a function that 
maps each old vertex of $\beta^{n+1}A$, that is, a vertex that is in $\beta^nA$, to itself and each 
new vertex $v$ of $\beta^{n+1}A$ to a vertex in $\beta^nA$ that is a neighbor of $v$; see also Figure~1. 
A precise definition of elementary selections 
is given in Section~\ref{S:1}.

\begin{figure}[ht]
\centering
\begin{tabular}{c c}
\begin{tikzpicture}[scale=0.50]
		\node  (0) at (-3, -2) {0};
		\node  (1) at (3, -2) {1};
         \node  (2) at (0, 2.5) {2};		
		\draw (0) to (1);
		\draw (0) to (2);
		\draw (1) to (2);
\end{tikzpicture}
&
\begin{tikzpicture}[scale=0.70]
		\node  (0) at (-3, -2) {\small{\{0\}}};
		\node  (1) at (3, -2) {\small{\{1\}}};
         \node  (2) at (0, 2.5) {\small{\{2\}}};
         
		\node  (01) at (0, -2) {\tiny{\{0,1\}}};
		\node  (02) at (-1.5, 0.25) {\tiny{\{0,2\}}};
         \node  (12) at (1.5, 0.25) {\tiny{\{1,2\}}};
         
         \node  (012) at (0, -0.7) {\tiny{\{0,1,2\}}};
	    \draw (0) to (01);
		\draw (1) to (01);
	    \draw (0) to (02);
		\draw (2) to (02);
	    \draw (1) to (12);
		\draw (2) to (12);
         \draw (0) to (012);
		\draw (1) to (012);
	    \draw (2) to (012);
		\draw (02) to (012);
	    \draw (01) to (012);
		\draw (12) to (012);
\end{tikzpicture}
\end{tabular}
\caption{The $2$-dimensional simplex $\Delta$ and its first barycentric subdivision $\beta \Delta$. A simplicial map $s\colon \beta \Delta\to \Delta$ is 
an elementary selection if $s(\sigma)\in \sigma$ for every non-empty  $\sigma\subseteq \{0,1,2\}$.
} \label{fig:M1}
\end{figure}
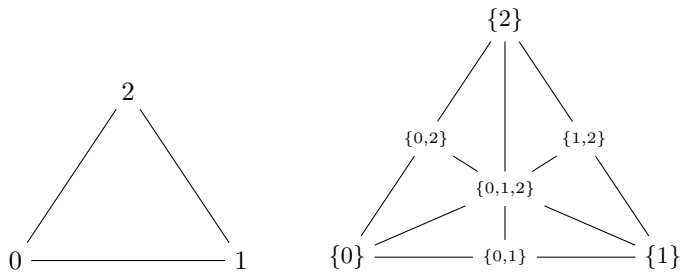

The notions in the theorem below 
are explained in Section~\ref{S:Background}. 

\begin{theorem}\label{Theorem:Intro1}
Let $A$ be a finite simplicial complex. 
\begin{enumerate} 
\item[(i)] $\mathcal{S}(A)$ is a projective \Fraisse{} category. 

\item[(ii)] The topological realization $\mathbb{A}/R^{\mathbb{A}}$ of 
the projective \Fraisse{} limit $\mathbb{A}$ of $\mathcal{S}(A)$ is homeomorphic to the geometric realization $|A|_{\mathbb{R}}$ of $A$. 
\end{enumerate}
\end{theorem}
 
As mentioned earlier, Theorem~\ref{Theorem:Intro1}(ii) gives an intrinsic combinatorial definition of the geometric realization of $A$. 
Point (i) of the theorem above is proved in Section~\ref{S:1}. The proof of point (ii) is quite involved; its final 
step is presented in Section~\ref{Su:proofend} and uses the material developed in earlier sections. 

In particular, this proof depends on 
the notion of domination closure that appears interesting in its own right, and that we motivate here.  In a nutshell, while
the inductive definition of the category $\mathcal{S}(A)$ makes $\mathcal{S}(A)$ amenable to combinatorial investigations, it also makes it quite rigid for the purpose of analyzing the face-preserving symmetries of $\mathbb A$.  Our goal is to extract from $\mathcal{S}(A)$ broader, more flexible projective \Fraisse{} categories, whose projective \Fraisse{} limits are still equal to $\mathbb A$. 
 The guiding principle will be provided by the abstract notion 
of domination closure for essentially countable categories that we introduce in Section \ref{S:CoinitialClosure}. Fix an ambient category $\mathcal E$ and let $\mathcal{F}\subseteq \mathcal{F}'$ be subcategories of $\mathcal{E}$ on the same class of objects as $\mathcal E$.  Recall that $\mathcal{F}$  {\bf is dominating in  $\mathcal{F}'$},
if for every $g'\in \mathcal{F}'$ there is 
$g''\in \mathcal{F}'$ so that $g'\circ g'' \in \mathcal{F}$. In Section~\ref{Su:coo}, we prove the following theorem. 

\begin{theorem}\label{T:IntroCoin}
Let $\mathcal{F},\, \mathcal{E}$ be as above. There exists a largest subcategory $[{\mathcal F}]$ of ${\mathcal E}$  which is dominated by $\mathcal F$. 
Moreover, $[[{\mathcal F}]]=[{\mathcal F}]$.
\end{theorem}

We call $[\mathcal{F}]$ the {\bf domination closure of $\mathcal{F}$ with respect to $\mathcal{E}$}. In Section~\ref{SS:CoinitionalClosureFraisse}, 
we develop a theory of domination closure in 
the case when $\mathcal{F}$ is a projective \Fraisse{} category and $\mathcal E$ is essentially countable.
In this case, $[\mathcal{F}]$ turns out to be a projective \Fraisse{} category that is potentially (and actually, for the category of selections) larger than 
$\mathcal F$, but has the same \Fraisse{} limit as $\mathcal{F}$.

It is convenient to consider the special instance of the category ${\mathcal S}(A)$ when $A$ is a simplex. 
The {\bf $n$-dimensional  simplex} $\Delta$ is the set of all non-empty subsets of $\{0,1,\ldots,n\}$.
Using Theorem~\ref{Theorem:Intro1}, we introduce now the main object of this paper.  

\begin{definition}
The {\bf generic combinatorial simplex} is the projective \Fraisse{} limit $\DD$ of the projective \Fraisse{} category $\mathcal{S}(\Delta)$. 
\end{definition}

In Section~\ref{S:CategoriesOfMaps}, we compute the domination closure $[\mathcal{S}(\Delta)]$ of $\mathcal{S}(\Delta)$ in the ambient category $\mathcal{R}(\Delta)$ 
of all simplicial maps among barycentric subdivisions of $\Delta$ that preserve the face structure of $\Delta$. 
In particular, we consider the categories $\mathcal{H}(\Delta)$, of all {\bf restricted near-homeomorphisms} and $\mathcal{C}(\Delta)$, of all {\bf hereditarily cellular maps} on barycentric subdivisions of $\Delta$. 
We give definitions of these classes in Section~\ref{Su:defrel}, but we point out here that near-homeomorphisms and cellular maps are well-studied 
in topology classes of maps. 
The following theorem is proved, in installments, in Sections~\ref{S:ResultsOnHereditarilyCellularMaps}, 
\ref{S:ResultsOnRestrictedNearHomeo}, and \ref{Su:close}. Its main point is that the domination closure of ${\mathcal S}(\Delta)$ in ${\mathcal R}(\Delta)$ 
is estimated from below by ${\mathcal C}(\Delta)$ and from above by ${\mathcal H}(\Delta)$, and that ${\mathcal C}(\Delta)$ and ${\mathcal H}(\Delta)$ 
are equal under appropriate assumptions. 

\begin{theorem}\label{Theorem:Intron}
Let $\Delta$ be the $n$-simplex and let 
$[\mathcal{S}(\Delta)]$ be the domination closure of 
$\mathcal{S}(\Delta)$  in  
$\mathcal{R}(\Delta)$. Then 
\[
\mathcal{S}(\Delta)\subseteq \mathcal{C}(\Delta) \subseteq [\mathcal{S}(\Delta)]\subseteq \mathcal{H}(\Delta).
\]
Furthermore, 
\[
\mathcal{C}(\Delta)=\mathcal{H}(\Delta),\,\hbox{ for }n<4,
\]
and, if the PL-Poincar{\'e} conjecture is positively resolved for $n=4$, then $\mathcal{C}(\Delta)=\mathcal{H}(\Delta)$ for all $n\geq 0$ . 
\end{theorem}

To establish Theorem~\ref{Theorem:Intron}, we will need to consider a more general framework and prove 
a more general result of which Theorem~\ref{Theorem:Intron} 
is a consequence. In particular, crucial to our proof will be combinatorial triangulations of $\Delta$ which are more flexible than barycentric 
subdivisions of $\Delta$.  
These are stellar $n$-simplexes, obtained by using stellar moves, first defined and studied by Alexander~\cite{Al} and Newman~\cite{Ne}; see also \cite{Li}.  
A {\bf stellar $n$-simplex} is 
a stellar $n$-ball together with a fixed decomposition of its boundary into stellar balls of lower dimensions, each of them corresponding to 
a combinatorial triangulation of a face of the $n$-dimensional simplex $\Delta$. 

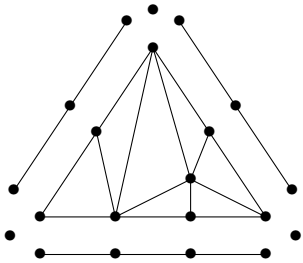
\begin{figure}[h]
\centering
\begin{tikzpicture}[scale=0.50]
\node  (d0) at (-3.8, -2.5){$\bullet$};
\node  (d1) at (3.8, -2.5){$\bullet$};
\node  (d2) at (0, 3.5){$\bullet$};         
\node  (a0) at (-3, -3){$\bullet$};
\node  (a1) at (3, -3){$\bullet$};
\node  (a01) at (-1, -3){$\bullet$};
\node  (a01') at (1, -3){$\bullet$};
\draw (a0.center) to (a01.center);
\draw (a1.center) to (a01'.center);
\draw (a01.center) to (a01'.center);
\node  (b1) at (3.7, -1.3){$\bullet$};
\node  (b2) at (0.7, 3.2){$\bullet$};         
\node  (b12) at (2.2, 0.95){$\bullet$};
\draw (b1.center) to (b12.center);
\draw (b2.center) to (b12.center);
\node  (c0) at (-3.7, -1.3){$\bullet$};
\node  (c2) at (-0.7, 3.2){$\bullet$};         
\node  (c02) at (-2.2, 0.95){$\bullet$};
\draw (c0.center) to (c02.center);
\draw (c2.center) to (c02.center);
\node  (0) at (-3, -2){$\bullet$};
\node  (1) at (3, -2){$\bullet$};
\node  (2) at (0, 2.5){$\bullet$};         
\node  (01) at (-1, -2){$\bullet$};
\node  (01') at (1, -2){$\bullet$};	
\node  (02) at (-1.5, 0.25){$\bullet$};
\node  (12) at (1.5, 0.25){$\bullet$};
\node  (012) at (1, -1){$\bullet$};
\draw (0.center) to (01.center);
\draw (1.center) to (01'.center);
\draw (01.center) to (01'.center);
\draw (0.center) to (02.center);
\draw (2.center) to (02.center);
\draw (01.center) to (02.center);
\draw (01.center) to (2.center);	    
\draw (01'.center) to (012.center);
\draw (01.center) to (012.center);
\draw (12.center) to (012.center);
\draw (2.center) to (012.center);	    
\draw (1.center) to (012.center);	    
\draw (1.center) to (12.center);
\draw (2.center) to (12.center);         
\end{tikzpicture}
\caption{A finite stellar $2$-simplex  $A=(A_{\mathrm{X}}\mid\mathrm{X}\in\Delta)$. 
} \label{fig:M2}
\end{figure}

Considering stellar $n$-simplexes naturally leads to defining a category broader than ${\mathcal S}(\Delta)$. 
In relation to $\mathcal{S}(\Delta)$, the category $\mathcal{S}_{\star}(\Delta)$ 
of {\bf selection maps on stellar $n$-simplexes} is defined as the union of all categories of the form $\mathcal{S}(A)$, where $A$ is a stellar $n$-simplex. 
Notice that  $\mathcal{S}(\Delta)$ can be canonically identified as a full subcategory of $\mathcal{S}_{\star}(\Delta)$.
We compute the domination closure $[\mathcal{S}_{\star}(\Delta)]$ of $\mathcal{S}_{\star}(\Delta)$ in the ambient category $\mathcal{R}_{\star}(\Delta)$ 
of all simplicial maps between stellar $n$-simplexes which preserve the face structure of $\Delta$. 
We obtain a theorem fully analogous to Theorem~\ref{Theorem:Intron}, where the categories 
$\mathcal{H}(\Delta)$ and $\mathcal{C}(\Delta)$ are replaced by 
the categories $\mathcal{H}_{\star}(\Delta)$ of all {\bf restricted near-homeomorphisms on stellar $n$-simplexes} 
and $\mathcal{C}_{\star}(\Delta)$ of all {\bf hereditarily cellular maps on stellar $n$-simplexes}.

The proof of Theorem \ref{Theorem:Intron}  relies on some new results in the theory of stellar moves which are presented in Section~\ref{S:strcomp}. 
One of these results is a variant of the main technical theorem of Alexander~\cite{Al} and Newman~\cite{Ne} which states that every stellar ball $B$ can be transformed to 
a cone $[v]\star \partial B$ using a sequence of internal stellar moves. In Section~\ref{S: Tame balls}, we define the notion of a {\bf strongly internal stellar move}, 
and we show that the moves transforming $B$ to $[v]\star \partial B$ can  be assumed to be strongly internal, as long as the boundary of $B$ is 
an induced subcomplex of $B$.

\subsection{Organization of the paper}
The main line of our argument runs from Section~\ref{S:1} through Section~\ref{S:CoinitialClosure} to 
Section~\ref{S:CategoriesOfMaps}, with Section~\ref{S:strcomp} providing the necessary tools for proving the results of 
Section~\ref{S:CategoriesOfMaps}.
In Section~\ref{S:1}, we introduce the projective \Fraisse{} category of certain simplicial maps, which we call selections, and represent simplexes 
as canonical quotients of limits of these categories.  In Section~\ref{S:CoinitialClosure}, we show that each \Fraisse{} category is included in a (canonical) largest \Fraisse{} category, within a fixed ambient category, 
that produces the same limit. We call this largest category the domination closure of the category we started with. In Section~\ref{S:CategoriesOfMaps}, 
we compute lower and upper estimates on the domination closure of the category of selections, and we show that under favorable circumstances 
the two estimates coincide.  The arguments in 
Section~\ref{S:CategoriesOfMaps} use the theory of stellar manifolds to establish the lower estimate. 
In Section~\ref{S:strcomp}, we prove new results in the theory of stellar manifolds. These results are needed in our considerations in 
Section~\ref{S:CategoriesOfMaps}, but we also find them interesting in their own right. 
In Section~\ref{S:Background}, we give the background for projective \Fraisse{} theory that is necessary to frame the discussion in this paper.

\subsection{A review of standard notions concerning simplicial complexes and simplicial maps} \label{SS:Complexes}
A {\bf simplicial complex} or simply a {\bf complex} $C$ is a family of finite non-empty sets that is closed under taking non-empty subsets, that is, 
if $\sigma\in C$ and $\tau\subseteq\sigma$ with $\tau\neq \emptyset$, then $\tau\in C$.
Notice that the empty set $\emptyset$ constitutes a complex, which we call the {\bf empty complex}. 
An element $\sigma$ of $C$ is called  a {\bf face} of $C$. The {\bf domain} of $C$ is the union $\bigcup C$ of all faces in $C$, 
and we denote it by  $\mathrm{dom}(C)$. 
A {\bf vertex} $v$ of $C$ is any element of $\mathrm{dom}(C)$.  A simplicial subcomplex or simply a {\bf subcomplex} $D$ of $C$ is a simplicial complex with $D\subseteq C$.

Let $C,D$ be two complexes. We view a function
\[f\colon \mathrm{dom}(C)\to \mathrm{dom}(D),\]
as a function which acts on the level of vertexes, faces and subcomplexes. 
If $\sigma$ is a face of $C$, $A$ is a subcomplex of $C$,  and $B$ is a subcomplex of $D$, then we let 
\[
f_{*}(\sigma) = \{f(v) \mid v\in\sigma\},\; \; f_{*}(A) = \{f_*(\tau) \mid \tau\in A\},
\]
\[\text{ and } \; f^{-1}_{*}(B) = \{\tau \in C\mid f_{*}(\tau)\in B\}.\]
If for all $\sigma\in C$ we have that $f_{*}(\sigma)\in D$, then we say that $f$ is {\bf simplicial} and we use the notation
\[
f\colon C\to D.
\]

A simplicial map $f\colon C\to D$ is an {\bf epimorphism}, if it surjective as a map from $C$ to $D$. It is an {\bf embedding}, if it is injective as a map from $C$ to $D$,. Finally, we say that $f$ is an {\bf isomorphism}, if it is both an embedding and an epimorphism.

A {\bf simplex} $\Delta$ is any finite complex with the property that every non-empty subset of $\mathrm{dom}(\Delta)$ is a face of $\Delta$. Every finite set $X$ generates a unique simplex $\mathcal{P}(X)\setminus \{ \emptyset\}$, where $\mathcal{P}(X)$ denotes the power set of $X$. We denote this simplex by $[X]$.  We will often use the handy $[v_1,\ldots,v_n]$  in place of $[\{v_1,\ldots,v_n\}]$. The {\bf dimension}  $\mathrm{dim}(\Delta)$ of a simplex $\Delta$  is by definition $(n-1)$, where $n$ is the size of $\mathrm{dom}(\Delta)$. Notice that the empty complex $\emptyset$ is a simplex with  $\mathrm{dim}(\emptyset)=(-1)$.  Up to isomorphism there is a unique simplex of dimension $n$. We set
\[\Delta^n=\mathcal{P}(\{0,1,\ldots,n\})\setminus \{ \emptyset\},\]
to be the ``canonical representative'' of its isomorphism class and we call it {\bf the $n$-dimensional simplex}. We will work here only with finite dimensional simplicial complexes. These are precisely the complexes $C$ for which there is a largest $n\in\mathbb{N}$ so that $\Delta^n$ embeds in $C$. We call this the dimension of $C$ and  denote it by $\mathrm{dim}(C)$.

If $A$ is a finite simplcial complex, we can identify ${\rm dom}(A)$ with the set  $\{ 0, \dots , n \}$ for some $n\in {\mathbb N}$. The {\bf geometric realization $|A|_{\mathbb{R}}$ 
 of} $A$ is  the set
\begin{equation}\label{E:geomr}
\big\{ (x_0, \dots, x_n)\in {\mathbb R}^{n+1}\colon x_i\geq 0 \text{ and } \sum_i x_i=  \sum_{i\in\sigma} x_i =  1\hbox{ for some } \sigma\in A\big\},
\end{equation}
If $B$ is a subcomplex of $A$, then the associated geometric 
realization of $B$ is 
\begin{equation}\label{E:geos}
\big\{ (x_0, \dots, x_n)\in {\mathbb R}^{n+1}\colon x_i\geq 0 \text{ and } \sum_i x_i=  \sum_{i\in\sigma} x_i =  1\hbox{ for some } \sigma\in B\big\}. 
\end{equation} 
Similarly, for a face $\sigma$ of $A$, its associated geometric realization is 
\begin{equation}\label{E:geot} 
\big\{ (x_0, \dots, x_n)\in {\mathbb R}^{n+1}\colon x_i\geq 0 \text{ and } \sum_i x_i=  \sum_{i\in\sigma} x_i =  1\big\}. 
\end{equation}
We will denote the sets in \eqref{E:geos} and \eqref{E:geot} by $|B|_{|\mathbb R}$ and 
$|\sigma|_{\mathbb R}$, respectively, suppressing their dependence on $A$, when it is clear from the context.

\section{Background in projective Fra{\"i}ss{\'e} theory}\label{S:Background}

While classical \Fraisse{} theory can be used, in theory, to approximate the dynamics of 
homeomorphism groups $\mathrm{Homeo}(K)$ of metrizable compact spaces $K$, the resulting \Fraisse{} classes are rarely natural for combinatorial investigations.  
The systematic study of homeomorphism groups via \Fraisse{} theoretic techniques came with the introduction of projective \Fraisse{} theory in \cite{IrSo}.
In short, the idea is to replace classical \Fraisse{} categories of embeddings between finite structures with categories $\mathcal C$ of epimorphisms 
between structures equipped with a binary reflexive and symmetric relation $R$, which often is the edge relation of a finite simplicial complex. 
The category is assumed to satisfy the  projective analogue of the \Fraisse{} axioms. In applications, this projective \Fraisse{} category $\mathcal{C}$ has to be  
chosen appropriately so that the profinite complex $\mathbb{K}$, which is the inverse limit of a generic inverse sequence in $\mathcal{C}$, captures the 
dynamics of the original space $K$ in the following sense: the edge relation $R^{\mathbb{K}}$ on $\mathbb{K}$ is 
an equivalence relation; the quotient $\mathbb{K}/R^{\mathbb{K}}$ is homeomorphic to $K$; and the quotient map ${\mathbb K}\to K$ induces 
a continuous homomorphism from 
$\mathrm{Aut}(\mathbb{K})$ to $\mathrm{Homeo}(K)$ whose image is dense. Projective \Fraisse{} theory has been used 
in the study of a number of homeomorphism groups; see \cite{BaKw,IrSo,Kw1,Kw2,PaSo}.

\subsection{Abstract Fra{\"i}ss{\'e} categories and generic sequences}\label{SS:AbstractFraisse}
Let  $\mathcal{E}$  be a category and let $f\in \mathcal{E}$. We denote by $\mathrm{Ob}(\mathcal{E})$ the collection of all objects of $\mathcal{E}$ and by 
$\mathrm{dom}(f)$, $\mathrm{codom}(f)$ the domain  and codomain of $f$.  A category $\mathcal{E}$ is {\bf essentially countable} if $\mathrm{Ob}(\mathcal{E})$ is countable up to isomorphism and for for every $o_1,o_2\in \mathrm{Ob}(\mathcal{E})$ there countably many $f\in \mathcal{E}$ with $\mathrm{dom}(f)=o_1$ and  $\mathrm{codom}(f)=o_2$.
A {\bf projective Fra{\"i}ss{\'e} category}, or simply a {\bf Fra{\"i}ss{\'e} category}, is an essentially countable category 
$\mathcal{E}$ which  satisfies the following two properties:
\begin{enumerate}
\item[(i)] ({\bf joint projection}) for any two $o_1, o_2\in \mathrm{Ob}(\mathcal{E})$ there are $e_1, e_2\in \mathcal{E}$ with  $\mathrm{codom}(e_1)=o_1$, 
$\mathrm{codom}(e_2)=o_2$, and  $\mathrm{dom}(e_1)=\mathrm{dom}(e_2)$;

\item[(ii)] ({\bf projective amalgamation}) for any two $e_1, e_2\in \mathcal{E}$ with $\mathrm{codom}(e_1)=\mathrm{codom}(e_2)$  there are $e'_1, e'_2\in \mathcal{E}$ with $e_1\circ e'_1= e_2\circ e'_2$.
\end{enumerate}

A sequence  $(e_n)$ of arrows in a category $\mathcal{E}$ is {\bf neat} if for every $n$ we have that $\mathrm{dom}(e_n)=\mathrm{codom}(e_{n+1})$. Let $(e_n)$ be a neat sequence. A sequence $(e'_m)$ is a {\bf block subsequence of} $(e_n)$ if there are $i_0<n_1<n_2<\dots$ such that,  
for each $m$, $e'_m =  e_{n_m}\circ\cdots \circ  e_{n_{m+1}-1}$. Note that $(e_m')$ is automatically neat.  Let $(e_n)$ and $(f_n)$ be neat sequences in $\mathcal{E}$ and let $\mathcal{F}$ be a subcategory of $\mathcal{E}$.  A neat sequence $(g_m)$ in $\mathcal{F}$ is called an {\bf  $\mathcal{F}$-isomorphism from $(e_n)$ to $(f_n)$} if there are block subsequences $(e'_m)$ and $(f'_m)$ of $(e_n)$ and $(f_n)$, respectively, such that, for each $m$,  
\[
e_m' = g_{2m} \circ g_{2m+1} \;\hbox{ and }\; f_m' = g_{2m+1} \circ g_{2m+2}
\]
In the situation above, we say that $(e_n)$ and $(f_n)$ are {\bf $\mathcal{F}$-isomorphic}.  Note that the sequence $(g'_m)$, where 
$g_m'=g_{m+1}$, is an $\mathcal{F}$-isomorphism from $(f_m)$ to $(e_m)$.  An $\mathcal{E}$-isomorphism will be called simple an {\bf  isomorphism} 
and in this situation we say that $(e_n)$ and $(f_n)$ are {\bf  isomorphic}.  
A neat sequence $(e_n)$ in $\mathcal{E}$ is a {\bf generic sequence for $\mathcal{E}$} if it satisfies the following two properties:
\begin{enumerate}
\item[(i)] ({\bf projective universality})  for every $o\in \mathrm{Ob}(\mathcal{E})$ there is $n\in\mathbb{N}$ and $e\in\mathcal{E}$ so that $\mathrm{dom}(e_n)=\mathrm{dom}(e)$ and $\mathrm{codom}(e)=o$;
\item[(ii)] ({\bf projective extension}) for every $e\in \mathcal{E}$ with $\mathrm{codom}(e_n)=\mathrm{codom}(e)$ there exists $m>n$ and $e'\in \mathcal{E}$ so that $ e_n\circ\cdots \circ e_m=e\circ e'.$   
\end{enumerate}

Note that a block subsequence of a generic sequence is also generic. The following  theorem relates the notion of a  Fra{\"i}ss{\'e} category with that of a generic sequence.

\begin{theorem}[Fra{\"i}ss{\'e}] \label{T:FraisseOriginal}
If $\mathcal{F}$ is a  Fra{\"i}ss{\'e} category, then  there exists a generic sequence for $\mathcal{F}$. 
Moreover, all generic sequences for $\mathcal{F}$  are $\mathcal{F}$-isomorphic. 
\end{theorem}

One can simultaneously justify the terminology above and strengthen the first conclusion of Theorem \ref{T:FraisseOriginal} as follows. 
Let $\mathcal{F}$ be a \Fraisse{} category, and assume without a loss of generality that $\mathcal{F}$  is countable. The set of all neat sequences 
$\mathrm{Neat}(\mathcal{F})$ of $\mathcal{F}$ can be identified with a closed subspace of the space $\mathbb{N}^{\mathbb{N}}$ of all sequences of 
natural numbers. One can now show that generic sequences for $\mathcal F$ form a comeager subset of $\mathrm{Neat}(\mathcal{F})$, that is, 
a generic sequence for $\mathcal{F}$ in the above sense is also generic in the sense of Baire category.

\subsection{Concrete \Fraisse{} categories} \label{SS:Concrete}
In applications, one usually works with Fra{\"i}ss{\'e}  categories whose objects are structured sets and the morphisms are maps between these sets, preserving the additional structure. Here, 
by a {\bf concrete  Fra{\"i}ss{\'e}  category} we mean a \Fraisse{} category $\mathcal{F}$ consisting 
of epimorphisms $f\colon B\to A$ between finite simplicial complexes.  In this case, we always view the domain of each complex in $\mathrm{Ob}(\mathcal{F})$ as 
a topological space endowed with the discrete topology.  Often one may endow the underlying set $\mathrm{dom}(A)$ of each complex $A\in \mathrm{Ob}(\mathcal{F})$ with additional model theoretic structure so that the morphisms $f\colon B\to A$ in $\mathcal{F}$ preserve this structure in the sense of \cite{IrSo}. The interested reader may consult \cite{IrSo} for more details since the only instance of a model theoretic structure we are going to use here is implicit in the simplicial complex structure; see Section \ref{SS:Prespace}.

Let $\mathcal{F}$ be a concrete Fra{\"i}ss{\'e} category  and let  $(f_n)$ be a neat sequence of morphisms in $\mathcal{F}$ with $f_n\colon A_{n+1}\to A_n$. 
We can associate to  $(f_n)$ a {\bf profinite simplicial complex} $\mathbb{A}$:  a simplicial complex $\mathbb{A}$ whose domain $\mathrm{dom}(\mathbb{A})$ is a $0$-dimensional compact metrizable  space so that for each $m>0$, the set 
\[
\{(a_0,\ldots,a_{m-1})\in \mathrm{dom}(\mathbb{A})^m\mid  \{a_0,\ldots,a_{m-1}\} \in \mathbb{A}\},
\]
is  a closed subset of $\mathrm{dom}(\mathbb{A})^m$. This profinite complex is constructed as follows. Let $\mathrm{dom}(\mathbb{A})$ 
be the topological space that is the inverse limit of  $(\mathrm{dom}(A_k), f^l_k)$, where $f^l_k=  f_k \circ \cdots\circ f_{l-1}$, let $f^{\infty}_k\colon \mathrm{dom}(\mathbb{A})\to \mathrm{dom}(A_k)$ be the natural projection, and set 
\[
\{a_0,\ldots,a_{m-1}\} \in \mathbb{A}\;\hbox{ if and only if }\;\{f^{\infty}_k(a_0),\ldots,f^{\infty}_k(a_{m-1})\} \in A_k, \hbox{ for each }k.
\]
Notice that the maps $f^{\infty}_{k}$ are continuous simplicial maps from $\mathbb{A}$ to $A_k$, for all $k\geq 0$. 
If the sequence $(f_n)$ is generic for $\mathcal{F}$, we call $\mathbb{A}$ the {\bf projective \Fraisse{} limit of $\mathcal{F}$ induced by $(f_n)$} or simply 
the {\bf projective \Fraisse{} limit of $\mathcal{F}$}, since as we will see in the next paragraph, the structure $\mathbb{A}$ does not depend on $(f_n)$ up to isomorphism. We call a simplicial map $f\colon \mathbb{A}\to A$   {\bf approximable by $\mathcal{F}$} if there exists $k$ and a morphism $f'\colon A_k\to A$ in 
$\mathcal{F}$ such that $f= f'\circ f^{\infty}_k$. Notice that such a map is always continuous.

Let $(f_n)$ and $(e_n)$  be two generic sequences for $\mathcal{F}$, and let $\mathbb{A}$ and $\mathbb{B}$ be the induced projective \Fraisse{} limits.  
By an {\bf isomorphism} from $\mathbb{B}$ to $\mathbb{A}$ we mean a simplicial map $\phi \colon \mathbb{B}\to \mathbb{A}$ that is an isomorphism of simplicial complexes and continuous as a map from $\mathrm{dom}(\mathbb{B})$ to $\mathrm{dom}(\mathbb{A})$. 
Given an $\mathcal F$-isomorphism $(g_m)$ from $(e_n)$ to $(f_n)$ one may define $\phi\colon \mathbb{B}\to \mathbb{A}$ as the inverse limit of $(g_{2m+1})_m$ and $\phi^{-1}$ as the inverse limit of $(g_{2m})_m$. In this case, we say that $\phi$ is an {\bf $\mathcal{F}$-isomorphism}.
Every $\mathcal{F}$-isomorphism is approximable by $\mathcal{F}$ in the following sense: 
for all simplicial maps $f\colon {\mathbb A}\to A$ and $g\colon {\mathbb B}\to B$ which are approximable by $\mathcal{F}$, $f\circ \phi$ and 
$g\circ \phi^{-1}$ are also approximable by $\mathcal{F}$.  
By Theorem \ref{T:FraisseOriginal} the \Fraisse{} limits  $\mathbb{A}$ and $\mathbb{B}$ are always $\mathcal{F}$-isomorphic. Let 
\[
\mathrm{Aut}_{\mathcal{F}}(\mathbb{A})\, \text{  and  }\, \mathrm{Aut}(\mathbb{A})
\]
be the  groups of all $\mathcal{F}$-isomorphisms and all  isomorphisms from $\mathbb{A}$ to $\mathbb{A}$, respectively. We clearly have that  $\mathrm{Aut}_{\mathcal{F}}(\mathbb{A})\subseteq \mathrm{Aut}(\mathbb{A})$. 
The following alternative characterization of the projective \Fraisse{} limit of a concrete \Fraisse{} category  provides some context regarding the connection between \Fraisse{} theory and dynamics. This result can be easily proved by a ``back and forth'' argument using repeated application of property (ii) in the definition of a generic sequence; see~\cite{IrSo}.

\begin{theorem}\label{T:ultrahomogeneity}
Let $\mathcal{F}$ be a concrete projective \Fraisse{} category and let $\mathbb{A}$ be a profinite simplicial complex. Then $\mathbb{A}$ is the  \Fraisse{} limit of $\mathcal{F}$ if and only if 
\begin{enumerate}
\item[(i)] ({\bf projective universality}) for every $A\in\mathrm{Ob}(\mathcal{F})$ there is a simplicial map $f\colon \mathbb{A}\to A$ that is approximable by $\mathcal{F}$;
\item[(ii)] ({\bf projective ultrahomogeneity}) if $A\in\mathrm{Ob}(\mathcal{F})$ and $f,g\colon \mathbb{A}\to A$ are two simplicial maps which are approximable by $\mathcal{F}$, then there is $\varphi\in \mathrm{Aut}_{\mathcal{F}}(\mathbb{A})$ so that $f\circ \varphi=g$.
\end{enumerate}
\end{theorem}

\subsection{Combinatorial representations of compact metrizable spaces} \label{SS:Prespace}

Let $\mathcal{F}$ be a concrete \Fraisse{} category. We endow the domain of each complex $A$ in $\mathrm{Ob}(\mathcal{F})$ with a binary relation $R^A$ which keeps track of the ($\leq 1$)-skeleton of $A$. That is, $a R^A a'$ if and only if $\{a,a'\}\in A$. Each map $f\colon B\to A$ in $\mathcal{F}$ preserves the relation $R$ in the sense of \cite{IrSo}: $aR^A a'$ if and only if there are $b\in f^{-1}(a), b'\in f^{-1}(a')$  so that $b R^B b'$. Moreover, if $\mathbb{A}$ is the projective \Fraisse{} limit of $\mathcal{F}$ induced by $(f_n)$, then the inverse limit $R^{\mathbb{A}}$ of the relation $R$ under $(f_n)$ coincides with the ($\leq 1$)-skeleton of $\mathbb{A}$. As a consequence, $R^{\mathbb A}$ is a symmetric and reflexive compact relation on $\mathrm{dom}(\mathbb{A})^2$. If it happens that  $R^{\mathbb A}$ is also transitive, then call $\mathbb{A}$  a {\bf pre-space} and we set
\begin{equation}\label{E:toprep}
|\mathbb{A}| = \mathrm{dom}(\mathbb{A})/ R^{\mathbb{A}},
\end{equation}
to be the {\bf topological realization} of $\mathbb{A}$ and let $\pi\colon \mathbb{A}\to |\mathbb{A}|$ be the map which sends each vertex $a$ of $\mathbb{A}$ 
to its $R^{\mathbb{A}}$-equivalence class $[a]$. Since every element of  $\mathrm{Aut}(\mathbb{A})$ preserves $R^{\mathrm{A}}$, we have that $\pi$ induces a 
homomorphism $\mathrm{Aut}(\mathbb{A})\to \mathrm{Homeo}(|\mathbb{A}|)$ which turns out to be continuous \cite{IrSo}.
We view ${\mathbb A}$ as a combinatorial representation of the topological space $|\mathbb{A}|$ and  the category $\mathcal{F}$ as the combinatorial representation of the subgroup of $\mathrm{Homeo}(|\mathbb{A}|)$ that is the closure of the image of $\mathrm{Aut}_{\mathcal{F}}(\mathbb{A})$ under the above embedding.

\section{The projective \Fraisse{} category of selection maps}\label{S:1}

In this section, we define for every finite simplicial complex $A$  the category $\mathcal{S}(A)$ of  selections on $A$.  
We then prove Theorem \ref{T:A} and Theorem \ref{T:topologicalRealization} which, if taken in conjunction, refine the statement of  
Theorem \ref{Theorem:Intro1} from the introduction. 
Finally we provide some alternative characterization of the  morphism in $\mathcal{S}(A)$.

Let $C$ be a simplicial complex. As usual, by $\beta C$ we denote the {\bf barycentric subdivision} of $C$. This is the simplicial complex that is defined as follows:
\begin{itemize}
\item[---] $\mathrm{dom}(\beta C)=C$, that is, vertices of $\beta C$ are all faces of $C$;
\item[---] the faces of $\beta C$ are all chains with respect to inclusion of faces of $C$, that is, all subsets $\{\sigma_0,\ldots,\sigma_{k-1}\}$ of $C$ with 
$\sigma_0\subseteq\cdots\subseteq \sigma_{k-1}$.
\end{itemize}
Iterating this process, we define $\beta^{k}C$  inductively  for  every $k\in\mathbb{N}$. We set   $\beta^{0}C=C$, and $\beta^{k+1}C=\beta(\beta^{k}C)$. See Figures \ref{fig:M1} and \ref{fig:MNew}. 

\begin{figure}[h]
\centering
\begin{tabular}{c c c c}
\includegraphics[width=20mm, angle=90,origin=c]{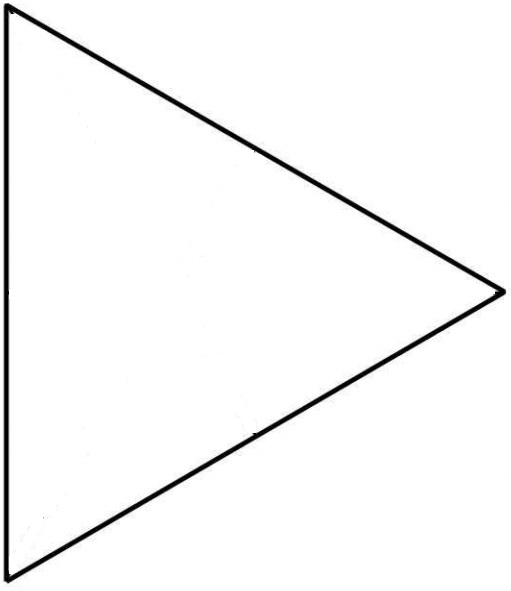} & \includegraphics[width=20mm, angle=90,origin=c]{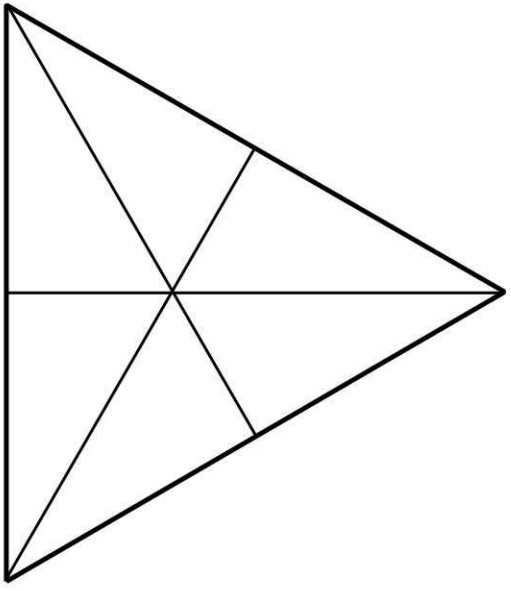} & \includegraphics[width=20mm, angle=90,origin=c]{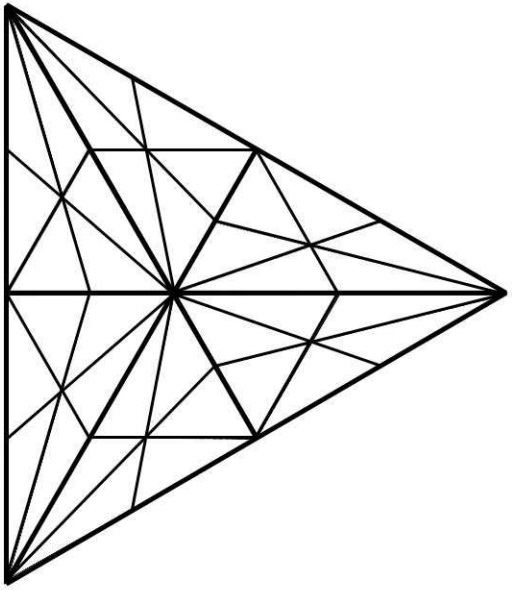} & \includegraphics[width=20mm, angle=90,origin=c]{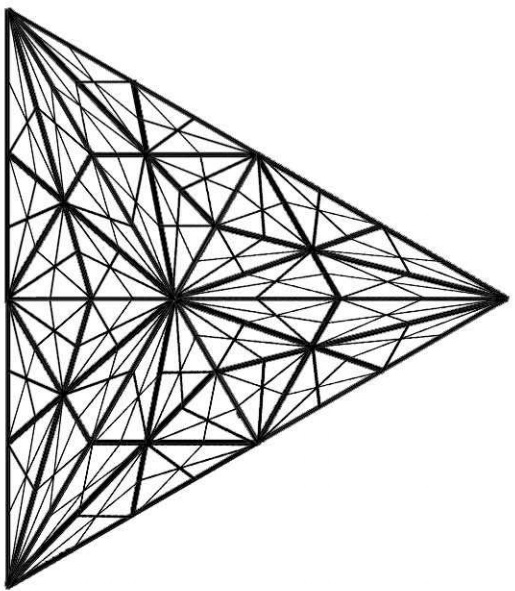}\\
\end{tabular}
\caption{The complex $\beta^k\Delta$, for $k=0,1,2,3$.} \label{fig:MNew}
\end{figure}

Let  $f:C\to D$ be  a simplicial map between simplicial complexes $C,D$. We define the {\bf barycentric subdivision of $f$} to be the map $\beta f\colon\beta C\to\beta D$, with  
\[
(\beta f)(\sigma)= f_{*}(\sigma) \; \hbox{ for all } \sigma\in\mathrm{dom}(\beta C).
\] 
Notice that since $f$ is simplicial, $f _*(\sigma)$ is a face of $D$ for every $\sigma\in C$, that is, for every element of ${\rm dom}(\beta C)$. 
Hence $\beta f$ is well defined. 
It is easy to check that $\beta f$ is simplicial and, in fact, an epimorphism whenever $f$ is an epimorphism. 

A function $s \colon C\to {\rm dom}(C)$ is called an {\bf elementary selection} if $s(\sigma)\in \sigma$ for each face $\sigma$ of $C$. 
It is easy to check that any such function induces a simplicial map from $\beta C$ to $C$ which is an epimorphism.  
We also use the name elementary selection for these simplicial maps. 

\begin{definition}\label{def:SelectionsCategory}
Let $A$ be any finite simplicial complex. We define  the category  $\mathcal{S}(A)$ of {\bf selections on $A$}  to be the smallest family of simplicial maps which
\begin{enumerate}
\item[---] contains all elementary selections $\beta^{k+1}A\to \beta^kA$, for all $k\geq 0$;
\item[---] contains the identity map $1_A:A\to A$;
\item[---] is closed under composition;
\item[---] is closed under the operation $f\to\beta f$ between simplicial maps.
\end{enumerate}
\end{definition}
Notice that $\mathcal{S}(A)$ consists entirely of epimorphisms and that a simplicial  complex $C$ is an object in $\mathcal{S}(A)$ if and only if $C$ is of the form $\beta^k A$ for some $k\geq 0$.

We can now prove the first part of the statement of Theorem \ref{Theorem:Intro1}. 

\begin{theorem}\label{T:A} 
If $A$ is a finite simplicial complex, then $\mathcal{S}(A)$ is a projective Fra\"{i}ss\'e{} category.
\end{theorem}

We will need the following standard 
lemma whose proof is straightforward, and will be left to the reader.

\begin{lemma}\label{L:betaf}
Let $g\colon C_1\to C_2$ and $f\colon C_2\to C_3$ be simplicial maps. Then
\[\beta(f\circ g)=\beta f\circ\beta g.\]
\end{lemma}

\begin{proof}[Proof of Theorem~\ref{T:A}]
The only point that needs an argument is the projective amalgamation property. 

Consider the following property of a map $h\in {\mathcal S}(A)$: 
for each $f\in {\mathcal S}(A)$ with ${\rm codom}(f)= {\rm codom}(h)$, there exist $g_1, g_2\in {\mathcal S}(A)$ such that $f\circ g_1 = h\circ g_2$. 
Using Lemma~\ref{L:betaf} one easily sees that each map in ${\mathcal S}(A)$ is of the form $\beta^k (1_A)$ or is a composition of maps 
of the form $\beta^k (s)$, where $s$ is an elementary selection and $k\geq 0$. It follows that to prove the projective amalgamation,
it suffices to show that $1_A$ and each elementary selection have the property above and that if $h$ has it, then so does $\beta h$. 

It is clear that $1_A$ has the property. The map $\beta h$ inherits it from $h$ by the following observation. 
If, for $f_1\colon \beta^{l_1}A \to \beta^kA$ and $f_2\colon \beta^{l_2}A \to \beta^kA$, there are 
$g_1, g_2\in \mathcal{S}(A)$, $g_1\colon \beta^mA\to \beta^{l_1}A$ and $g_2\colon \beta^mA\to \beta^{l_2}A$ with 
$f_1\circ g_1 = f_2\circ g_2$, then, by Lemma \ref{L:betaf}, we also have 
\[
\beta f_1 \circ\beta g_1 = \beta f_2\circ\beta g_2
\] 
and $\beta g_1, \beta g_2\in\mathcal{S}(A)$. 
It remains to prove the property for elementary selections. This is an immediate consequence of 
the following claim. 

\begin{claim}
Let $f\colon\beta^{l}A\to\beta^{k}A$ be a simplicial map, with $l\geq k$, and let $s:\beta^{k+1}A\to\beta^{k}A$ be an elementary selection. 
Then there is an elementary selection $s':\beta^{l+1}A\to\beta^{l}A$  so that 
\begin{equation}\label{E:ama}
f\circ s' = s\circ \beta f. 
\end{equation}
\end{claim}

\noindent{\em Proof of Claim.} 
We define $s'$ by chasing the amalgamation diagram. Elements of ${\rm dom}(\beta^{l+1} A)$ 
are faces of $\beta^l A$, that is, they are of the form $\{v_1,\dots, v_p\}$, 
where each $v_i$ is in ${\rm dom}(\beta^{l} A)$. 
So, for each $\{v_1, \dots, v_p\}\in {\rm dom}(\beta^{l+1}A)$, we need to find $i_0\leq p$ and set 
\[
s'(\{ v_1, \dots, v_p\}) = v_{i_0}
\]
so that \eqref{E:ama} holds. We have that 
\[
\beta f(\{v_1,\dots, v_p\})=\{f(v_1), \dots , f(v_p)\} 
\] 
with the latter being a face of $\beta^{k+1}A$. Since $s$ is an elementary selection, we have that 
\[
s(\{f(v_1), \dots,  f(v_p)\})\in \{f(v_1), \dots,  f(v_p)\}.
\]
So we can pick, not necessarily in a unique way, some $i_0\leq p$ such that 
\[
s(\{f(v_1), \dots, f(v_p)\})=f(v_{i_0}).
\]
Set $s'(\{ v_1, \dots, v_p\})= v_{i_0}$. It follows that $s'$ is an elementary selection and that (\ref{E:ama}) is satisfied.
\end{proof}

Since $\mathcal{S}(A)$ is a concrete projective \Fraisse{} category we can consider the inverse system $(A_k, f^l_k)$ that is associated to the unique, 
up to isomorhism,  generic sequence $(f_n)$ for $\mathcal{S}(A)$; see Section~\ref{SS:Concrete}. We denote by $\beta^\infty A$ 
the projective Fra{\"i}ss{\'e} limit $\varprojlim (A_k, f^l_k)$ of $\mathcal{S}(A)$ and let $f^{\infty}_k\colon \beta^\infty A \to A_k$ be the natural projection map. 
Let $R$ be the associated binary relation on $\beta^\infty A$; see Section~\ref{SS:Prespace}.

\begin{lemma}\label{L:EquivalenceRelation}
The projective \Fraisse{} limit  $\beta^\infty A$ of $\mathcal{S}(A)$ is a pre-space.
\end{lemma}
\begin{proof}
The only thing that needs to be shown is that $R$ is transitive. Let $x_0,x_1,x_2\in\mathrm{dom}(\beta^\infty A)$ with 
$\{x_0,x_1\}, \{x_1,x_2\}\in \beta^\infty A$. Let  $k>0$ and set $a_i:=f^{\infty}_k(x_i)$. It follows that both $\{a_0,a_1\} $ and 
$\{a_1,a_2\}$ are faces of  $A_k$. Consider the elementary selection map $s\colon \beta A_k\to A_k$ defined as follows. 
First, if $a_2\in \sigma\in A_k$, 
then  set $s(\sigma)=a_2$. If $a_1\in \sigma$ but $a_2\not\in \sigma$, then set $s(\sigma)=a_1$. 
For all remaining $\sigma\in A_k$ let $s(\sigma)$ be 
any element of $\mathrm{dom}(A_k)$. By property (ii) in the definition of a generic sequence there is $l>k$ an $g\colon A_l\to \beta A_k$ in $\mathcal{S}(A)$ 
so that $ s \circ g \circ f^{\infty}_l = f^{\infty}_k$. Let $\sigma_0,\sigma_1,\sigma_2\in A_k$ with $\sigma_i=g \circ f^{\infty}_l(x_i)$. But then both 
$\{\sigma_0,\sigma_1\}$ and $\{\sigma_1,\sigma_2\}$ are faces of $\beta A_k$ with $a_i=s(\sigma_i)$. Since $a_2\not\in \sigma_1$ and $a_1\not\in \sigma_0$ 
we have that $\sigma_0\subseteq \sigma_1$ and $\sigma_1\subseteq \sigma_2$. If follows that $\{\sigma_0,\sigma_2\}\in\beta A_k$ and 
since $s$ is simplicial: $\{a_0,a_2\}\in A_k$.
\end{proof}

We extend definition \eqref{E:toprep} of Section~\ref{SS:Prespace} as follows.
\begin{definition}
The {\bf topological realization} of a finite complex $A$ is the space
\[
|A| =|\beta^\infty A|=\mathrm{dom}(\beta^\infty A)/R^{\beta^\infty A}.
\]
\end{definition}
Before we characterize $|A|$ topologically, we need to introduce some notation. Let $B$ be any subcomplex of $A$ and notice that for every $m\geq 0$ the complex 
$\beta^m B$ is naturally included in  $\beta^m A$ as a subcomplex. Moreover, every map $f\colon \beta^n A \to \beta^m A$ in $\mathcal{S}(A)$ restricts to 
a map from $\beta^n B$ to $\beta^m B$ contained in $\mathcal{S}(B)$. As a consequence the generic sequence $(A_k, f^l_k)$ for $\mathcal{S}(A)$ can 
be restricted to a neat sequence in $\mathcal{S}(B)$ which we denote by $(B_k, f^l_k\res B)$. 
Clearly, the inverse limit of the above system is a closed subcomplex of $\beta^\infty A$. 
Moreover, since every map in $\mathcal{S}(B)$ can be extended, possibly in many ways, to a map in $\mathcal{S}(A)$, we see that the inverse system 
$(B_k, f^l_k\res B)$  is a generic sequence for $\mathcal{S}(B)$. As a consequence, the associated subcomplex of $\beta^\infty A$ is 
an isomorphic copy of  the projective \Fraisse{}  limit  of $\mathcal{S}(B)$ which we denote by $\beta^\infty  B$.  The following theorem provides 
a refined version of the second part of the statement of Theorem \ref{Theorem:Intro1}.

\begin{theorem}\label{T:topologicalRealization} Let $A$ be a finite complex. 
Then, $|A|$ is homeomorphic to the geometric realization $|A|_{\mathbb{R}}$ as in \eqref{E:geomr}, 
by a homeomorphism that sends $|B|$ onto the associated geometric subcomplex of $|A|_{\mathbb{R}}$ as in \eqref{E:geos}, for each subcomplex $B$ of $A$.
\end{theorem}

The proof of Theorem~\ref{T:topologicalRealization} must be postponed till Section~\ref{Su:proofend} as it requires notions and results proved in Sections~\ref{S:CoinitialClosure} and \ref{S:ResultsOnHereditarilyCellularMaps}. 
Here we foreshadow it with Lemma~\ref{L:geomsel} below. 
By arguments more direct than our proof of Theorem~\ref{T:topologicalRealization}, but using \cite[Chapters 8, 25, 26]{Da}, 
one can obtain partial information on the topological nature of $|A|$, for example, 
that $|A|$ is homeomorphic to $|A|_{\mathbb R}$ if ${\rm dim}(A)\leq 2$ or that the topological dimension of $|A|$ is equal to ${\rm dim}(A)$. 

\begin{lemma}\label{L:geomsel}
Fix $n\in {\mathbb N}$. Let $s_i\colon \beta^{n+i+1} A\to \beta^{n+i}A$ be elementary selections, for $i\in {\mathbb N}$. Consider 
the profinite simplicial complex ${\mathbb A}$ associated with the sequence $(s_i)$, 
and assume $R^{\mathbb A}$ is an equivalence relation. Then, ${\mathbb A}/R^{\mathbb A}$ is homeomorphic with $|A|_{\mathbb{R}}$ 
by a homeomorphism that sends ${\mathbb B}/R^{\mathbb A}$ onto the associated geometric subcomplex of $|A|_{\mathbb{R}}$, for each subcomplex $B$ of $A$, 
where ${\mathbb B}$ is the profinite simplicial complex associated with  $(s_i\res \beta^{n+i+1}B)$. 
\end{lemma}

\begin{proof} 
Let $A_i=\beta^{n+i}A$. Let $d$ be the Euclidean metric on $|A|_{\mathbb R}$. 
Let $\epsilon_i\geq 0$ be the supremum of the $d$-diameters of the simplexes in 
the standard $(n+i)$-th geometric barycentric subdivision of $|A|_{\mathbb{R}}$. We have that $\lim_i \epsilon_i =0$. 
Let $r_i\colon \mathrm{dom}(A_i)\to |A|_{\mathbb{R}}$ be the map that sends each vertex of $A_i$ to the corresponding vertex of 
the standard $(n+i)$-th geometric barycentric subdivision of $|A|_{\mathbb{R}}$, and set
\[
g_i:=r_i\circ s^{\infty}_i.  
\]
Each map $g_i\colon \mathrm{dom}({\mathbb A})\to |A|_{\mathbb{R}}$
 is clearly continuous. 
The sequence of maps $(g_i)$ converges uniformly since for every $x\in {\rm dom}({\mathbb A})$ and $i\leq j$, we have that
\[
d(g_i(x), g_j(x))\leq \epsilon_i. 
\]
Let $g\colon {\rm dom}({\mathbb A}) \to |A|_{\mathbb{R}}$ be the continuous function that is the uniform limit of the sequence $(g_i)$. Since 
the image of $g_i$ contains all vertices of the standard $(n+i)$-th geometric barycentric subdivision of $|A|_{\mathbb{R}}$, we see that the image of 
$g_i$ is $\epsilon_i$-dense in $|A|_{\mathbb{R}}$. Now, by compactness of $\mathrm{dom}({\mathbb A})$ and uniform convergence of $(g_i)$ to $g$, 
we get that  $g$ is surjective. We check that, for $x,y\in {\rm dom}({\mathbb A})$, 
\begin{equation}\label{E:invariants} 
g(x)=g(y)\;\hbox{ if and only if }\; xR^{\mathbb A}y. 
\end{equation}
If $xR^{\mathbb A}y$, then, for each $i$, we have  $d(g_i(x), g_i(y)) \leq \epsilon_i$; thus, $g(x)=g(y)$. 

Assume now that $\neg(x R^{\mathbb A} y)$. 
Let $C\subseteq {\mathbb A}$ be an $R^{\mathbb A}$-equivalence class.  Since $C$ forms a clique with respect to $R^{\mathbb A}$, 
so does $s^\infty_i(C)$ with respect to $R^{A_i}$. Therefore, for each $i$, $g_i(C)$ is the set of vertices of 
a simplex in the standard $(n+i)$-th geometric barycentric subdivision of $|A|_{\mathbb{R}}$. We denote that simplex by 
\[
(C)_i.
\]
Note that, if $C$ is an $R^{\mathbb A}$-equivalence class, then, for all $j\geq i$, $g_j(C)\subseteq (C)_i$. Hence, 
\begin{equation}\label{E:inclc} 
g(C)\subseteq (C)_i,\; \hbox{ for all }i.
\end{equation} 
Note further that if $C,\, D$ are distinct $R^{\mathbb A}$-equivalence classes, the 
\begin{equation}\label{E:inclc2}
(C)_i\cap (D)_i=\emptyset, \; \hbox{ for large enough }i.
\end{equation} 
So assuming that 
$\neg(x R^{\mathbb A} y)$, we have $[x]_{R^{\mathbb A}}\not= [y]_{R^{\mathbb A}}$, and obviously 
$x\in [x]_{R^{\mathbb A}}$ and $y\in [y]_{R^{\mathbb A}}$. Now  $g(x)\not= g(y)$ follows from \eqref{E:inclc} and \eqref{E:inclc2}. 

As a consequence of \eqref{E:invariants} and $g$ being a continuous surjection, we see that the map 
${\mathbb A}/R^{\mathbb A} \to |A|_{\mathbb{R}}$ induced by $g$ is a homeomorphism. The assertion about a subcomplex $B$ follows 
directly from the observation that, for each $i$, $s_i\res \beta^{n+i+1} B$ maps $\beta^{n+i+1} B$ onto $\beta^{n+i} B$ and from the definition 
of the sequence $(g_i)$. 
\end{proof}

\section{Domination closure of categories}\label{S:CoinitialClosure}

Throughout this section, we fix two  categories $\mathcal{F}$ and $\mathcal{E}$, with $\mathcal{F}\subseteq \mathcal{E}$
and $\mathrm{Ob}(\mathcal{F})=\mathrm{Ob}(\mathcal{E})$. 
 In Section~\ref{Su:coo} we develop the basic theory of the domination closure $[\mathcal{F}]$ of $\mathcal{F}$ in this abstract setup. 
 In Sections~\ref{SS:CoinitionalClosureFraisse} and \ref{SS:CoinitialClosureConcreteFraisse} we consider the cases where $\mathcal{F}$ is  
 additionally a \Fraisse{} category and a concrete \Fraisse{} category, respectively. An important example to keep in mind is the case where $\mathcal{F}$ is 
 the category $\mathcal{S}(\Delta)$ of selections  and  $\mathcal{E}$ is the ambient category $\mathcal{R}(\Delta)$ of all 
 face-preserving simplicial maps, i.e., all simplicial maps $f\colon \beta^l\Delta\to \beta^k\Delta$ with the  property that for every $\mathrm{X}\in \Delta$ 
we have that the restriction of $f$ on $\beta^l[\mathrm{X}]$ is an epimorphism from $\beta^l[\mathrm{X}]$ to $\beta^k[\mathrm{X}]$.
The computation of the domination closure of $\mathcal{S}(\Delta)$ in $\mathcal{R}(\Delta)$ will be done  in Section \ref{S:CategoriesOfMaps}.

\subsection{Domination closure---definition and main properties}\label{Su:coo}

Let $X$ be any subset of $\mathcal{E}$. We say that $\mathcal{F}$  {\bf is dominating in}  $X$ if we have that:
\begin{enumerate}
\item[(i)] $g\circ f \in X$ for every $f\in \mathcal{F}$ and $g\in X$ with $\mathrm{codom}(f)=\mathrm{dom}(g)$; 
\item[(ii)] for every $g'\in X$ there is $g''\in X$ so that $g'\circ g'' \in {\mathcal F}$. 
\end{enumerate}  
Notice that our use of the term ``dominating" is non-standard. However, it coincides  with its common use  (see e.g. \cite[Section 3.2]{Ku}) when $X$ is a category. In this case,  condition (i) is equivalent to $\mathcal{F}\subseteq X$, which in our context implies that $\mathrm{Ob}(X)=\mathrm{Ob}(\mathcal{F})$.

\begin{definition}
The {\bf domination closure of $\mathcal{F}$ in $\mathcal{E}$} is the union of all sets $X\subseteq \mathcal{E}$ in which $\mathcal{F}$ is dominating. We denote this union by $[\mathcal{F}]$.
\end{definition}

The next proposition collects the fundamental properties of domination closure. 

\begin{theorem}\label{P:lar}
Let $\mathcal{F}, \mathcal{E}$ be two categories with $\mathcal{F}\subseteq\mathcal{E}$ and
 $\mathrm{Ob}(\mathcal{F})=\mathrm{Ob}(\mathcal{E})$. 
\begin{enumerate}
\item[(i)] $[\mathcal{F}]$ is a category and $\mathcal{F}$ is dominating in $[\mathcal{F}]$; in particular, $\mathcal{F}\subseteq [\mathcal{F}]$.
\item[(ii)] $[[\mathcal{F}]]=[\mathcal{F}]$.
\end{enumerate}
\end{theorem}

We start by establishing the following basic transitivity property of domination.  

\begin{lemma}\label{L:tea}
Let  $\mathcal{F}, \mathcal{F}'$ be two subcategories of  $\mathcal{E}$ and let $X$ be any subset of $\mathcal{E}$. 
If $\mathcal{F}$ is dominating in $\mathcal{F}'$ and $\mathcal{F}'$ is dominating in $X$, then $\mathcal{F}$ is dominating in $X$. 
\end{lemma}

\begin{proof} First, we check point (ii) of  $\mathcal{F}$ being dominating in $X$. Let $g\in X$. There exists $f''\in X$ such that $g\circ f''\in \mathcal{F}'$. 
Now there exists $f'\in \mathcal{F}'$ such that 
$(g\circ f'')\circ f'\in \mathcal{F}$. So $g\circ (f''\circ f')\in \mathcal{F}$ and $f''\circ f'\in X$ since $X$ is closed under pre-composition by elements of $\mathcal{F}'$. 
To check point (i)  of $\mathcal{F}$ being dominating in $X$ notice that $\mathcal{F}\subseteq \mathcal{F}'$ since $\mathcal{F}$ is dominating in $\mathcal{F}'$, and 
$\mathcal{F}$ and $\mathcal{F}'$ are categories. 
If $g\in X$ and $f\in \mathcal{F}$ are such that $g\circ f$ is defined, then $f\in \mathcal{F}'$ since $\mathcal{F}\subseteq \mathcal{F}'$, and we get $g\circ f\in X$ 
since $\mathcal{F}'$ is dominating in $X$. We conclude that $\mathcal{F}$ is dominating in $X$. 
\end{proof}

\begin{proof}[Proof of Theorem~\ref{P:lar}]
(i) It is clear $\mathcal{F}$ is dominating  any union of sets which are dominated by $\mathcal{F}$. Thus, $\mathcal{F}$ is dominating in $[\mathcal{F}]$.  

Note that ${\mathrm{id}}_o\in [\mathcal{F}]$ for each $o\in \mathrm{Ob}(\mathcal{F})$ since $\mathcal{F}\subseteq [\mathcal{F}]$, which holds since $\mathcal{F}$ is dominating in itself. 
It remains to check that $[\mathcal{F}]$ is closed under composition. This conclusion is an immediate consequence from the following observation: 
if $\mathcal{F}$ is dominating in $X$, then $\mathcal{F}$ is dominating in the set obtained from $X$ by closing it under composition. To prove this observation,  
fix $g_0\circ \cdots \circ g_n$ with $g_0, \dots, g_n\in X$. First, we need to see that if $f\in \mathcal{F}$ and $(g_0\circ \cdots \circ g_n)\circ f$ is defined, 
then it is a product of elements of $X$. This is clear since $g_n\circ f\in X$ and 
\[
(g_0\circ \cdots \circ g_n)\circ f = g_0\circ \cdots \circ g_{n-1}\circ (g_n\circ f).
\] 

Now continue fixing $g_0\circ \cdots \circ g_n$ with $g_0, \dots, g_n\in X$. Find $g_n'\in X$ such that $f_n:= g_n\circ g_n'$ is in $\mathcal{F}$. 
Note that $g_{n-1}\circ f_n$ is in $X$ since $X$ is closed under pre-composition by elements of $\mathcal{F}$. 
So there exists
$g'_{n-1}\in X$ such that $f_{n-1}= (g_{n-1}\circ f_n)\circ g'_{n-1}$ is in $\mathcal{F}$. Now consider $g_{n-2}\circ f_{n-1}$, which is in $X$, 
and continue as above eventually producing $f_0\in \mathcal{F}$ and $g_0'\in X$ with $f_0= (g_0\circ f_1)\circ g_0'$. By unraveling the definitions, 
one easily checks that 
\[
\mathcal{F}\ni f_0 = (g_0\circ \cdots \circ g_n)\circ (g_n'\circ \cdots \circ g_0'). 
\]

(ii) One only needs to check that if $[\mathcal{F}]$ is dominating in a set $X$, then so is $\mathcal{F}$. This follows from (i) and Lemma~\ref{L:tea}. 
\end{proof}

There is another way of generating $[\mathcal{F}]$. Let $\mathcal{F}$ be a category and let $X\subseteq {\mathbb \mathcal{E}}$. Define 
$\delta_\mathcal{F}(X)$ 
to be the set of all $g\in \mathcal{E}$ that fulfill the following condition
\[  
\exists h\in X \; (g\circ h\in \mathcal{F})\; \hbox{ and }\; \forall f\in \mathcal{F} \; (\hbox{if } g\circ f \hbox{ is defined, then }g\circ f\in X).  
\]
Note that $\delta_{\mathcal{F}}$ is monotone, that is, $X\subseteq Y$ implies $\delta_{\mathcal{F}}(X)\subseteq 
\delta_{\mathcal{F}}(Y)$. Note also that the second part of the condition in the definition of $\delta_{\mathcal{F}}$ insures that $\delta_{\mathcal{F}}(X)\subseteq X$. It follows \cite[7.36]{Mos} that 
$\delta_{\mathcal{F}}$ has a greatest fixed point, that is, there exists a set $X_0\subseteq \mathcal{E}$ such that $\delta_{\mathcal{F}}(X_0)=X_0$ and, 
for each $X$ with $\delta_{\mathcal{F}}(X)=X$, we have $X\subseteq X_0$. 

\begin{proposition}\label{P:del}
$[\mathcal{F}]$ is equal to the greatest fixed point of  $\delta_{\mathcal{F}}$.
\end{proposition}
\begin{proof} Observe that $\delta_{\mathcal{F}}(X)=X$ precisely when $\mathcal{F}$ is dominating in $X$. Hence, the conclusion follows from the definition of $[\mathcal{F}]$ and Theorem~\ref{P:lar}(i).
\end{proof}

\subsection{Domination closure---Fra{\"i}ss{\'e} categories}\label{SS:CoinitionalClosureFraisse}

\begin{proposition}\label{P:cos}
Let $\mathcal{E}$ be an essentially countable category, and let $\mathcal{F},\mathcal{F}'$ be subcategories of $\mathcal{E}$ with $\mathrm{Ob}(\mathcal{F})=\mathrm{Ob}(\mathcal{E})$. Assume that $\mathcal{F}$ is dominating in $\mathcal{F}'$.  

\begin{enumerate}
\item[(i)] If $\mathcal{F}$ satisfies the joint projection property, 
then so does $\mathcal{F}'$. 

\item[(ii)] If $\mathcal{F}$ satisfies the projective amalgamation property, then so does $\mathcal{F}'$.  

\item[(iii)] Every generic sequence for $\mathcal{F}$ is generic for $\mathcal{F}'$. 

\item[(iv)] If ${\mathcal F}'$ has a generic sequence, then it has one whose morphisms 
are in ${\mathcal F}$. 
\end{enumerate}
\end{proposition}

\begin{proof} 
(i) If $\mathcal{F}$ is dominating in $\mathcal{F}'$, then $\mathcal{F}\subseteq\mathcal{F}'$. Since $\mathrm{Ob}(\mathcal{F})=\mathrm{Ob}(\mathcal{F}')$, it follows 
that if $\mathcal{F}$ satisfies the joint projection property, then so does $\mathcal{F}'$

(ii) Assume now that $\mathcal{F}$ satisfies the projective amalgamation property and let  $f_1', f_2'\in \mathcal{F}'$ having  the same codomain. 
Find $g_1', g_2'\in \mathcal{F}'$ 
such that $f_1'\circ g_1', f_2'\circ g_2'\in \mathcal{F}$. Let  $g_1, g_2\in \mathcal{F}$ with   
\[
(f_1'\circ g_1')\circ g_1= (f_2'\circ g_2')\circ g_2, 
\]
so
\[
f_1'\circ (g_1'\circ g_1)= f_2'\circ (g_2'\circ g_2). 
\]
Since $g_1'\circ g_1, g_2'\circ g_2\in \mathcal{F}'$ we have that $\mathcal{F}'$ satisfies the projective amalgamation property. 

(iii) Let now  $(f_i)$ be a generic sequence for $\mathcal{F}$. Since $\mathcal{F}\subseteq \mathcal{F}'$  
it is  enough to check that for each $f'\in \mathcal{F}'$ and $i_0$ with 
$f'$ and $f_{i_0}$ having the same codomains, there are $g\in \mathcal{F}'$ and $j_0>i_0$ such that $f'\circ g = f_{i_0}\circ \cdots \circ f_{j_0}$. 
Since $\mathcal{F}$ is dominating in $\mathcal{F}'$, there exists $f''\in \mathcal{F}'$ such that $f'\circ f''\in \mathcal{F}$. 
Since $(f_i)$ is Fra{\"i}ss{\'e} for $\mathcal{F}$, there are $j_0>i_0$ and $f\in \mathcal{F}$ such that 
\[
f'\circ (f''\circ f) = (f'\circ f'')\circ f = f_{i_0}\circ \cdots \circ f_{j_0},
\]
and we are done by taking $g= f''\circ f$. 

(iv) Let $(f_j')$ be a generic sequence for ${\mathcal F}'$. We find $g_i, h_i\in {\mathcal F}'$, $f_i\in {\mathcal F}$, 
and a sequence of natural numbers $(j_i)$ with 
\begin{equation}\label{E:smgen}
\begin{split}
j_i &< j_{i+1}\\
g_i\circ h_i &= f_i \\
h_i\circ g_{i+1} &= f'_{j_i}\circ \cdots \circ f'_{j_{i+1}}
\end{split}
\end{equation}
To start the construction we take $g_0= f_0'$ and $j_0=1$. 
Having constructed $g_i$, we find $h_i$ and $f_i$ by ${\mathcal F}$ being dominating in ${\mathcal F}'$. Having constructed $j_i$ and $h_i$, we find 
$j_{i+1}$ and $g_{i+1}$ by $(f_j')$ being generic for ${\mathcal F}'$. 

Properties \eqref{E:smgen} and genericity of $(f_j')$ for ${\mathcal F}'$ immediately imply that $(f_i)$ is generic for ${\mathcal F}'$, so this sequence is as required. 
\end{proof}

The following corollary follows immediately from Proposition~\ref{P:cos}. 

\begin{corollary}
If $\mathcal{F}$ is a Fra{\"i}ss{\'e} subcategory of an essentially countable category $\mathcal{E}$, then so is $[\mathcal{F}]$. 
Moreover, every generic sequence in $\mathcal{F}$ is also generic in $[\mathcal{F}]$.
\end{corollary}

The next theorem provides a characterization of the elements of $[\mathcal{F}]$ under the assumption that $\mathcal{F}$ is \Fraisse{}, 
and it will be used in Section \ref{S:ResultsOnRestrictedNearHomeo}.
To phrase this characterization, we need a notion of iso-sequence. 
A sequence $(g_i)$ is called an {\bf $\mathcal E$-iso-sequence for} $\mathcal{F}$ if each $g_i$ is in $\mathcal E$, the sequence $(g_i)$ 
is neat, and the sequences 
$(g_{2i}\circ g_{2i+1})$ and $(g_{2i+1}\circ g_{2i+2})$ are both generic sequences for $\mathcal{F}$. 
Note that the sequence $(g_i)$ is an isomorphism between the two generic sequences.

\begin{theorem}\label{T:iid}
Let $\mathcal{F}$ be a countable Fra{\"i}ss{\'e} category and let $g\in \mathcal{E}$. 
Then $g\in [\mathcal{F}]$ if and only if there exists an $\mathcal E$-iso-sequence $(g_i)$ for $\mathcal{F}$ such that $g=g_0$. 
\end{theorem}

\begin{proof} $(\Leftarrow)$ Fix an $\mathcal E$-iso-sequence $(g_i)$ with $g=g_0$. Define $C$ to consist of all morphisms of the form 
$f\circ g_i \circ f'$ with $f, f'\in \mathcal{F}$ and $i\in {\mathbb N}$. We show that $\delta_{\mathcal{F}}(C)=C$. 
It then follows from Proposition~\ref{P:del} that $C\subseteq [\mathcal{F}]$; in particular, we get $g_0\in [\mathcal{F}]$ since $g_0\in C$. 

Note that since $\delta_{\mathcal{F}}(X)\subseteq X$ for each $X$, it will suffice to prove that $C\subseteq \delta_{\mathcal{F}}(C)$. 
To see $C\subseteq \delta_{\mathcal{F}}(C)$, first, we need to show that if $g\in C$, then $g\circ h\in \mathcal{F}$ for some $h\in C$. So given $f', f''\in \mathcal{F}$ and $i$, we need to find $h\in C$ 
with $f'\circ g_i\circ f''\circ h\in \mathcal{F}$. Since $\mathcal{F}$ is closed under composition, it suffices to find $h\in C$ with $g_i\circ f''\circ h\in \mathcal{F}$. Since 
${\rm codom}(f'') = {\rm dom}(g_i)$, we have that 
\[
{\rm codom}(f'') = {\rm codom}(g_{i+1}\circ g_{i+2}). 
\]
Since $f''\in \mathcal{F}$, there exist $f\in \mathcal{F}$ and $j>i$ of the same parity as $i$ such that 
\[
f''\circ f = (g_{i+1}\circ g_{i+2})\circ \cdots \circ (g_{j+1}\circ g_{j+2}). 
\]
Then we have 
\[
(g_i\circ f'') \circ (f\circ g_{j+3}) = (g_i\circ g_{i+1})\circ \cdots \circ (g_{j+2}\circ g_{j+3}). 
\]
After noticing that the right-hand side of the equality above is in $\mathcal{F}$ and that $f\circ g_{j+3}$ is in $C$, we see that we can take $h= f\circ g_{j+3}$. 

To complete the argument for $C\subseteq \delta_{\mathcal{F}}(C)$, we need to show that if $g\in C$, then $g\circ f\in C$ for each $f\in \mathcal{F}$ for which $g\circ f$ is defined. 
This is clear since $g=f'\circ g_i\circ f''$ for some $i$ and $f', f''\in \mathcal{F}$ and, therefore, 
\[
g\circ f =  f'\circ g_i\circ (f'' \circ f)\in C, 
\]
as $F$ is closed under composition. 

$(\Rightarrow)$ Given $g\in [\mathcal{F}]$, we modifying the Fra{\"i}ss{\'e} sequence construction to build two Fra{\"i}ss{\'e} sequences and an isomorphism between them whose first element is $g$.  We only indicate the inductive step 
in the construction and leave the bookkeeping details of the construction to the reader. The inductive step is included in the following claim. 

\begin{claim}
Let $g=g_0, g_1, \dots, g_{j_0}\in [\mathcal{F}]$ be such that ${\rm dom}(g_i)={\rm codom}(g_{i+1})$ and $g_i\circ g_{i+1}\in \mathcal{F}$ for all $i<j_0$. Let 
$i_0\leq j_0$ be of the same parity as $j_0$. Let $f',f''\in \mathcal{F}$ have the same codomain and ${\rm dom}(f') = {\rm codom}(g_{i_0})$. Then there exist 
$g_{j_0+1}\in [\mathcal{F}]$ and $f\in \mathcal{F}$ such that 
\[
g_{j_0}\circ g_{j_0+1}\in \mathcal{F}\;\hbox{ and }\; f''\circ f = f' \circ (g_{i_0}\circ g_{i_0+1}) \circ \cdots \circ (g_{j_0}\circ g_{j_0+1}). 
\]
\end{claim}

\noindent {\em Proof of Claim.} Since $g_{j_0}\in [\mathcal{F}]$, there exists $g\in [\mathcal{F}]$ with $g_{j_0}\circ g\in \mathcal{F}$. Consider the following two morphisms in $\mathcal{F}$ 
with the same codomain: 
\[
f''\;\hbox{ and }\; f' \circ (g_{i_0}\circ g_{i_0+1}) \circ \cdots \circ (g_{j_0}\circ g). 
\]
Since $\mathcal{F}$ is Fra{\"i}ss{\'e}, there exist $f, h\in \mathcal{F}$ such that 
\[
f''\circ f = f' \circ (g_{i_0}\circ g_{i_0+1}) \circ \cdots \circ (g_{j_0}\circ g)\circ h. 
\]
Let $g_{j_0+1} = g\circ h$. By Theorem~\ref{P:lar}(i), we see that $g\circ h\in [\mathcal{F}]$, so $g_{j_0+1}$ is as required, which proves the claim and the theorem. 
\end{proof}

\subsection{Domination closure---concrete Fra{\"i}ss{\'e} categories}\label{SS:CoinitialClosureConcreteFraisse}

Let $\mathcal{F}$ be a concrete \Fraisse{} category and let $\mathbb{A}$ be its \Fraisse{} limit. 
Theorem~\ref{T:iid} reformulates as follows.

\begin{theorem}\label{T:car} 
If $\mathcal F$ is a  concrete Fra{\"i}ss{\'e} category, then $[{\mathcal F}]$ 
consists of all $h\colon B\to A$, for which there exist $g\colon {\mathbb A}\to B$ that is approximable by $\mathcal F$ and $\phi\in {\rm Aut}({\mathbb A})$ 
with $h\circ g\circ \phi$ in $\mathcal F$.
Also, $[{\mathcal F}]$ 
consists of all $h\colon B\to A$ such that for each $g\colon {\mathbb F}\to B$ that is approximable by $\mathcal F$, there exists $\phi\in {\rm Aut}({\mathbb F})$ 
with $h\circ g\circ \phi$ in $\mathcal F$.
\end{theorem}

\begin{proof} Assume $h\colon B\to A$ is in $[{\mathcal F}]$. By Theorem~\ref{T:iid}, we can fix an iso-sequence $(h_k)$ for $\mathcal F$ such that $h_0=h$. Let 
${\mathbb B}$ and ${\mathbb C}$ be the projective limits of Fra{\"i}ss{\'e} sequences $((h_{B})_i)$ and $((h_{C})_i)$ for $\mathcal F$, where 
$(h_{B})_i=h_{2i}\circ h_{2i+1}$ and $(h_{C})_i = h_{2i+1}\circ h_{2i+2}$.  Fix continuous isomorphisms 
$\psi_C\colon {\mathbb C} \to {\mathbb A}$ and $\psi_B\colon {\mathbb A}\to {\mathbb B}$. Since ${\mathbb A}$, ${\mathbb C}$ and ${\mathbb B}$ 
are all projective limits of Fra{\"i}ss{\'e} sequences for $\mathcal F$, both $\psi_B$ and $\psi_C$ are $\mathcal F$-isomorphisms. 
Let $\phi'\colon {\mathbb B}\to {\mathbb C}$ be the isomorphism induced by $(h_k)$ (which is not necessarily an $\mathcal F$-isomorphism). 
Define 
\[
g_B = (h_B)^\infty_0\circ \psi_B,\;  g_C = (h_C)^\infty_0\circ \psi_C^{-1}, \hbox{ and } \phi=\psi_C\circ \phi' \circ \psi_B. 
\]
It is now routine to check that $g_B$ and $g_C$ are approximable by $\mathcal F$, and  $h\circ g_C\circ \phi = g_B$.

Now assume $h\circ g \circ \phi$ is approximable by $\mathcal F$, for some $\phi\in {\rm Aut}({\mathbb A})$ 
and similarly is $g\colon {\mathbb A}\to B$. Let $\phi$ be induced by an isomorphism $(h_k)$ from $(f_i)$ to $(f_i)$, where 
$(f_i)$ is a Fra{\"i}ss{\'e} sequence whose limit is $\mathbb A$. Since $g$ is in $\mathcal F$, we can 
find $i_0$ such that for all $i\geq i_0$ there exists an $f'$ that is approximable by $\mathcal F$ such that $g = f'_i\circ f^\infty_{i}$. Since $h\circ g\circ \phi$ is 
approximable by $\mathcal F$, we can find $k_0$ such that ${\rm codom}(h_{k_0})= {\rm codom}(f_{j_0})$ for some $j_0\geq i_0$ and 
$h\circ f'_{j_0}\circ h_{k_0}$ is in $\mathcal F$. Consider the sequence $(h_i')$ defined by 
\[
\begin{split}
h_0' &= h;\\
h_1' &= f_{j_0}'\circ h_{k_0};\\
h_i' &= h_{k_0+i-1},\hbox{ for }i\geq 2.
\end{split}
\]
It is now easy to see that the sequence $(h_i')$ is an iso-sequence for $\mathcal F$ and $h'_0=h$. So $h\in [{\mathcal F}]$ by Theorem~\ref{T:iid}. 

The second sentence of the theorem follows immediately from the first one after we notice that, by universality of $\mathbb A$, for each 
$B\in \mathrm{Ob}({\mathcal F})$, there exists $g\colon {\mathbb A}\to B$ in ${\mathcal F}$ and that, by ultrahomogeneity of $\mathbb A$, 
for any $g, g'\colon {\mathbb A}\to B$ in $\mathcal F$ there exist $\phi'\in\mathrm{Aut}_{\mathcal{F}}(\mathbb{A})$ with 
$g'=g\circ \phi'$; see Theorem \ref{T:ultrahomogeneity}. 
\end{proof}

\section{Results in the theory of stellar moves}\label{S:strcomp}

In this section, we review the basic theory of stellar moves  from \cite{Ne,Al} and we prove new results that will be used in Section~\ref{S:CategoriesOfMaps}. 
In particular, in Section~\ref{S:Systems}, we introduce the notion of an induced system and, in Section~\ref{S:Cell-systems}, its specialization a cell-system.
Variations of the notion of a cell-system have been used in the literature of transversely cellular maps; see \cite{Akin,Co1,Mc}. 
Proposition~\ref{P:TameCellSystemAdmitsStarring} establishes the property of cell-systems, which will make it possible, in  
Section~\ref{S:CategoriesOfMaps}, to give a lower estimate on the domination closure $[\mathcal{S}(\Delta)]$ of  $\mathcal{S}(\Delta)$ in $\mathcal{R}(\Delta)$. 
The proof of Proposition~\ref{P:TameCellSystemAdmitsStarring} relies on two technical results---Theorems~\ref{T: star with strongly internal} and 
\ref{T:SystemTransformation}---which we find interesting in their own right.

\subsection{Basic notions from stellar theory}\label{S:StellarTheory}

We review some  basic notions of stellar theory. This theory is the combinatorial counterpart to PL-topology that was introduced in \cite{Ne,Al}. 
A more modern treatment can be found in \cite{Li}. Notationally we follow \cite{Li}, building upon the definitions of Section~\ref{SS:Complexes}.

Let $\sigma$ be a finite set and let 
$[\sigma]=\mathcal{P}(\sigma)\setminus \{\emptyset\}$
be the associated simplex. The {\bf boundary} $\partial [\sigma]$ of $[\sigma]$ is the simplicial complex  $[\sigma]\setminus\{\sigma\}$. 
Notice that $[\emptyset] =\emptyset$ and $\partial \emptyset=\emptyset$. For sets $\sigma,\tau$, if $\sigma\cap\tau=\emptyset$, we define the {\bf join} $\sigma\star\tau$ of $\sigma$ 
and $\tau$ to  be the union $\sigma\cup\tau$.  In other words, the expression $\sigma\star\tau$ implies $\sigma\cap\tau=\emptyset$ and 
$\sigma\star\tau=\sigma\cup\tau$. Similarly, for complexes $A,B$ with 
$\mathrm{dom}(A)\cap\mathrm{dom}(B)=\emptyset$, 
we set 
\[
A\star B=A\cup\{\sigma \star \tau \colon \; \sigma\in A, \; \tau\in B\} \cup B
\] 
to be {\bf join} of $A$ and $B$. Notice that $A\star\emptyset=\emptyset\star A=\emptyset$. 
Let $A$ be a complex and $\sigma$ be a set. We define the {\bf open star} $\mathrm{st}(\sigma,A)$ of $\sigma$ in $A$ to be the collection
\[
\mathrm{st}(\sigma,A) = \{\tau\in A \colon \sigma\subseteq\tau\}, 
\]
and the {\bf link} $\mathrm{lk}(\sigma, A)$ of $\sigma$ in $A$ to be the complex 
\[
\mathrm{lk}(\sigma, A) = \{\tau\in A \colon \sigma\star\tau\in A\}. 
\]
In particular, 
$\mathrm{st}(\emptyset, A)=\mathrm{lk}(\emptyset, A)=A$, and $\mathrm{st}(\sigma, A)=\mathrm{lk}(\sigma, A)=\emptyset$, if $\sigma\not\in A, \sigma\neq \emptyset$.

Let $A$ be a simplicial complex, $\sigma$ be a finite set and $a$ some point.  We define the {\bf stellar subdivision $(\sigma, a)$ on $A$} to be the simplicial complex 
\[
(\sigma,a)A= \begin{cases}
\big( A \setminus \mathrm{st}(\sigma,A)\big) \cup \big([a]\star\partial[\sigma]\star\mathrm{lk}(\sigma, A)\big),& \text{ if } \sigma\in A \text{ and }a\not\in {\rm dom}(A);\\
A,& \text{ otherwise}. 
\end{cases}
\]
Similarly, we define the {\bf stellar weld $(\sigma, a)^{-1}$ on $A$} to be the simplicial complex 
\[
\begin{split}
(\sigma&,  a)^{-1} A= 
\begin{cases} 
\big(A\setminus \mathrm{st}(\{a\},A)\big) \cup \big([\sigma]\star L\big),&  \text{ if } \sigma\not\in A,\, a\in \mathrm{dom}(A), \text{ and}  \\ &  \mathrm{lk}(\{a\},A)= \partial [\sigma] \star L, \text{ for some } L  ;\\
A,&  \text{ otherwise}; 
\end{cases}
\end{split}
\]
where $L$ is a complex. The notation is justified by the easily checked identities 
\[
(\sigma, a)^{-1} (\sigma, a) A = A, 
\]
for each complex $A$ with $\sigma\in A$ or $a\not\in {\rm dom}(A)$, and 
\[
(\sigma, a) (\sigma, a)^{-1} A = A, 
\]
for each complex $A$ with $\sigma\not\in A$ or $a\in {\rm dom}(A)$. 
In the first clauses of the definitions of $(\sigma, a)$ and $(\sigma, a)^{-1}$, we say that $(\sigma, a)$ or $(\sigma, a)^{-1}$, respectively, is {\bf essential for $A$}. 
We say that $(\sigma, a)$ and $(\sigma, a)^{-1}$ are {\bf based on } $\sigma$, and we call $a$ the {\bf vertex of} $(\sigma, a)$ and $(\sigma, a)^{-1}$, respectively.
We refer to both stellar subdivisions and stellar welds using the term {\bf stellar moves}.

Two simplicial complexes $A,B$ are {\bf stellar equivalent} if there exists a sequence $\delta_0, \dots , \delta_k$ of stellar moves such that $\delta_k \cdots \delta_0A=B$.  A {\bf stellar $n$-ball} is a simplicial complex $B$ stellar equivalent to the $n$-simplex $\Delta$. A {\bf stellar $(n-1)$-sphere} is a simplicial complex $S$ stellar equivalent to $\partial \Delta$. Notice that $\Delta^{-1}=\emptyset =\partial \Delta^{-1}$ is the only complex that is both a stellar $(-1)$-ball and a stellar $(-2)$-sphere.  A {\bf stellar $n$-manifold} is a simplicial complex $M$ with the property that the link $\mathrm{lk}(\{v\},M)$ of  every vertex $v$ of $M$ is either a stellar $(n-1)$-sphere or a stellar $(n-1)$-ball. This implies (see \cite{Li}) that $\mathrm{lk}(\sigma,M)$ is either a stellar sphere or 
a stellar ball of dimension $(n-\mathrm{dim}(\sigma)-1)$. The {\bf boundary $\partial M$ of a stellar $n$-manifold} is the collection of all faces $\sigma$ of 
$M$ whose link in $M$ is not a sphere. Notice that, since $\emptyset$ is both a sphere and a ball, this is not the same as the collection of all faces whose link in $M$ is a ball.
The boundary $\partial M$ forms a subcomplex of $M$ which in particular is a stellar $(n-1)$-manifold without boundary when $n\geq(-1)$. If $\delta$ is a stellar move  and $M$ is stellar manifold then it is easy to see that $\partial (\delta M)= \delta (\partial M)$. As a consequence $\partial B$ is always a stellar $(n-1)$-sphere when $B$ is a stellar $n$-ball and $\partial S=\emptyset$ if $S$ is a stellar sphere. Let $B^p$ and $S^p$  be, respectively, a stellar ball and a sphere of dimension $p\geq 0$; and similarly let $B^q$ and $S^q$  be a stellar ball and a sphere
of  dimension $q\geq 0$. Then the following useful facts follow directly from the definitions:
\begin{equation}
 \label{SB-identities}
\begin{aligned}
 B^p \star B^q  \text{ and } B^p \star S^q \text{ are stellar } (p+q+1)\text{-balls}; \\ S^p \star S^q \text{ is a stellar } (p+q+2)\text{-sphere}.
\end{aligned}
\end{equation}

\subsection{Strongly internal moves} \label{S: Tame balls}
    
Let $B$ be a stellar ball and let $\delta$ be a stellar move based on $\sigma$. We say that $\delta$ is {\bf internal} for $B$ if  $\delta$ is essential for $M$ but not essential $\partial M$.
Let $B$ be a stellar $n$-ball. A {\bf starring} of $B$ is a sequence  $\delta_0,\ldots,\delta_{k}$ of
stellar moves so that:
\begin{itemize}
\item for every $i\leq k$, $\delta_{i}$ is internal for $\delta_{i-1}\cdots\delta_1\delta_0 B$;
\item $\delta_{k}\ldots\delta_{0} B =[v]\star\partial B$ for some element $v$. 
\end{itemize}

The following theorem of Alexander and Newman is the main technical result on which most of the theory of stellar manifolds is based.

\begin{theorem}[\cite{Ne,Al} see also \cite{Li}]\label{Th: starring}
Every stellar $n$-ball $B$ admits a starring. 
\end{theorem}

Here, we introduce the notion of a strongly internal stellar move and we show, in Theorem \ref{T: star with strongly internal}, 
that the starring of a stellar ball can always be assumed to be strongly internal, as long as the boundary of $B$ is an induced subcomplex. 
    
\begin{definition}\label{def: tame ball}
A stellar ball $B$ is {\bf tame} if its boundary $\partial B$ is an induced subcomplex of $B$, i.e., $\sigma\in B$ and 
$\sigma\subseteq\mathrm{dom}(\partial B)$ implies $\sigma\in\partial B$.
\end{definition}

Notice that $\Delta$, in contrast to $[v]\star \partial \Delta$,  is not a tame $n$-ball. However, every stellar ball becomes tame after a barycentric  subdivision.  It is easy to see that a stellar subdivision of a tame ball is tame. The same is not always true for a weld of a tame ball. 

\begin{definition}
Let $M$ be  a stellar manifold and let $\delta$ be a stellar move based on $\sigma$. We say that $\delta$ is {\bf strongly internal for $M$}, if  $\sigma\not\subseteq\mathrm{dom}(\partial M)$.
\end{definition}

It is easy to check that the class of tame balls is invariant under strongly internal stellar moves. 
The following is the main result of this subsection.

\begin{theorem}\label{T: star with strongly internal}
Let $B$ be a tame ball. Then $B$ can be starred using only strongly internal moves.
\end{theorem}

The proof of Theorem \ref{T: star with strongly internal} will rely on the following lemma.

\begin{lemma}\label{L: stellar commutator}
Let $K$ be a stellar ball and let $\sigma\in K$. Let also $\delta$ be an internal stellar move on $K$ so that $\sigma\in \delta K$. Then $(\sigma,a)\delta K= \bar{\varepsilon}(\sigma,a) K$, where $\bar{\varepsilon}$ is a sequence of strongly internal moves with the exception of exactly one entry $\varepsilon_j$, which is equal to $\delta$. 
\end{lemma}
\begin{proof}
Assume first that $\delta$ is an internal subdivision $(\tau,b)$. If $\sigma\cap\tau=\emptyset$, or if there is no face in $K$ which extends both $\sigma$ and $\tau$, then it is easy to see that $(\sigma,a)(\tau,b)K=(\tau,b)(\sigma,a)K$. Otherwise, let $\rho=\sigma\cap\tau$, $\sigma=\sigma_0\star\rho$, $\tau=\tau_0\star\rho$ and $\sigma_0\star\tau_0\star\rho\in K$. By the proof of Lemma 3.4 in \cite{Li} we have that
\[(\sigma,a)(\tau,b)K= (\{b\}\star\sigma_0,d)^{-1}(\{a\}\star\tau_0,d)(\tau,b)(\sigma,a)K.\]
Notice that both moves $(\{b\}\star\sigma_0,d)^{-1}$ and $(\{a\}\star\tau_0,d)$  here are strongly internal. 

We now take care of the case when $\delta$ is an internal weld $(\tau,b)^{-1}$. Again, if $\sigma\cap \tau=\emptyset$, or if $\sigma\not\in\mathrm{lk}(\{b\},K)=\mathrm{lk}(\tau,\delta K)\star\partial[\tau]$, then it is easy to see that $(\sigma,a)(\tau,b)K=(\tau,b)(\sigma,a)K$. Let therefore $\sigma\in\mathrm{lk}(\{b\},K)$, with $\sigma\cap\tau=\rho\neq\emptyset$, and set $\sigma=\sigma_0\star\rho$ and $\tau=\tau_0\star\rho$ with $\sigma_0\star\tau_0\star\rho\in \delta K$. Or equivalently, with
$[\sigma_0]\star \partial [\tau]\star[b]$ being a subcomplex of $K$. We will show that 
\[(\sigma,a)(\tau,b)^{-1}K= (\tau,b)^{-1} \; (\{a\}\star\tau_0,d)^{-1} \; (\{b\}\star\sigma_0,d) (\sigma,a)K.\]

It is actually enough to consider how the sequences transform the subcomplex $A=[\sigma_0]\star\partial[\tau]\star[b]$ of $K$. It is easy to compute that  
\[(\sigma,a)(\tau,b)^{-1} A= [a]\star\partial[\sigma]\star[\tau_0].\]
For the second sequence of stellar moves we use the identity 
\[\partial [X\star Y]=(\partial [X]\star [Y])\cup(\partial [Y]\star [X]),\]
 and we compute as follows:
\begin{gather*}
A = [\sigma_0]\star\partial[\tau]\star[b] = \big([\sigma_0]\star[b] \star \partial[\rho]\star [\tau_0]\big) \cup \big([\sigma_0]\star[b] \star \partial[\tau_0]\star [\rho]\big)=\\
=\big([\sigma_0]\star[b] \star \partial[\rho]\star [\tau_0]\big) \cup \big([\sigma]\star[b] \star \partial[\tau_0]\big)  \xrightarrow{(\sigma,a)} \\
\big([\sigma_0]\star[b] \star \partial[\rho]\star [\tau_0]\big) \cup \big([a]\star\partial[\sigma]\star[b] \star \partial[\tau_0]\big)= \\
\big([\sigma_0]\star[b] \star \partial[\rho]\star [\tau_0]\big) \cup \big([a]\star\partial[\sigma_0]\star[\rho]\star[b] \star \partial[\tau_0]\big) \cup \big([a]\star[b]\star[\sigma_0]\star\partial[\rho] \star \partial[\tau_0]\big)\\
\xrightarrow{(\{b\}\star\sigma_0,d)}  \big([d]\star[\sigma_0] \star \partial[\rho]\star [\tau_0]\big) \cup \big( [d]\star\partial[\sigma_0]\star[b] \star \partial[\rho]\star [\tau_0]\big) \cup\\
\cup \big([a]\star\partial[\sigma_0]\star[\rho]\star[b] \star \partial[\tau_0]\big)\cup \\
\cup \big([d]\star[a]\star[\sigma_0]\star\partial [\rho] \star \partial[\tau_0]\big) \cup \big([d]\star[a]\star[b]\star\partial[\sigma_0]\star \partial[\rho] \star \partial[\tau_0]\big).\\
\end{gather*}
Merging together the first and fourth term  of the last expression, as well as the second and fifth term, we have that the above expression is equal to

\begin{gather*}
\big(\partial([a]\star[\tau_0])\star [d] \star[\sigma_0]\star \partial [\rho]\big) \cup \big([a]\star\partial[\sigma_0]\star [\rho]\star[b] \star \partial[\tau_0]\big)\cup\\
\cup\big(\partial([a]\star[\tau_0])\star[d]\star[b]\star\partial[\sigma_0]\star\partial [\rho]\big)\xrightarrow{([a]\star\tau_0,d)^{-1}} \\
\big([a]\star[\tau_0]\star[\sigma_0]\star \partial [\rho]\big) \cup \big([a]\star\partial[\sigma_0]\star[\rho]\star[b] \star \partial[\tau_0]\big)\cup\\
\cup\big(\{a\}\star\partial[\sigma_0]\star[\tau_0]\star[b]\star\partial [\rho]\big)=\\
= \big([a]\star[\tau_0]\star[\sigma_0]\star \partial [\rho]\big) \cup \big([a]\star\partial[\sigma_0]\star[b]\star\partial ( [\tau_0]\star [\rho]) \big)\xrightarrow{(\tau,b)^{-1}} \\
= \big([a]\star[\tau_0]\star[\sigma_0]\star \partial [\rho]\big) \cup \big([a]\star\partial[\sigma_0]\star [\tau_0]\star [\rho] \big)=\\
=[a]\star\partial[\sigma]\star[\tau_0].
\end{gather*}
\end{proof}

\begin{proof}[Proof of Theorem~\ref{T: star with strongly internal}]
By Theorem~\ref{Th: starring} let $\bar{\delta}=\delta_0,\ldots,\delta_{k-1}$ be a sequence of internal moves which star $B$. 
Set $B_{-1}=B$ and $B_i=\delta_i B_{i-1}$, for every $i<k$. If for every $i<k$, $\delta_i$ is strongly internal for $B_{i-1}$, then we are done. 
Assume therefore that there is at least one non-strongly internal move in the starring sequence and pick $p$ to be the smallest index $p<k$ so that 
$\delta_p$ is not strongly internal. We have therefore that $B_{p-1}$ is a tame $n$-ball, and since internal subdivisions on tame balls are strongly internal, 
we have that $\delta_p=(\sigma,a)^{-1}$ for some set $\sigma$ with $\partial [\sigma]\in \partial B_{p-1}=\partial B_{i}$, for every $i<k$. 
Notice that $\sigma\not\in B_{k-1}=[a]\star\partial B_{k-1}$, since otherwise $\sigma\in\partial B_{k-1}= \partial B_{p}$, contradicting that $\delta_p$ is internal. 
Let $r$ be the smallest index $r>p$ with $\sigma\not\in B_r$. Then, $\delta_r$ cannot be a weld $(\tau,b)^{-1}$ because in this case $b$ would belong to 
$\sigma$ contradicting that $\delta_r$ is internal. Therefore, $\delta_r=(\sigma,a')$. Hence, for every $i$ with $p<i<r$ we have that $\sigma\in B_i$, and therefore, 
by Lemma \ref{L: stellar commutator} we can replace the sequence $\delta_0,\ldots,\delta_p,\ldots,\delta_r,\ldots ,\delta_{k-1}$ with a new starring of $B$:
\[\delta_0,\ldots,\delta_p, \delta_{r},\bar{\varepsilon}_{p+1}, \ldots,\bar{\varepsilon}_{r-1},\delta_{r+1},\ldots, \delta_{k-1}.\]
Replacing $\delta_{r}\delta_{p}$ in the above starring with $(\{a\},a')$ we produce a starring of $B$ with strictly fewer non-internal moves than 
the initial sequence $\bar{\delta}$. By induction on the number non-internal moves of $\bar{\delta}$ the proof is complete.
\end{proof}

\subsection{Systems and their transformations}\label{S:Systems}

In this section we study certain decompositions of simplicial complexes into coherent collections of subcomplexes, which we call (abstract) systems.  In Theorem~\ref{T:SystemTransformation} we provide sufficient conditions under which  such a system can be transformed to a ``canonical form.'' Concrete examples of transformable systems will be provided in Section \ref{S:Cell-systems}.

A {\bf system} is a collection $\mathcal S$ of non-empty complexes that is closed under non-empty intersections, that is, for all $C_1, C_2\in {\mathcal S}$, if 
$C_1\cap C_2\not= \emptyset$, then 
$C_1\cap C_2\in {\mathcal S}$. 
The union $\bigcup{\mathcal S}$ of all complexes in $\mathcal S$ 
is a complex; we write 
\[
U_{\mathcal S} = \bigcup{\mathcal S}. 
\]
We say that $\mathcal S$ is a {\bf system on $U_{\mathcal S}$}. 
We associate with $\mathcal S$ the partial order 
\[
P_{\mathcal S} = ({\mathcal S}, \subseteq),
\]
that is, the elements of $P_{\mathcal S}$ are complexes in $\mathcal S$ and the partial order relation is inclusion. 
Recall that the {\bf chain complex}  of a poset $P$ is simplicial complex $\mathcal{C}(P)$ which contains all finite non-empty linearly order subsets of $P$. 
In particular, $\mathrm{dom}(\mathcal{C}(P))=P$. This construction allows us to associate with $\mathcal S$ the  complex 
\[
{\mathcal C}(P_{\mathcal S}).
\] 
We note that $P_{\mathcal S}$ has the following property: if two elements $C_1, C_2$ of $P$ have a lower bound in $P$, then they have the greatest 
lower bound $C_1\cap C_2$. 
For $C\in {\mathcal S}$, we set 
\[
D_{\mathcal S}(C)=\bigcup \{ C'\in {\mathcal S}\mid C'\subsetneq C\}.
\]
Note that $D_{\mathcal S}(C)$ is a complex; if $C$ is a minimal element of $P_{\mathcal S}$, then 
$D_{\mathcal S}(C)$ is the empty complex.

Our goal will be to use the structure of a system $\mathcal S$ to transform 
$U_{\mathcal S}$ into $\mathcal{C}(P_{\mathcal S})$. In order to achieve this goal in Theorem~\ref{T:SystemTransformation}, 
we need to introduce notions that concern the interaction of a system with stellar moves. 
We say that a move $\delta$ that is based on $\sigma$ and whose vertex is $a$ is {\bf concentrated on $C\in {\mathcal S}$ with respect to $\mathcal S$} provided that 
\[
([a]\star [\sigma])\cap U_{\mathcal S} \subseteq C, \; a\not\in {\rm dom}(D_{\mathcal S}(C)),\hbox{ and }\sigma\not\subseteq {\rm dom}(D_{\mathcal S}(C)). 
\]
We say that a sequence $\overline{\delta}$ of moves is {\bf concentrated on $C$} if each entry of $\overline{\delta}$ is concentrated on $C$. 
Let $\delta$ and $\delta'$ be moves with 
$\delta$ having vertex $a$ and $\delta'$ based on $\sigma'$ with vertex $a'$. We say that $\delta'$ is {\bf free from} $\delta$ if $\delta$ is a weld or 
we have $a\not= a'$ and $a\not\in \sigma'$. 

A {\bf starring} of the system $\mathcal S$ is a collection $(\overline{\delta}_C\colon C\in {\mathcal S})$ of sequences of moves such that for some 
$v_C$, for $C\in {\mathcal S}$, the following properties 
hold for $C, C'\in {\mathcal S}$. 
\begin{enumerate}
\item[---]  $\overline{\delta}_C$ is concentrated on $C$. 
\item[---] $\overline{\delta}_C (C) = [v_C]\star D_{\mathcal S}(C)$.
\item[---] $v_C\not\in {\rm dom}(C)$, and if $C\not= C'$, then $v_C\not= v_{C'}$.
\item[---] If $C\not\subseteq C'$ and $\delta$ and $\delta'$ are entries of $\overline{\delta}_C$ and $\overline{\delta}_{C'}$, respectively, then 
$\delta'$ is free from $\delta$. 
\end{enumerate} 
A {\bf linearization} 
$\overline{\delta}$ of $(\overline{\delta}_C\mid C\in {\mathcal S})$ is a concatenation $\overline{\delta}_{C_1}\cdots \overline{\delta}_{C_m}$ of all sequences in 
the above family such that $C_i\subseteq C_j$ implies $j<i$.  Finally,  we say that a system $\mathcal S$ is {\bf induced} provided that, for each $C\in {\mathcal S}$, if $\sigma \subseteq {\rm dom}(C)$ and 
$\sigma \in U_{\mathcal S}$, then $\sigma \in C$. 
The following proposition is the main result of this section.

\begin{theorem}\label{T:SystemTransformation}
Let $\mathcal S$ be an induced system, let $(\overline{\delta}_C\mid C\in {\mathcal S})$ be a starring of $\mathcal S$, and 
let $\overline{\delta}$ be a linearization of  $(\overline{\delta}_C\mid C\in {\mathcal S})$. Then $v_C\to C$ is an isomorphism from 
$\overline{\delta} (U_{\mathcal S})$ to $\mathcal{C}(P_{\mathcal S})$.
 \end{theorem}

In the proof of the above proposition, it will be convenient, for notational reasons, to shift emphasis from the system $\mathcal S$ to the partial order $P_{\mathcal S}$. 
We will write $P$ for $P_{\mathcal S}$. Then the complexes in $\mathcal S$ are naturally indexed by $P$, namely each complex in $\mathcal S$ is an element of 
$P$ and vice versa. We
will write $(E_p)_{p\in P}$, or simply $(E_p)$, for the system $\mathcal S$. In particular, 
when we say that $(E_p)_{p\in P}$ is a system, we understand that $P$ is a partial order 
with the property that if $p, p'\in P$ have a lower bound in $P$, they have a greatest lower bound, which we denote by $p\wedge p'$, that each $E_p$ is a complex, 
\[
E_p\cap E_{p'} = \begin{cases}
E_{p\wedge p'}, &\text{ if } p\text{ and } p'\text{ have a lower bound in }P;\\
\emptyset, &\text{ otherwise,}
\end{cases}
\]
and the map $P\ni p\to E_p$ is injective.

We will need two auxiliary lemmas below. The proof of the first one is straightforward, and we leave it to the reader.

\begin{lemma}\label{L:eas} 
Let $(E_p)_{p\in P}$ be a system. If $\sigma\in E_p$ and $\sigma \cap (E_q\setminus D(E_q))\not=\emptyset$, then $p\geq q$. 
\end{lemma}

\begin{lemma}\label{L:tog}
Let $(E_p)_{p\in P}$ be a system, let $q\in P$, and let $\delta$ be a move concentrated on $E_q$. 
\begin{enumerate}
\item[(i)] If $p\not\geq q$, then $\delta(E_p) = E_p$ and, for each $p'\in P$, $\delta(E_{p'})\cap E_p = E_{p'\wedge p}$, if 
$p$ and $p'$ have a lower bound in $P$, and $\delta(E_{p'})\cap E_p =\emptyset$, otherwise. 
\end{enumerate}
Assume, additionally, that the system $(E_p)_{p\in P}$ is induced and that if $\delta$ is a weld based on $\sigma$ with vertex $a$ and 
${\rm lk}(\{a\}, E_q) = \partial [\sigma]\star L_q$ for a complex $L_q$, 
then ${\rm lk}(\{a\}, \bigcup_p E_p) = \partial [\sigma]\star L$ for a complex $L$. Then 
\begin{enumerate}
\item[(ii)] $(\delta(E_p))_{p\in P}$ is an induced system; 

\item[(iii)] $\delta(\bigcup_p E_p) = \bigcup_p \delta(E_p)$. 
\end{enumerate}
\end{lemma}

\begin{proof} If $E\subseteq F$ are complexes, we say that $E$ is {\bf induced in} $F$ if for all $\sigma\subseteq {\rm dom}(E)$ with $\sigma\in F$, we 
have $\sigma \in E$. We note the following straightforward consequences of $E$ being induced in $F$ that we will use in the remainder of the proof: 
for complexes $A, B$  and for $\tau\in F$, we have 
\begin{equation}\notag
\begin{split}
A\star B\subseteq F &\Longrightarrow \bigl( E\cap (A\star B) = (E\cap A)\star (E\cap B)\bigr)\\
\tau\in E &\Longrightarrow \bigl( {\rm lk}(\tau, E)= {\rm lk}(\tau, F)\cap E\bigr) \\
\tau\in E &\Longrightarrow \bigl( {\rm st}(\tau, E)= {\rm st}(\tau, F)\cap E\bigr). 
\end{split}
\end{equation}

Let the move $\delta$ be based on $\sigma$ and have vertex $a$.  Set also 
\[
E=\bigcup_p E_p. 
\]

Point (i) follows since $\delta$ is concentrated on $E_q$ and, therefore, for $p\not\geq q$, $\sigma\not\subseteq {\rm dom}(E_p)$, if $\sigma\not=\emptyset$, 
and $a\not\in {\rm dom}(E_p)$.

In proving points (ii) and (iii), we only consider the case $\delta= (\sigma, a)^{-1}$, the case $\delta=(\sigma, a)$ being easier. 
We also assume that 
\begin{equation}\label{E:ase}
[a]\star\partial [\sigma] \subseteq E_q.
\end{equation}
Otherwise, since $\delta$ is concentrated on $E_q$, we would have $[a]\star \partial [\sigma]\not\subseteq E$. This would imply 
that $\delta(E)=E$ and $\delta(E_p)=E_p$ for all $p\in P$, and there would be nothing to prove.

We make the following observation. 
\begin{claim}
Fix $p\geq q$. The following conditions are equivalent. 
\begin{enumerate} 
\item[(a)] ${\rm lk}(\{a\}, E_p) = \partial [\sigma]\star L_p, \hbox{ for some complex }L_p$;
\item[(b)] ${\rm lk}(\{a\}, E) = \partial [\sigma]\star L, \hbox{ for some complex }L$.
 \end{enumerate}
Furthermore, if the two conditions above hold, then 
\begin{enumerate}
\item[(c)] $L_p= L\cap E_p$ and $L = \bigcup_{p\geq q} (L\cap E_p)$.
\end{enumerate} 
\end{claim}

\noindent {\em Proof of Claim.} 
Assume (b). Since $a\in {\rm dom}(E_q)$ and since each $E_p$ is induced in $E$, we have that for all $p\geq q$ 
\begin{equation}\label{E:lkc}
{\rm lk}(\{a\}, E_p) = E_p\cap {\rm lk}(\{a\}, E). 
\end{equation}
Again, since $E_p$ is induced in $E$, we get that for each $p$ 
\[
E_p\cap (\partial [\sigma]\star L) = (E_p\cap \partial [\sigma]) \star (E_p\cap L). 
\]
By \eqref{E:ase}, the above equality gives that, for $p\geq q$, 
\begin{equation}\label{E:stp}
E_p\cap (\partial [\sigma]\star L) = \partial [\sigma]\star (E_p\cap L). 
\end{equation}
Thus, by \eqref{E:lkc} and \eqref{E:stp}, we have 
\begin{equation}\notag
{\rm lk}(\{a\}, E_p) = \partial [\sigma]\star (E_p\cap L), \hbox{ for all }p\geq q, 
\end{equation}
that is (a) holds, and the first equality in (c) is justified as well. 

To see that (a) implies (b), note that, by the above argument, (a) implies that  ${\rm lk}(\{a\}, E_q) = \partial [\sigma]\star L_q$.
This equality gives (b) by the lemma's additional assumption to points (ii) and (iii). 

To finish the proof of the claim, it remains to justify the second equality in (c). Only the inclusion $\subseteq$ requires an argument. For each $\tau\in L$ 
there is $p$ such that $\tau\in E_p$. We only need to check that such a $p$ can be found with $p\geq q$.  
This follows from Lemma~\ref{L:eas} since $a\in {\rm dom}(E_q)$, $a\not\in {\rm dom}(D(E_q))$ and $L\subseteq {\rm lk}(\{a\}, E)$. The claim is proved.

\smallskip

From this point on, we assume that the two equivalent conditions (a) and (b) from the claim hold. Otherwise, by Claim and point (i), 
we would have $\delta(E)=E$ and $\delta(E_p)=E_p$ for all $p\in P$, and there would be nothing to prove. 

To prove point (ii), we first show that $(\delta(E_p))$ is a system. We need to see that 
\begin{equation}\label{E:epep}
 \delta(E_p)\cap \delta(E_{p'}) = \begin{cases}
 \delta(E_{p\wedge p'}), &\text{ if $p, p'$ have a lower bound in $P$};\\
 \emptyset, &\text{ otherwise}. 
 \end{cases}
\end{equation}
If $q\leq p, p'$, then $p, p'$ have a lower bound in $P$ and the conclusion follows from our assumption that point (a) in the claim holds and from 
\[
[\sigma]\star L_{p\wedge p'} = [\sigma] \star (L_p\cap L_{p'}) = ([\sigma]\star L_p)\cap ([\sigma]\star L_{p'}),
\]
where the first equality holds by point (c) of the claim, and the second one by disjointness of $[\sigma]$ with $L_p$ and $L_{p'}$. 
If $q\not\leq p$ or $q\not\leq p'$, then \eqref{E:epep} follows from point (i). We also need to see that the map $P\ni p\to \delta(E_p)$ is injective. 
Let $p_1, p_2\in P$ be distinct. If $p_1\not\geq q$ or $p_2\not\geq q$, then the inequality $\delta(E_{p_1})\not= \delta(E_{p_2})$ follows from (i). So assume 
$p_1, p_2\geq q$. We can suppose that $E_{p_1}$ has a face $\tau$ that is not in $E_{p_2}$. Since $a$ is in both ${\rm dom}(E_{p_1})$ and ${\rm dom}(E_{p_2})$, 
we can further assume that $a\not\in \tau$. Then, by the definition of welds, $\tau\in \delta(E_{p_1})\setminus \delta(E_{p_2})$, so 
$\delta(E_{p_1})\not= \delta(E_{p_2})$, as required. 

It remains to check that for each $p_0\in P$, $\delta(E_{p_0})$ is induced in $\bigcup_p \delta(E_p)$. 
First, we point out the following easy implication: for each $p_0\in P$,
\begin{equation}\label{E:old}
v\in {\rm dom}\bigl(\bigcup_p E_p\bigr)\cap {\rm dom}\bigl(\bigcup_p \delta(E_p)\bigr)\Rightarrow\bigl(v\in {\rm dom}(E_{p_0})\Leftrightarrow v\in {\rm dom}(\delta(E_{p_0}))\bigr). 
\end{equation}
Now, fix $p_0$, and let $\tau\subseteq {\rm dom}(\delta(E_{p_0}))$ and $\tau\in  \bigcup_p \delta(E_p)$. 

Case 1. $\tau\in \bigcup_p E_p$

It follows that, for each $v\in \tau$, we have $\{ v\}\in \bigcup_p E_p\cap \delta(E_{p_0})$, so by \eqref{E:old}, $\tau\subseteq E_{p_0}$. Since $E_{p_0}$ 
is induced in $\bigcup_pE_p$, we get $\tau\in E_{p_0}$. Then, again by \eqref{E:old}, we have $\tau\in \delta(E_{p_0})$, as required. 

Case 2. $\tau\not\in \bigcup_p E_p$

In this case, 
\begin{equation}\label{E:tall}
\tau=\sigma\star \tau',\;\hbox{ for some }\tau'\in L. 
\end{equation}
First we check that 
\begin{equation}\label{E:pql}
p_0\geq q. 
\end{equation}
Since $\tau$ contains $\sigma$, we see that it contains a vertex in ${\rm dom}(E_{q})\setminus {\rm dom}(D(E_{q}))$, 
so Lemma~\ref{L:eas} gives \eqref{E:pql}.
Note now that, for each vertex $v\in \tau'$, we have that $\{ v\}\in \bigcup_p E_p\cap \delta(E_{p_0})$, so by \eqref{E:old}, we get $\tau'\subseteq E_{p_0}$. 
Since $E_{p_0}$ is induced in $\bigcup_pE_p$, we obtain $\tau'\in E_{p_0}$. Therefore, $\tau'\in E_{p_0}\cap L = L_{p_0}$. Thus, by \eqref{E:pql}, 
$\tau\in \delta(E_{p_0})$, as required.

To see point (iii), note first that 
\[
E\cap \delta(E) = \bigcup_p E_p\cap \delta(E_p).
\]
The above equality follows from the definition of welds and from
\[
{\rm st}(\{a\}, E) = \bigcup_{p\geq q} {\rm st}(\{a\}, E_p) = \bigcup_p {\rm st}(\{a\}, E_p),
\]
which is a consequence of $E_p$ being induced in $E$ for each $p$ and $a\in {\rm dom}(E_p)$ precisely when $p\geq q$. 
Thus, to get (iii), in light of (i), it remais to show that each face of $\delta (E)$ that is not in $E$ is in $\delta (E_p)$ for some 
$p$ and that, given $p\geq q$, each face of $\delta (E_p)$ that is not in $E_p$ is in $\delta (E)$. 
The faces of $\delta (E)$ not in $E$ are all in $[\sigma]\star L$
while, for $p\geq q$, the faces of $\delta (E_p)$ not in $E_p$ are all in $[\sigma]\star L_p$. Thus, (iii) follows from (c) in the claim. 
\end{proof}

\begin{lemma}\label{L:frin}
Let $(E_p)_{p\in P}$ be a system, let $q\in P$, and let $\delta$ be a move concentrated on $E_q$ and such that $(\delta(E_p))_{p\in P}$ is an 
induced system. 
\begin{enumerate}
\item[(i)] If $\delta'$ is a move concentrated on $E_q$ with respect to the system $(E_p)_{p\in P}$, then $\delta'$ is concentrated on $\delta(E_q)$ with 
respect to the system $(\delta(E_p))_{p\in P}$. 

\item[(ii)] If $\delta'$ is a move concentrated on $E_{p'}$ with respect to the system $(E_p)_{p\in P}$ and $\delta'$ is free from $\delta$, then 
$\delta'$ is concentrated on $\delta(E_{p'})$ with respect to the system $(\delta(E_p))_{p\in P}$.
\end{enumerate}
\end{lemma}

\begin{proof} Let $\delta$ be based on $\sigma$ and have vertex $a$, and let $\delta'$ be based on $\sigma'$ and have vertex $a'$. 
We point out again that, for each $p_0\in P$, 
\begin{equation}\label{E:oldn}
v\in {\rm dom}\bigl(\bigcup_p E_p\bigr)\cap {\rm dom}\bigl(\bigcup_p \delta(E_p)\bigr)\Rightarrow\bigl(v\in {\rm dom}(E_{p_0})\Leftrightarrow v\in {\rm dom}(\delta(E_{p_0}))\bigr). 
\end{equation}

(i) We assume that 
\begin{equation}\label{E:hgf}
([a]\star [\sigma])\cap \bigcup_p E_p \subseteq E_q, \; a\not\in {\rm dom}(D(E_q)),\hbox{ and }\sigma \not\subseteq {\rm dom}(D(E_q)) 
\end{equation}
and 
\begin{equation}\label{E:tru}
([a']\star [\sigma'])\cap \bigcup_p E_p \subseteq E_q, \; a'\not\in {\rm dom}(D(E_q)),\hbox{ and }\sigma'\not\subseteq {\rm dom}(D(E_q)),
\end{equation}
and need to prove 
\[
([a']\star [\sigma'])\cap \bigcup_p \delta(E_p)  \subseteq \delta(E_q), \; a'\not\in {\rm dom}(D(\delta(E_q))),\hbox{ and }\sigma'\not\subseteq {\rm dom}(D(\delta(E_q))). 
\]
It follows from $\sigma \not\subseteq {\rm dom}(D(E_q))$ that 
\[
{\rm dom}(D(\delta(E_q)))\subseteq {\rm dom}(D(E_q)),
\]
which gives 
\[
 a'\not\in {\rm dom}(D(\delta(E_q)))\hbox{ and }\sigma'\not\subseteq {\rm dom}(D(\delta(E_q)))
\]
directly from the assumptions. It remains to prove 
\[
([a']\star [\sigma'])\cap \bigcup_p \delta(E_p)  \subseteq \delta(E_q). 
\]
Since the system $(\delta(E_p))_p$ is assumed to be induced, it suffices to show that 
\begin{equation}\label{E:aprs}
(\{ a'\} \cup \sigma') \cap  {\rm dom}(\bigcup_p \delta(E_p)) \subseteq {\rm dom}(\delta(E_q)). 
\end{equation}
Let $v$ be a vertex that is an element of the set on the left hand side of the inclusion \eqref{E:aprs}. If $v\in {\rm dom}(\bigcup_p E_p)$, 
then $v\in {\rm dom}(\delta(E_q))$ by \eqref{E:oldn} and the first inclusion in \eqref{E:tru}. 
So assume that $v\not\in  {\rm dom}(\bigcup_p E_p)$. It follows that $\delta$ is a division, $v= a$, and $\sigma\in \bigcup_pE_p$. 
This last condition implies by the first inclusion in \eqref{E:hgf} that $\sigma\in E_q$. It follows that $v=a\in {\rm dom}(\delta(E_q))$, as required by 
\eqref{E:aprs}.

(ii) We assume that 
\begin{equation}\label{E:tru2}
([a']\star [\sigma'])\cap \bigcup_p E_p \subseteq E_{p'}, \; a'\not\in {\rm dom}(D(E_{p'})),\hbox{ and }\sigma'\not\subseteq {\rm dom}(D(E_{p'})),
\end{equation}
and need to prove 
\[
([a']\star [\sigma'])\cap \bigcup_p \delta(E_p)  \subseteq \delta(E_{p'}), \; a'\not\in {\rm dom}(D(\delta(E_{p'}))),\hbox{ and }\sigma'\not\subseteq {\rm dom}(D(\delta(E_{p'}))). 
\]
But using the assumption that $\delta(E_{p'})$ is induced in $ \bigcup_p \delta(E_p)$, it suffices to show 
\begin{equation}\label{E:aprs2}
\begin{split}
&(\{ a'\} \cup \sigma') \cap  {\rm dom}(\bigcup_p \delta(E_p)) \subseteq {\rm dom}(\delta(E_{p'}))\; \hbox{ and }\\
&a'\not\in {\rm dom}(D(\delta(E_{p'}))),\; \sigma'\not\subseteq {\rm dom}(D(\delta(E_{p'}))). 
\end{split}
\end{equation}
Since $\delta'$ is free from $\delta$, $\delta$ is a weld or we have $a\not= a'$ and $a\not\in \sigma'$. It follows that 
\[
(\{ a'\} \cup \sigma') \cap  {\rm dom}(\bigcup_p \delta(E_p)) \subseteq {\rm dom}(\bigcup_p E_p).
\]
So if $v$ belongs to the set on the left hand side of the above inclusion, then, by \eqref{E:oldn}, for each $p_0\in P$, 
$v\in {\rm dom}(E_{p_0})$ if and only if $v\in {\rm dom}(\delta(E_{p_0}))$. Thus, \eqref{E:aprs2} follows from \eqref{E:tru2}. 
\end{proof}

We introduce a technical notion that will be helpful in the proof of Theorem~\ref{T:SystemTransformation}. Let 
$U\subseteq P$ be upward closed. We say that a system $(E_p)_{p\in P}$ is {\bf good for} $U$ if, for some $v_q$ with $q\in U$, we have, for each $p\in P$, 
\begin{equation}\label{E:goodness}
E_p = \bigcup_{(q_1, \dots , q_k)} \bigl( [v_{q_1}]\star \cdots \star [v_{q_k}]\star \bigcup_{r\leq q_k, r\not\in U} E_r\bigr), 
\end{equation}
where $(q_1, \dots , q_k)$ run over the set of all sequences of elements of $P$ such that 
\begin{equation}\label{E:auxg}
q_k\leq \cdots \leq q_1\leq p\;\hbox{ and }\; q_1, \dots, q_k\in U,
\end{equation}
and where we assume that $v_q$, $q\in U$, are pairwise distinct and $v_q\not\in {\rm dom}(E_r)$ for each $q\in U$ and $r\in P\setminus U$. Note that 
formula eqref{E:goodness} determines the whole system $(E_p)_{p\in P}$ from the complexes $E_r$ with $r\in P\setminus U$. 

We register the following basic lemma. 

\begin{lemma}\label{L:vgoo}
Let $(E_p)_{p\in P}$ be a system. 
\begin{enumerate}
\item[(i)] $(E_p)_{p\in P}$ is good for $U=\emptyset$. 

\item[(ii)] If $(E_p)_{p\in P}$ is good for $U= P$, then $v_p\to p$ induces an isomorphism between $\bigcup_{p\in P} E_p$ and ${\mathcal C}(P)$. 
\end{enumerate}
\end{lemma}

\begin{proof}
If $p\not\in U$, then only the empty sequence satisfies \eqref{E:auxg} 
for $(q_1, \dots, q_k)$, so \eqref{E:goodness} boils down to 
\[
E_p=\bigcup_{r\leq p} E_r,
\]
which is true of an arbitrary system, so (i) follows. On the other hand, being good for $U= P$ means 
that, for each $p\in P$, 
\[
E_p=  \bigcup_{(q_1, \dots , q_k)} [v_{q_1}]\star \cdots \star [v_{q_k}], 
\]
where $(q_1, \dots , q_k)$ run over the set of all sequences of elements of $P$ such that 
$q_k\leq \cdots \leq q_1\leq p$, which implies (ii). 
\end{proof}

\begin{lemma}\label{L:goodu}
Assume that $(E_p)_{p\in P}$ is an induced system that is good for $U\subseteq P$. Let $\overline{\delta}$  be a sequence of moves concentrated 
on $E_{p_0}$, for some  $p_0\in P\setminus U$ that is maximal in $P\setminus U$. Then $(\overline{\delta}(E_p))_{p\in P}$ is an induced system 
that is good for $U$. 
\end{lemma}

\begin{proof} The lemma is proved by induction on the length of the sequence $\overline{\delta}$. Note that, by Lemma~\ref{L:frin}(i), if $\delta_1$ and $\delta_2$ 
are moves concentrated on $E_{p_0}$ in the system $(E_p)_{p\in P}$ and $(\delta_1(E_p))_{p\in P}$ is a system, 
then $\delta_2$ is concentrated on $\delta_1(E_{p_0})$ in the system $(\delta_1(E_p))_{p\in P}$. It follows from this observation that it suffices to show the lemma  
in the case when $\overline{\delta}$ consists of a single move $\delta$. There are two cases; $\delta$ is a division or $\delta$ is a weld. The first case is easy and 
we leave it to the reader. 

Assume $\delta$ is a weld based on $\sigma$ with vertex $a$ that is concentrated on $E_{p_0}$. 
For convenience of notation, we expand the partial order $P$ to a partial order $\widehat{P}$ by adding 
two new elements $0, \infty$ to $P$ and declaring that $0\leq p\leq \infty$ for each $p\in P$, and defining $E_0= \emptyset$. 
Fix $p\in \widehat{P}\setminus \{ 0\}$, and consider the family of complexes 
\[
{\mathcal S}(p) = \{ [v_{q_1}]\star \cdots \star [v_{q_\ell}]\star E_r\mid r\leq q_\ell\leq \cdots \leq q_1\leq p,\; r\not\in U,\; q_1, \dots, q_\ell\in U\}.
\]
The family ${\mathcal S}(p)$ 
is an induced system, that is, it is closed under non-empty intersections and each complex in the family is closed under unions within $U_{{\mathcal S}(p)}$. 
We leave it to the reader to check these properties. Since the system $(E_p)_{p\in P}$ is good for $U$, we have
\begin{equation}\label{E:fiin}
U_{{\mathcal S}(p)}=E_p, \hbox{ for } p\in P,\; \hbox{ and } U_{{\mathcal S}(\infty)}=\bigcup_{p\in P} E_p.
\end{equation}

Note that if $\widehat{P}\ni p\geq p_0$, then $E_{p_0}$ is an element of ${\mathcal S}(p)$, in which case, if 
\[
{\rm lk}(\{a\}, E_{p_0}) = \partial [\sigma] \star L_{p_0},
\]
for a complex $L_{p_0}$, then, by maximality of $p_0$ in $P\setminus U$, 
\[
{\rm lk}(\{a\}, U_{{\mathcal S}(p)}) = \partial [\sigma] \star \bigl( L_{p_0} \star \bigcup_{(q_1, \dots, q_k)} [v_{q_1}]\star \cdots \star [v_{q_k}]\bigr), 
\]
where $(q_1, \dots , q_k)$ run over all sequences of elements of $P$ such that 
\[
p_0< q_k\leq\cdots \leq q_1\leq p,\; q_1, \dots, q_k\in U. 
\]
It follows that in the above situation $\delta$ fulfills the assumption of points (ii) and (iii) of Lemma~\ref{L:tog} with respect to the system $(E_p)_{p\in P}$ 
(by the second equality in \eqref{E:fiin}) and with respect to the system ${\mathcal S}(p)$ for each $p\in P$. 
Therefore, from Lemma~\ref{L:tog}(ii) applied to the system $(E_p)_{p\in P}$, 
we get that $(\delta(E_p))_{p\in P}$ is an induced system. 
By Lemma~\ref{L:tog}(iii) applied to ${\mathcal S}(p)$ with $p\in P$, using the first equality in \eqref{E:fiin}, we have 
\begin{equation}\notag
\delta \bigl(E_p \bigr)
= \bigcup_{(q_1, \dots , q_\ell, r)} \delta\bigl([v_{q_1}]\star \cdots \star [v_{q_\ell}]\star E_r\bigr). 
\end{equation}
where $(q_1, \dots , q_\ell, r)$ run over the set of all sequences of elements of $P$ such that 
\begin{equation}\label{E:spec}
r\leq q_\ell\leq \cdots \leq q_1\leq p,\; r\not\in U,\; q_1, \dots, q_\ell\in U. 
\end{equation}
Thus, to see that the system $(E_p)$ is good for $U$, it suffices to check that 
\begin{equation}\label{E:gfd}
\delta\bigl([v_{q_1}]\star \cdots \star [v_{q_\ell}]\star E_r\bigr) =  [v_{q_1}]\star \cdots \star [v_{q_\ell}]\star \delta(E_r),
\end{equation}
for $(q_1, \dots , q_\ell, r)$ as in \eqref{E:spec}. Since $\delta$ is assumed to concentrate on $E_{p_0}$ and $p_0\in P\setminus U$, 
we see that $a\not= v_q$ and $v_q\not\in \sigma$, for all $q\in U$. Thus, 
\eqref{E:gfd} follows. 
\end{proof}

\begin{proof}[Proof of Theorem~\ref{T:SystemTransformation}]   
Let $\overline{\delta}(p_1)\cdots \overline{\delta}(p_m)$ be the linearization $\overline{\delta}$ of $(\overline{\delta}(p)\colon p\in P)$.
For $k\leq m$, let 
\[
U_k=\{ p_i\colon i\leq k\}. 
\]
Note that $U_k$ is upward closed.

By induction on $k\leq m$, we will prove the following statement: 
\[
\bigl( (\overline{\delta}(p_1)\cdots \overline{\delta}(p_k)\bigr) (C_p))_{p\in P}
\]
is an induced system that is good for $U_k$ and, for $p\not\in U_k$, 
\[
\bigl((\overline{\delta}(p_1)\cdots \overline{\delta}(p_k)\bigr)(C_p) = C_p. 
\]

Note that when $k=m$, then $U_k=P$; therefore, by Lemma~\ref{L:vgoo}(ii), the statement for $k=m$ yields the proposition. 
So it remains to prove the above statement. Set 
\[
D^k_p = (\overline{\delta}(p_1)\cdots \overline{\delta}(p_k)\bigr) (C_p). 
\]

If $k=0$, then $U_k=\emptyset$, and there is nothing to prove by Lemma~\ref{L:vgoo}(i). Now, we assume the statement holds for $k<m$, and we prove it for $k+1$. 
Since $p_{k+1}\in P\setminus U_k$ is maximal in $P\setminus U_k$, it follows from Lemma~\ref{L:goodu} (with a use of Lemma~\ref{L:frin}(ii)) and our inductive assumption 
that $(D^{k+1}_p)$ is an induced system that is good for $U_k$. 
This last condition means that, for each $p\in P$,  
\begin{equation}\label{E:rec}
D^{k+1}_p 
= \bigcup_{(q_1, \dots , q_\ell, r)} [v_{q_1}]\star \cdots \star [v_{q_\ell}]\star D^{k+1}_r, 
\end{equation}
where $(q_1, \dots , q_\ell, r)$ run over the set of all sequences of elements of $P$ such that 
\[
r\leq q_\ell\leq \cdots \leq q_1\leq p,\; r\not\in U_k,\; q_1, \dots, q_\ell\in U_k. 
\]
Moreover, in \eqref{E:rec}, since $r\not\in U_k$, by inductive assumption, we have 
\[
D^{k+1}_r = \overline{\delta}(p_{k+1})(C_r).
\]
Furthermore, by maximality of $p_{k+1}$ in $P\setminus U_k$, we have $r\not\geq p_{k+1}$ or $r=p_{k+1}$. In the first case, by 
Lemma~\ref{L:tog}(i), 
\[
D^{k+1}_r  = \overline{\delta}(p_{k+1})(C_r) = C_r,
\]
while, in the second case $r=p_{k+1}$, by our assumption on $\overline{\delta}(p_{k+1})$, 
\[
D^{k+1}_{r} = \overline{\delta}(p_{k+1})(C_{p_{k+1}}) = [ v_{p_{k+1}}] \star D(C_{p_{k+1}}) = [ v_{p_{k+1}}] \star  \bigcup_{q<p_{k+1}} C_q.
\]
It follows that the system $(D^{k+1}_p)$ is good for $U_{k+1}$ and $D^{k+1}_r = C_r$ if $r\not\in U_{k+1}$. The proposition is proved. 
\end{proof}

\subsection{Starrings of cell-systems}\label{S:Cell-systems}

Here we introduce the notion of a cell-system as a special type of a system. Cell-systems and Proposition \ref{P:TameCellSystemAdmitsStarring}, 
which we prove below will be important in Section~\ref{S:CategoriesOfMaps}.

\begin{definition} \label{Def:CellSystem}
A {\bf cell-system} is a system $\mathcal{S}$ so that every complex $C$ in $\mathcal{S}$ is a stellar ball with $D_{\mathcal S}(C)=\partial C$. If each $C$ is additionally a tame ball then we say that $\mathcal{S}$ is a {\bf tame cell-system}. 
\end{definition}

As in Section \ref{S:Systems}, we associate to each cell-system the poset $P_{\mathcal{S}}=(\mathcal{S},\subseteq)$ and 
the simplicial complex $\mathcal{C}(P_{\mathcal{S}})$ of all finite, non-empty, linearly ordered subsets of $P_{\mathcal{S}}$.

\begin{proposition}\label{P:TameCellSystemAdmitsStarring}
Every tame cell-system $\mathcal{S}$ is induced and admits a starring. Moreover, if $\overline{\delta}$ is the linearization of such a starring and 
$\delta\overline{\varepsilon}$ is an initial segment of  $\overline{\delta}$, then there is some $C\in\mathcal{S}$ so that $\delta$ is strongly internal for 
$\overline{\varepsilon}C$.
\end{proposition}
\begin{proof}
To see that $\mathcal{S}$ is induced notice that  if $\sigma\in U_{\mathcal{S}}$, then $\sigma\in D$ for some $D\in\mathcal{S}$. But if $\sigma\subseteq \mathrm{dom}(C)$ and $C\neq D$, then $\sigma\subseteq \mathrm{dom}(C\cap D)\subseteq \mathrm{dom}(\partial D)$, contradicting that $D$ is a tame ball.

By Theorem \ref{T: star with strongly internal}, there is a strongly internal starring  $\overline{\delta}_C$ for every $C\in \mathcal{S}$. We may assume without loss of generality that if $\delta$ is a subdivision from  $\overline{\delta}_C$ with vertex $a$, $\delta'$ is a subdivision from  $\overline{\delta}_{C'}$ with vertex $a'$, and $C\neq C'$, then
\begin{equation}\label{E:starring}
a\not\in \mathrm{dom}(U_{\mathcal{S}}) \text{ and } a'\neq a.
\end{equation}
To see this, let $\delta$ be a subdivision on vertex $a$ with $\overline{\delta}_C=\overline{\varepsilon}\delta \overline{\zeta}$, and let $v\not\in\mathrm{dom}(U_{\mathcal{S}})$ be any new element that is not the vertex of any subdivision of  $\overline{\delta}_{C'}$, for any $C'\in\mathcal{S}$. We may now modify $\overline{\delta}_C$ by renaming all occurrences of $a$ in $\delta \overline{\zeta}$ to $v$. 
For every fixed $C\in\mathcal{S}$ we apply this procedure to every  subdivision in  $\overline{\delta}_C$  starting from the beginning  of $\overline{\delta}_C$  and moving in linear order to the end.

\begin{claim}
$(\overline{\delta}_C: C\in\mathcal{S})$
is a starring of $\mathcal{S}$
\end{claim}
\begin{proof}[Proof of Claim.]
First we show that every entry  $\delta$ of $\overline{\delta}_C$ is concentrated on $C$ with respect to $\mathcal{S}$. Let $\overline{\delta}_C=\overline{\varepsilon}\delta \overline{\zeta}$ and assume that $\delta$ is based on $\sigma$ and has $a$ as a vertex. Since $\overline{\delta}_C$ is a strongly internal starring of $C$ it immediately follows that   $a\not\in {\rm dom}(D_{\mathcal S}(C))$ and  $\sigma\not\subseteq {\rm dom}(D_{\mathcal S}(C))$. Let now $\rho\in([a]\star[\sigma])\cap U_{\mathcal{S}}$. Since $\delta$ is essential for $\overline{\varepsilon}C$, we have that $\rho\subseteq \mathrm{dom}(\overline{\varepsilon}C)$ or $\rho\subseteq \mathrm{dom}(\delta\overline{\varepsilon}C)$, depending on whether $\delta$ is a subdivision or a weld. Either way, $\rho\subseteq \mathrm{dom}(C)\bigcup V$, where $V$ is the collection of all vertexes which are introduced by any subdivision in $\overline{\varepsilon}\delta$.  By the first part of (\ref{E:starring})  we have that $V\cap \mathrm{dom}(U_{\mathcal{S}})=\emptyset$ and therefore  $\rho\subseteq \mathrm{dom}(C)$. But this implies $\rho \in C$ since otherwise we have that $\rho\in D$ for some $D\in\mathcal{S}$ with $C\neq D$, and therefore $\rho\subseteq \mathrm{dom}(C\cap D)\subseteq \mathrm{dom}(\partial D)$, contradicting that $D$ is a tame ball.

It is clear that  $\overline{\delta}_C C =[v_C]\star\partial C= [v_C]\star D_{\mathcal{S}(C)}$ for some $v_C$. By  (\ref{E:starring}) we have that $v_C\not\in {\rm dom}(C)$ and $v_C\not= v_{C'}$ when   $C\not= C'$. We are left to show that whenever $C\not\subseteq C'$, any two entries $\delta$ and $\delta'$  of $\overline{\delta}_C$ and $\overline{\delta}_{C'}$, respectively, are free.  This follows from (\ref{E:starring})  in the only non-trivial case, i.e., the case when $\delta=(\sigma,a)$ is a subdivision.
\end{proof}

Let finally $\overline{\delta}$  be a linearization of  $(\overline{\delta}_C: C\in\mathcal{S})$. A simple induction on the length of $\overline{\delta}$,  based on the assumption that $\overline{\delta}_C$ is strongly internal starring of $C$ and  Lemma \ref{L:tog}(ii),(iii), proves the second part of the Proposition.
\end{proof}

As a consequence we have the following corollary for cell-systems which are not necessarily tame. Let $\mathcal{S}$ be any cell-system and let $\overline{\delta}=\delta_1,\ldots,\delta_k$ be a sequence of stellar moves. We say that {\bf $\overline{\delta}$ factors through $\mathcal{S}$}, if for all $l\leq k$ we have that $\delta_l\cdots\delta_1\mathcal{S}:=\{\delta_l\cdots\delta_1 C\mid C\in \mathcal{S}\}$ is still a system.

\begin{corollary}\label{C: stellar-equivalent to Poset}
Let $\mathcal{S}$ be a cell-system. Then there is a sequence $\overline{\delta}$ of stellar moves  which factor through $\mathcal{S}$, with $\overline{\delta}\, U_{\mathcal{S}}= \mathcal{C}(P_{\mathcal{S}})$.  
\end{corollary}
\begin{proof}
Subdividing all faces of $U_{\mathcal{S}}$ in a never $\subseteq$-increasing order produces a sequence $\overline{\varepsilon}$  that has the property that $\overline{\varepsilon}\, U_{\mathcal{S}}=\beta U_{\mathcal{S}}$ and $\overline{\varepsilon}\, C=\beta C$, for all $C\in\mathcal{S}$.   Notice now that $\beta \mathcal{S}:= \{\beta C \mid C\in\mathcal{S}\}$ is a tame cell system on $\overline{\varepsilon}\, C$.  By Proposition \ref{P:TameCellSystemAdmitsStarring} we have a starring $\overline{\zeta}$ of $\beta \mathcal{S}$. Set $\overline{\delta}=\overline{\varepsilon}\overline{\zeta}$ and notice that $\overline{\delta}$ factors through $\mathcal{S}$ since all welds involved in $\overline{\zeta}$
satisfy  the second part of Lemma \ref{L:tog}. The rest follows from  the observation that the posets $P_{\mathcal{S}}$ and $P_{\beta\mathcal{S}}$  are isomorphic.
\end{proof}

\section{Categories of maps and computation of the domination closure}\label{S:CategoriesOfMaps}

\subsection{Definitions of relevant categories}\label{Su:defrel}

We defined the generic combinatorial simplex $\DD$ as the  \Fraisse{} limit of the projective \Fraisse{} category $\mathcal{S}(\Delta)$ of selection maps.  In this section we define categories of maps associated to $\DD$ that are broader than $\mathcal{S}(\Delta)$.

A {\bf stellar $n$-simplex} is a stellar  $n$-ball $A$ together with a family $(A_{\mathrm{X}}\colon \mathrm{X}\in \Delta)$ indexed by the faces of the $n$-simplex $\Delta$, so that $A=A_{\{0,\ldots,n\}}$, and  for every $\mathrm{X},\mathrm{Y}\in \Delta$ we have that:
\begin{itemize}
\item[---] $A_{\mathrm{X}}$ is a stellar ball of the same dimension as $[\mathrm{X}]$;
\item[---] $A_{\mathrm{X}\cap\mathrm{Y}}=A_{\mathrm{X}}\cap A_{\mathrm{Y}}$, where we set $A_\emptyset = \emptyset$;
\item[---] $\partial A_{\mathrm{X}}=\bigcup_{\mathrm{Z}\subsetneq\mathrm{X}} A_{\mathrm{Z}}$.
\end{itemize}
We will often refer to such a stellar $n$-simplex by its maximum complex $A$. Notice that for every $k\geq 0$ we can identify the complex $\beta^k\Delta$ with the stellar $n$-simplex $(\beta^k[\mathrm{X}]\colon \mathrm{X}\in \Delta)$. These identifications will be implicit throughout the rest of this paper.
Let $A,B$ be stellar $n$-simplexes. A {\bf face-preserving}  map from $B$ to $A$ is a simplicial map $f\colon B\to A$ so that $f(B_{\mathrm{X}})=A_{\mathrm{X}}$, for all $\mathrm{X}\in \Delta$. We will denote by $f_{\mathrm{X}}$ the restriction of $f$ on $B_{\mathrm{X}}$. 

By $\mathcal{R}_{\star}(\Delta)$ we denote the collection of all face-preserving maps between stellar $n$-simplexes. 
We will consider three subcategories of $\mathcal{R}_{\star}(\Delta)$:  
the category $\mathcal{S}_{\star}(\Delta)$  of selections on stellar simplexes; 
the category $\mathcal{C}_{\star}(\Delta)$ of hereditarily cellular maps on stellar simplexes; 
and the category $\mathcal{H}_{\star}(\Delta)$ of restricted near-homeomorphisms on stellar simplexes.  
By restricting these categories to stellar simplexes of the form $\beta^k\Delta$, we get the full subcategories: 
$\mathcal{S}(\Delta)$, $\mathcal{C}(\Delta)$, $\mathcal{H}(\Delta)$, and 
$\mathcal{R}(\Delta)$. 

\subsubsection{The category $\mathcal{S}_{\star}(\Delta)$  of selection maps on stellar simplexes} 
For every stellar $n$-simplex  and  $k\geq0$ we identify $\beta^{k} A$ with the stellar $n$-simplex $(\beta^k A_{\mathrm{X}}: \mathrm{X}\in \Delta)$. Since all maps in  the category $\mathcal{S}(A)$  are face-preserving we can view $\mathcal{S}(A)$  as a subcategory of $\mathcal{R}_{\star}(\Delta)$. We define the category $\mathcal{S}_{\star}(\Delta)$ of {\bf selections on stellar $n$-simplexes} to be the union of all categories of the form $\mathcal{S}(A)$, where $A$ ranges over all stellar $n$-simplexes.

\subsubsection{The category $\mathcal{C}_{\star}(\Delta)$ of hereditarily cellular maps} 
We will need the following notions from \cite{Co1}. Let $A$ be a complex. For every $\sigma\in A$ consider the subcomplex 
\[
D(\sigma,A):=\big\{ \{\sigma_0\subseteq\ldots\subseteq\sigma_k\}\colon  \sigma\subseteq \sigma_0\big\}
\]
of $\beta A$. This complex is called the {\bf dual to $\sigma$ in $A$}. Similarly,  if $f\colon B\to A$ is a simplicial map, the subcomplex 
\[
D(\sigma,f):=(\beta f)^{-1} D(\sigma,A)=\big\{ \{\tau_0\subseteq\ldots\subseteq\tau_k\}\colon  \sigma\subseteq f_{*}(\tau_0)\big\} 
\]
of $\beta B$ is called the {\bf dual to $\sigma$ with respect to $f$}. It is a theorem of Cohen \cite{Co1} that whenever $B$ is  stellar manifold, so is $D(\sigma,f)$; see Theorem~\ref{Cohen's original theorem}. 
A simplicial map $f\colon M\to N$ between stellar manifolds is called  {\bf cellular} if for every face $\sigma$ of $N$, $D(\sigma,f)$ is a stellar ball.
The map $f$ is {\bf transversely cellular} if both $f$ and $f\res \partial M$ are cellular. Transversely cellular maps were first considerer in \cite{Co1}.  
Here we will consider the  category $\mathrm{C}_{\star}(\Delta)$ of the {\bf hereditarily cellular maps}. These are  all face-preserving maps  $f\colon B\to A$ between the stellar $n$-simplexes with the property that for every $\mathrm{X}\in \Delta$, the map $f_{\mathrm{X}}\colon B_{\mathrm{X}}\to A_{\mathrm{X}}$ is cellular. 

\subsubsection{The category $\mathcal{H}_{\star}(\Delta)$ of restricted near-homeomorphisms}\label{SS: Near-homeos}

While $\mathcal{C}_{\star}(\Delta)$ captures the piecewise linear structure on $|\Delta|$, the topological manifold structure is reflected by 
the category $\mathcal{H}_{\star}(\Delta)$, which we now define. 
Let $A$, $B$ be stellar $n$-simplexes. 
The map $\varphi\colon B\to A$ is called a {\bf restricted near-homeomorphism} if it is the limit of a uniformly convergent sequence $(\varphi_p)$ of 
functions $\varphi_p \colon |B| \to |A|$ each of which is a homeomorphism that maps $|B_X|$ to $|A_X|$ for each $X\in \Delta$.
A face-preserving map $f\colon  B\to A$ is a {\bf restricted near-homeomorphism} if the geometric realization $|f|\colon |B| \to |A|$ 
of $f$ is a restricted near-homeomorphism. We collect all these maps in the category $\mathcal{H}_{\star}(\Delta)$.

\subsection{Survey of results from literature}\label{S:SurveyLiterature}
Here, we collect  some results from the literature needed in our proofs of Theorems~\ref{Theorem:CellularAndSelections} 
and \ref{Theorem:HeredNearHomeoand Selection}.

The following theorem provides a useful criterion for showing when a stellar $n$-manifold is a stellar $n$-ball. Let $C$ be a simplicial complex and let $\sigma\subsetneq \tau$ be two faces of $C$. We say that $\sigma$ is a {\bf free face} of $\tau$ if t$\tau$ is there is no other $\rho\in C$ with $\sigma\subsetneq \rho$. In that case $C':= C\setminus \{\sigma,\tau\}$ is also a complex and we say that $C$ {\bf elementary collapses to} $C'$. 
We say that $C$ {\bf collapses to $D$} if there is a sequence $C_0,\ldots, C_{n-1}$ of complexes so that: $C_0=C$, $C_{n-1}=D$,  and $C_i$ elementary collapses to $C_{i+1}$, for all $i$. If in the above we have that $D=[v]$, for some $v\in \mathrm{dom}(C)$, then we say that $C$ is {\bf collapsible}.

\begin{theorem}[Whitehead \cite{Wh}]\label{T:Collapsible}
A collapsible stellar manifold  is a stellar ball.
\end{theorem}

One useful fact relating collapsibility with the dual structure in the sense of Cohen is the following proposition, which can be proved by a simple induction.

\begin{proposition}[\cite{Co1}, Proposition 5.5]\label{P:CollapseRegularNeigh}
Let $f\colon B\to A$ be a simplicial map, and let $\sigma\in A$. Then $D(\sigma,f)$ collapses to the subcomplex $(\beta f)^{-1}(\sigma)$ of $\beta B$.
\end{proposition}

The first part of the following theorem is Theorem 10.1 from \cite{Co1}. The second part combines Corollaries IV.12 and V.3 from  \cite{Akin}.

\begin{theorem}\label{Theorem:Akin-Cohen}
Let $M, L, K$ be stellar manifolds, let $f\colon M\to L$ be a transversely cellular map, and let $g\colon L\to K$ be a simplicial map.
\begin{enumerate}
\item[(i)] (Cohen) If $f'\colon M'\to L'$ is a simplicial map between    stellar manifolds so that after taking the geometric realizations we have $|M|_{\mathbb{R}}=|M'|_{\mathbb{R}}$, $|N|_{\mathbb{R}}=|N'|_{\mathbb{R}}$, and $|f|_{\mathbb{R}}=|f'|_{\mathbb{R}}$ then $f'$  is also transversely cellular. 
\item[(ii)] (Akin) $g$ is transversely cellular if and only if $g\circ f$ is.
\end{enumerate} 
\end{theorem}

We will need a version of the following theorem due to Cohen. 

\begin{theorem}[Cohen \cite{Co1}]\label{Cohen's original theorem}
Let $M$ be a stellar manifold of dimension $n$ and let $f\colon M\to A$ be a simplicial map. Let also $\sigma$ be in the image of $f$. Then 
\begin{enumerate}
\item[(i)] $D(\sigma,f)$ is a combinatorial manifold of dimension $n-\mathrm{dim}(\sigma)$;
\item[(ii)] $\partial D(\sigma,f)$ is the union of $D(\sigma,f\res \partial M)$ with $\big( \bigcup_{\sigma\subset\tau}D(\sigma,f) \big)$. 
\end{enumerate}
\end{theorem}

In our arguments concerning near-homeomorphisms, we will make use of the following theorem, which is a refined version of a theorem of 
Brown \cite{Br}; see also \cite[Theorem 1.10.23]{vM}. The proof of this refinement is implicit in \cite{Br}. 

\begin{theorem}\label{T:Brown}
Let $Y_n$ be compact metrizable spaces for $n\in {\mathbb N}$, and let $Y^i_n\subseteq Y_n$, for $i=0, \dots , m$, be closed. Let 
$f_n\colon Y_{n+1}\to Y_n$ be uniform approximable by homeomorphisms $h\colon Y_{n+1}\to Y_n$ such that 
$h(Y^i_{n+1})=Y^i_n$ for $i\leq m$. Consider 
\[
Y_\infty =\varprojlim (Y_n,f_n)\;\hbox{ and }\; Y^i_\infty =\varprojlim (Y^i_n, f\res Y^i_n), \hbox{ for }i\leq m.
\]
Then the projection map $f^\infty_0\colon Y_\infty\to Y_0$ is uniformly approximable by homeomorphisms 
$h\colon Y_\infty\to Y_0$ such that $h(Y^i_\infty)= Y^i_0$ for $i\leq m$. 
\end{theorem}

\begin{proof} We indicate how to derive the theorem from what is proved in \cite{Br}. 
For simplicity of notation we set $m=0$, that is, there is only one sequence of closed set in our assumptions---$Y^0_n$, for $n\in {\mathbb N}$. 
For each $n\in {\mathbb N}$, fix a metric $d_n$ on $Y_n$ compatible with the topology on $Y_n$. 
By \cite[Theorems 1 and 2]{Br} and the proof of \cite[Theorem 3]{Br}, the following statement holds. 

\noindent {\em Statement}. If $K_n$ is a family of continuous functions from $Y_{n+1}$ to $Y_n$ 
such that $f_n$ is uniformly approximable by functions in $K_n$, then for each 
sequence of real numbers $\delta_n>0$, $n\in {\mathbb N}$, there are $h_n\in K_n$ with 
\[
\sup_{x\in Y_{n+1}} d_n(h_n(x), f_n(x)) <\delta_n, 
\]
and such that, for each $N$, the sequence of functions from $Y_\infty$ to $Y_N$
\begin{equation}\label{E:FNs}
\bigl((h_N\circ \cdots \circ h_n)\circ f^\infty_{n}\bigr)_{n\geq N}
\end{equation}
converges uniformly, and if we let 
\begin{equation}\label{E:FN}
F_N(x) = \lim_{n\geq N} (h_N\circ \cdots \circ h_n)\circ f^\infty_{n}(x),
\end{equation}
then the function 
\begin{equation}\label{E:prfu}
Y_\infty\ni x \to (F_0(x), F_1(x), F_2(x), \dots )\in  \varprojlim_n (Y_n, h_n),  
\end{equation}
is a homeomorphism. (The uniformity of the convergence of \eqref{E:FNs} is established in \cite[formula (1.3)]{Br}.) 

Note further the following three easy to justify points. 

(1) If each $h_n$ is a homeomorphism, then all the projection maps in $\varprojlim_n (Y_n, h_n)$ are homeomorphisms, so 
the function \eqref{E:prfu} being a homeomorphism implies that $F_0$ is a homeomorphism. 

(2) If each $h_n$ is such that 
$h_n(Y^0_{n+1})= Y^0_n$, then, by its definition \eqref{E:FN}, $F_0$ maps $Y_\infty^0$ to $Y^0_0$. 

(3) From the definition \eqref{E:FN} of $F_0$, if each $h_n$ uniformly approximates 
$f_{n}$ closely enough, then $F_0$ uniformly approximates $f_{0}^\infty$ with a given in advance degree of precision.  

Now, by assumption, we can uniformly approximate each $f_n$ as closely as we wish by homeomorphisms $h_n$ with 
$h_n(Y^0_{n+1}) = Y^0_n$. 
It follows from Statement and points (1--3) that $f_{0}^\infty$ is uniformly approximated, to an arbitrary degree of precision, by homeomorphisms $F_0$ with 
$F_0(Y_\infty^0)=Y^0_0$, as required. 
\end{proof}

\subsection{The categories $\mathcal{C}_{\star}(\Delta)$ and $\mathcal{C}(\Delta)$---lower bounds on  $[{\mathcal S}_\star(\Delta)]$ and $[{\mathcal S}(\Delta)]$} \label{S:ResultsOnHereditarilyCellularMaps}

The theorem below gives lower estimates on the coininitial closure of ${\mathcal S}_*(\Delta)$. 

\begin{theorem}\label{Theorem:CellularAndSelections}
Let $\Delta$ be the $n$-simplex. 
\begin{enumerate}
\item[(i)] $\mathcal{S}_{\star}(\Delta)\subseteq \mathcal{C}_{\star}(\Delta)$ and $\mathcal{S}(\Delta)\subseteq \mathcal{C}(\Delta)$. 

\item[(ii)] $\mathcal{C}_{\star}(\Delta)\subseteq [ \mathcal{S}_{\star}(\Delta)]$ and $\mathcal{C}(\Delta)\subseteq [\mathcal{S}(\Delta)]$.
In fact, for every $f\in \mathcal{C}_{\star}(\Delta)$, there exists $g\in \mathcal{C}_{\star}(\Delta)$ such that $f\circ g$ is the composition of elementary selections 
or an identity map. Similarly, for every $f\in \mathcal{C}(\Delta)$, there exists $g\in \mathcal{C}(\Delta)$ such that $f\circ g$ is the composition of elementary selections 
or an identity map.
\end{enumerate}
\end{theorem} 

In point (ii) of the theorem above, the inclusion $\mathcal{C}_{\star}(\Delta)\subseteq [ \mathcal{S}_{\star}(\Delta)]$ is equivalent to saying that for every 
$f\in \mathcal{C}_{\star}(\Delta)$, there exists $g\in \mathcal{C}_{\star}(\Delta)$ such that $f\circ g$ is a selection map. The additional statement in point (ii) asserts 
that more is true; the map $f\circ g$ can be obtained from elementary selections and identity maps by composition only without resorting to the operation 
$h\to\beta h$ of barycentric subdivision on maps. The analogous remark applies to the inclusion $\mathcal{C}(\Delta)\subseteq [\mathcal{S}(\Delta)]$. 
This subtle point will turn out to be important in our proof of Theorem~\ref{T:topologicalRealization} in Section~\ref{Su:proofend}. 
(Note, however, that the operation of barycentric subdivision is essential in proving in Theorem~\ref{T:A} that the category ${\mathcal S}(\Delta)$ is projective 
Fra{\"i}ss{\'e}.) 

We also have the following auxiliary result that may be of independent interest. 

\begin{theorem}\label{Theorem:JointProjection}
For every stellar $n$-simplexes $A,B$ there is $k\geq 0$ and a map $h\colon \beta^kB\to A$ with $h\in \mathcal{C}_{\star}(\Delta)$. 
\end{theorem}

The following corollary is a consequence of Theorems~\ref{Theorem:CellularAndSelections} and \ref{Theorem:JointProjection} above, 
Theorem~\ref{T:A}, and results of Section~\ref{S:CoinitialClosure}. Theorem~\ref{Theorem:JointProjection}
is used in checking that $\mathcal{C}_{\star}(\Delta)$ has the joint projection property. 

\begin{corollary} \label{Corollary:CellularIsFraisse}
Let $\Delta$ be the $n$-simplex. Then $\mathcal{C}_{\star}(\Delta)$ and $\mathcal{C}(\Delta)$ are projective \Fraisse{} classes and their \Fraisse{} limit is 
isomorphic to $\DD$.
\end{corollary}

We start proving the results above by establishing some basic closure properties of  
$\mathcal{C}_{\star}(\Delta)$ and $\mathcal{C}(\Delta)$.

\begin{lemma}\label{L: Cellular has composition,anti-composition,triangulations}
Let $A,B,C$ be stellar $n$-simplexes, and let $f\colon B\to A$ and $g\colon C\to B$ be face preserving maps. 
\begin{enumerate}
\item[(i)] The identity map $\mathrm{id}\colon A\to A$ is in $\mathcal{C}_{\star}(\Delta)$.
\item[(ii)]  If $f\in \mathcal{C}_{\star}(\Delta)$ then $\beta f\in \mathcal{C}_{\star}(\Delta)$;
\item[(iii)]  $g\in\mathcal{C}_{\star}(\Delta)$ implies that $f \circ g\in\mathcal{C}_{\star}(\Delta)$ if and only if 
$f\in\mathcal{C}_{\star}(\Delta)$.
\end{enumerate} 
Analogous statements hold for $\mathcal{C}(\Delta)$.
\end{lemma}
\begin{proof} 
Notice that all these properties remain true after passing to subcollections $\mathcal{C}$ of arrows from $\mathcal{C}_{\star}(\Delta)$ with the property that if 
$f\colon B\to A$ is in $\mathcal{C}_{\star}(\Delta)$ with $A,B\in\mathrm{Ob}(\mathcal{C})$ then $f\in\mathcal{C}$. Hence, it suffice to prove them only for 
$\mathcal{C}_{\star}(\Delta)$.

For (i), let $\mathrm{X}\in \Delta$ and  $\sigma\in A_{\mathrm{X}}$. We need to show that $D(\sigma,\mathrm{id}_{\mathrm{X}})=D(\sigma,A_{\mathrm{X}})$ is a stellar ball. But $D(\sigma,A_{\mathrm{X}})$ is isomorphic to  the complex $[v]\star \beta(\mathrm{lk}(\sigma,A_{\mathrm{X}}))$, for some point $v$. The rest follows from (\ref{SB-identities}) and the fact that in any stellar manifold $M$, the  complex $\mathrm{lk}(\sigma,M)$ is always, either a stellar ball or a stellar sphere. 

Properties  (ii) and (iii)  follow directly from Theorem \ref{Theorem:Akin-Cohen}, given the next claim.
\begin{claim}
Let $f\colon B\to A$ be in $\mathcal{C}_{\star}(\Delta)$. Then $f$ is transversely cellular on every face, i.e., 
for all $\mathrm{X}\in\Delta$ both $f\res B_{\mathrm{X}}$ and   $f\res \partial B_{\mathrm{X}}$ are cellular.
\end{claim}

\noindent {\em Proof of Claim.}
Let $f\colon B\to A$ be a hereditarily cellular map. We need to show that  $f\res \partial B_{\mathrm{X}}$ is cellular for all $\mathrm{X}\in\Delta$.  
The argument below will not depend on $\mathrm{X}$ and it therefore suffices to show that $(f\res \partial B)\colon \partial B \to \partial A$ is cellular. 
Fix some $\sigma\in \partial A$. By Theorem \ref{Cohen's original theorem}(ii), and since $f_{\mathrm{Y}}$ is cellular for each $\mathrm{Y}\in\Delta$, the collection: 
\[\mathcal{B}=\big\{ D(\tau ,f_{\mathrm{Y}}) \mid \; \mathrm{Y} \in\partial \Delta, \; \tau\in A_{\mathrm{Y}}, \;  \sigma\subseteq \tau   \big\}\]
is a cell-system with $U_{\mathcal{B}}=D(\sigma, f\res \partial B)$. Arguing as in (1) we have that the family:
\[\mathcal{A}=\big\{ D(\tau ,A_{\mathrm{Y}}) \mid \; \mathrm{Y} \in\partial \Delta, \; \tau \in A_{\mathrm{Y}}, \;  \sigma\subseteq \tau\big\}\]
is also a cell-system with $U_{\mathcal{A}}=D(\sigma,\partial A)$. Since $\partial A$ is a stellar sphere we have that $D(\sigma,\partial A)$ is a stellar ball. Hence it suffices to show that  the complexes $U_{\mathcal{A}}$ and $U_{\mathcal{B}}$ are stellar equivalent. This follows from Corollary \ref{C: stellar-equivalent to Poset}, since  the posets $(\mathcal{A},\subseteq)$ and $(\mathcal{B},\subseteq)$ are isomorphic.  The claim and, therefore, the lemma are now proved. 
\end{proof}

The above lemma immediately gives the following proposition.

\begin{proposition}\label{C:CelluAreCat}
Both  collections $\mathcal{C}_{\star}(\Delta)$ and $\mathcal{C}(\Delta)$ are categories.
\end{proposition}

Let $A$ be a finite simplicial complex and let $\sigma$ be a finite set. 
Consider the complex $\langle \sigma\rangle A$ that is defined by enumerating the set $\{\tau  \in A \mid \sigma\subseteq \tau\}$ in a never  $\subseteq$-increasing order $\sigma_1,\sigma_2,\ldots,\sigma_m=\sigma$ and subdividing these faces in this order, i.e.,
\[\langle \sigma\rangle A:= (\sigma_m,\sigma_m)\cdots (\sigma_2,\sigma_2)(\sigma_1,\sigma_1)A.\]
It is not difficult to see that any two  never $\subseteq$-increasing enumerations give rise to the same complex and therefore $\langle \sigma\rangle A$ is well defined. Let $a$ be a new point.  By a {\bf connection map} $e\colon(\sigma,a)A\to A$ we mean any map that extends the identity map on $\mathrm{dom}(A)$  and satisfies $e(a)\in \sigma$. By the {\bf star-contraction}  $t\colon \langle \sigma\rangle A\to (\sigma,a)A$ we mean the  map that extends the identity map on $\mathrm{dom}(A)$  and satisfies $t(\tau)=a$ for every $\tau\supseteq \sigma$. It is easy to see that both connection maps and star-contractions are simplicial epimorphisms. We extend these definitions on any stellar simplex $A=(A_{\mathrm{X}}\mid \mathrm{X}\in\Delta)$ to get the stellar simplexes $(\sigma,a) A:=((\sigma,a) A_{\mathrm{X}}\mid \mathrm{X}\in\Delta)$ and   $\langle \sigma\rangle A:=(\langle \sigma\rangle A_{\mathrm{X}}\mid \mathrm{X}\in\Delta)$. It follows that the associated connection maps $e$ and the star-contraction $t$ are elements of $\mathcal{R}_{\star}(\Delta)$.

\begin{lemma}\label{L:SimpleCellularMaps}
Let  $A$ be  stellar simplex, let $\sigma\in A$, and let $a$ be a new point. Then $\mathcal{C}_{\star}(\Delta)$ contains 
\begin{enumerate}
\item[(i)] all connection maps, $e\colon (\sigma,a) A\to A$;
\item[(ii)] the contraction map, $t\colon \langle \sigma\rangle A\to (\sigma,a)A$.
\end{enumerate}
\end{lemma}
\begin{proof}
Let $f\colon C\to B$ be any of the maps above.  Let also $\mathrm{X}\in \Delta$ and let $\rho\in B_{\mathrm{X}}$.  
By Theorem~\ref{T:Collapsible}, it suffices to show that $D(\rho,f_{\mathrm{X}})$ is collapsible. By Proposition~\ref{P:CollapseRegularNeigh}, we have that  $D(\rho,f_{\mathrm{X}})$ collapses to $D:=(\beta f_{\mathrm{X}})^{-1}(\rho)$, and therefore, it suffices to show that $D$ is collapsible. To see this, consider three cases. If $f=e$, $e(a)=b$, and $b\in\rho$, then $D$ is isomorphic to $\beta [\{a,b\}]$ which is clearly collapsible. If $f=t$ and $a\in\rho$, then $D$ is isomorphic to the chain complex of the poset $\big( \mathrm{st}(\sigma,A), \subseteq \big)$ which is a cone, and therefore, collapsible. In all other cases $D$ is a singleton.
\end{proof}

Let $C,D$ be two stellar simplexes and let $\overline{\delta}=\delta_1,\ldots,\delta_k$ be a sequence of stellar moves which factor through the cell-system $(C_{\mathrm{X}}\mid \mathrm{X}\in \Delta)$; see Section \ref{S:Cell-systems}.
A {\bf zig-zag connection $\overline{e}$ from $C$ to $D$} is a sequence of connection maps $e_l$ between consecutive stellar simplexes $C_{l-1},C_l=\delta_l C_{l-1}$:
\[D \overset{e_k}{\longleftarrow} C_{k-1} \overset{e_{k-1}}{\longrightarrow} C_{k-2} \longleftarrow \cdots  \overset{e_3}{\longleftarrow}  C_2 \overset{e_2}{\longrightarrow} C_1  \overset{e_1}{\longleftarrow} C\]
whose direction depends on whether $\delta_l$ is a division or a weld. The next lemma shows that, within $\mathcal{C}_{\star}(\Delta)$, we can always find a composable sequence $\overline{s}$ of elementary selection maps from $\mathcal{S}(C)\subseteq\mathcal{S}_{\star}(\Delta)$, which ``refines'' $\overline{e}$.

\begin{lemma}\label{L: factoring through stellar moves}
Let $g\colon D\to C$ be in $\mathcal{C}_{\star}(\Delta)$ and let $e$ be a connection map between the stellar simplexes $C$ and $C'$. Then there is a map $g'\colon D' \to C'$ in $\mathcal{C}_{\star}(\Delta)$ and an elementary selection or an identity map $s\colon D'\to D$, so that the square diagram formed by $g,g',e,s$ commutes. 
\begin{center}
\begin{tikzcd}[sep=large, every label/.append style={font=\normalsize}]
C' \arrow[r, "e", swap ] &  C   &  C'   &  C\arrow[l, "e"]  \\
D'  \arrow[u, "{g'}", dotted] \arrow[r,  "{s}", dotted]  & D \arrow[u, "{g}"]  & D'  \arrow[u, "{g'}", dotted ] \arrow[r,  "{s}", dotted] &  D  \arrow[u, "{g}"]   
\end{tikzcd}
\end{center}
\end{lemma}
\begin{proof}
If $e$ is a connection map from $C$ to $C'$, then let $D'=D$, $s=\mathrm{id}_{D'}$ and $g'=e\circ g\circ s$. The map $g'$ is restricted cellular by Lemma \ref{L:SimpleCellularMaps} and since by (ii) of Lemma \ref{L: Cellular has composition,anti-composition,triangulations}, $\mathcal{C}_{\star}(\Delta)$ is closed under composition. 

If $e$ is a connection map from $C'$ to $C$ then $C'=(\sigma,a)C$, for some $\sigma\in C$. Let $b=e(a)$ and set  $D'=\beta D$. For every $\tau\in C$, with $\tau\not\supseteq\sigma$, pick a point $a_\tau\in \tau$ and consider the map $h: \beta C \to (\sigma,a) C$ with $h(\tau)=a$ if $\sigma\subseteq\tau$ and $h(\tau)=a_{\tau}$ otherwise. The map $h$ is formed by composing star-contraction map $t\colon \langle \sigma\rangle A\to (\sigma,a)A$ with a sequence of connection maps.
By Lemma \ref{L:SimpleCellularMaps} and since by (ii) of Lemma \ref{L: Cellular has composition,anti-composition,triangulations} $\mathcal{C}_{\star}(\Delta)$ is closed under composition, $h$ belongs to $\mathcal{C}_{\star}(\Delta)$.

Set $g'= h \circ \beta g$ and let $\rho\in D$. If $g_{*}(\rho)\supseteq \sigma$ let $b_\rho\in\rho$ with $g(b_\rho)=a_{\sigma}$. Otherwise, set $\tau=g_{*}(\rho)$. Since $a_\tau\in\tau$, there is a point $b_\rho$ in $\rho$ with $g(b_\rho)=a_\tau$. Let $s\colon\beta D\to D$ with $s(\rho)=b_\rho$. Then $s$ is an elementary selection and $g \circ s= e\circ g'$.
\end{proof}

We can now prove Theorem \ref{Theorem:JointProjection}.

\begin{proof}[Proof of Theorem \ref{Theorem:JointProjection}]
Let $A,B$ be two stellar $n$-simplexes we will find $h\colon \beta^kB\to A$ with $h\in \mathcal{C}_{\star}(\Delta)$, for some large enough $k\geq 0$. 

By Corollary \ref{C: stellar-equivalent to Poset}, we have a sequence  $\overline{\delta}=\delta_1,\ldots,\delta_k$ of stellar moves factoring through the stellar simplex cell-system $B=(B_{\mathrm{X}}\mid \mathrm{X}\in \Delta)$, with $\overline{\delta}B=A$. 
Choose any zig-zag sequence $\overline{e}$ of connection maps following the transformation $B\to\overline{\delta}B$. Set $B_0=B$ and let $g_0\colon B_0\to B$ be the identity map. By Lemma \ref{L: factoring through stellar moves}, we can inductively produce a stellar simplex $B_l\in\mathcal{S}(B)$ 
and a $\mathcal{C}_{\star}(\Delta)$-map $g_l\colon B_l \to \delta_l \cdots \delta_1 B$, for every $l$ with $0<l\leq k$. Set $C=B_k$ and $g=g_k$.
\end{proof}

Let $f\colon C\to A$ and $f'\colon D\to A$ be two maps in $\mathcal{C}(\Delta)$. On the level of the stellar simplexes $C, D$  one can always find a zig-zag connection $\overline{e}$ from $C$ to $D$ using Corollary \ref{C: stellar-equivalent to Poset}. The next lemma uses the full strength of Proposition \ref{P:TameCellSystemAdmitsStarring} to provide conditions under which we can ``lift'' this zig-zag connection to the level of the hereditarily cellular maps $f, f'$. A {\bf cellular connection $(\overline{e},\overline{f})$ from $f$ to $f'$} is a zig-zag connection $e_1,\ldots,e_k$ from $C$ to $D$, together with a sequence $f_0,\ldots,f_k\in\mathcal{C}_{\star}(\Delta)$, with $f_0=f$, $f_k=f'$, so that for all $l>0$ the triangle formed by $f_{l-1},e_l,f_l$ commutes.

\begin{center}
\begin{tikzcd}[sep=large, every label/.append style={font=\normalsize}]
& & &  A & & & \\ 
\\
\\
D \ar[uuurrr, "f_k" description] & C_{k-1}\ar[l,  "e_k"] \ar[r,  "e_{k-1}" swap]\ar[uuurr, "f_{k-1}" description] & C_{k-2}  \ar[uuur, "f_{k-2}" description]  &\cdots \ar[l,  "e_{k-2}"]  &
C_2 \ar[l,  "e_3"] \ar[r,  "e_2" swap] \ar[uuul, "f_2" description] & C_1   \ar[uuull, "f_1" description] & C \ar[l,  "e_1"] \ar[uuulll, "f_0" description] \\
\end{tikzcd}
\end{center}

\begin{lemma}\label{L: connection}
Let $f\colon B\to A$ and $g\colon C\to A$ be two maps in $\mathcal{C}_{\star}(\Delta)$ and let $s\colon \beta\beta B\to \beta B$, $t\colon \beta\beta C\to \beta C$ be the selection maps defined by the assignment
\begin{equation}\label{eq:s,t selection}
\{\sigma_0\subsetneq\ldots\subsetneq\sigma_j\}\to\sigma_0.
\end{equation}
Then there is a cellular connection $(\overline{e},\overline{f})$, from $\beta f\circ s$ to $\beta g\circ t$.
\end{lemma}
\begin{proof}

Consider the system $\mathcal{A}=\{ D(\sigma, A_{\mathrm{X}}) \mid \mathrm{X}\in \Delta, \sigma\in A_{\mathrm{X}}\}$ and  let $\mathcal{C}(P_{\mathcal{A}})$ be the chain complex of the associated poset $P_{\mathcal{A}}=(\mathcal{A},\subseteq)$; see Figure \ref{fig:M3}. Let also
\[\mu\colon \mathcal{C}(P_{\mathcal{A}})\to \beta A,\]
be the simplicial map defined by the assignment $D(\sigma, A_{\mathrm{X}})\to \sigma$. 
\begin{figure}[h]
\centering
\begin{tabular}{c c c}
\begin{tikzpicture}[scale=0.30]
		\node  (0) at (0, 3) {};
		\node  (1) at (-4, -2) {};
		\node  (2) at (4, -2) {};
		\node  (3) at (0, -2) {};
		\draw (0.center) to (1.center);
		\draw (2.center) to (0.center);
		\draw (3.center) to (0.center);
		\draw (1.center) to (3.center);
		\draw (2.center) to (3.center);
\end{tikzpicture}
&
\begin{tikzpicture}[scale=0.35]
		\node  (0) at (0, 3) {};
		\node  (1) at (-4, -2) {};
		\node  (2) at (4, -2) {};
		\node  (3) at (0, -2) {};
		\node  (4) at (-2, 0.4999999) {};
		\node  (5) at (1.999999, 0.4999999) {};
		\node  (6) at (-2, -2) {};
		\node  (7) at (1.999999, -2) {};
		\node  (8) at (0, 0.4999999) {};
		\draw (0.center) to (1.center);
		\draw (2.center) to (0.center);
		\draw (3.center) to (0.center);
		\draw (1.center) to (3.center);
		\draw (2.center) to (3.center);
		\draw (1.center) to (8.center);
		\draw (0.center) to (6.center);
		\draw (3.center) to (4.center);
		\draw (7.center) to (0.center);
		\draw (3.center) to (5.center);
		\draw (2.center) to (8.center);
\end{tikzpicture}
&
\begin{tikzpicture}[scale=0.236]
		\node  (0) at (0, 3) {};
		\node  (1) at (-4, -2) {};
		\node  (2) at (4, -2) {};
		\node  (3) at (0, -2) {};
		\node  (4) at (-2, 0.4999999) {};
		\node  (5) at (1.999999, 0.4999999) {};
		\node  (6) at (-2, -2) {};
		\node  (7) at (1.999999, -2) {};
		\node  (8) at (0, 0.4999999) {};
		\node  (9) at (-6, -3.25) {};
		\node  (10) at (6, -3.25) {};
		\node  (11) at (0, 4.75) {};

		\node  (12) at (-5.1 , -2) {};
		
		\node  (13) at (-4, -3.25) {};
		\node  (14) at (-0.9999998, 3.5) {};
		\node  (15) at (1.000001, 3.5) {};
		\node  (16) at (3.999999, -3.25) {};
		\node  (17) at (5, -2) {};
		\node  (18) at (3, 0.7500001) {};
		\node  (19) at (-3, 0.75) {};
		\node  (20) at (0, -3.25) {};
		\node  (21) at (-2, -3.25) {};
		\node  (22) at (2, -3.25) {};
		
		\draw (0.center) to (1.center);
		\draw (2.center) to (0.center);
		\draw (3.center) to (0.center);
		\draw (1.center) to (3.center);
		\draw (2.center) to (3.center);
		\draw (1.center) to (8.center);
		\draw (0.center) to (6.center);
		\draw (3.center) to (4.center);
		\draw (7.center) to (0.center);
		\draw (3.center) to (5.center);
		\draw (2.center) to (8.center);
		\draw (9.center) to (8.center);
		\draw (8.center) to (10.center);
		\draw (8.center) to (11.center);
		\draw (10.center) to (11.center);
		\draw (9.center) to (11.center);
		\draw (9.center) to (10.center);
		\draw (20.center) to (3.center);
		\draw (18.center) to (5.center);
		\draw (4.center) to (19.center);
		\draw (14.center) to (0.center);
		\draw (15.center) to (0.center);
		\draw (2.center) to (17.center);
		\draw (2.center) to (16.center);
		\draw (1.center) to (13.center);
		\draw (12.center) to (1.center);
		\draw (6.center) to (21.center);
		\draw (7.center) to (22.center);
		\draw (22.center) to (2.center);
		\draw (2.center) to (18.center);
		\draw (0.center) to (18.center);
		\draw (19.center) to (0.center);
		\draw (19.center) to (1.center);
		\draw (21.center) to (1.center);
		\draw (21.center) to (3.center);
		\draw (22.center) to (3.center);
\end{tikzpicture}
\\
$A$ & $\beta A$ & $\mathcal{C}(P_{\mathcal{A}})$\\
\end{tabular}
\caption{} \label{fig:M3}
\end{figure}

Since the existence of cellular paths is a symmetric and transitive relation between maps in $\mathcal{C}_{\star}(\Delta)$, it suffice to find  cellular connections from  $\beta f\circ s$ to $\mu$ and  from $\beta g\circ t$ to $\mu$. Since the cases are similar, the lemma follows from the next claim.

\begin{claim}
There is a cellular connection $(\overline{e},\overline{f})$, from $\beta f\circ s$ to $\mu$. 
\end{claim}

\noindent{\em Proof of Claim.} 
In the process of this proof we will consider various cell-systems $\mathcal{S}$ whose associated poset $(\mathcal{S},\subseteq)$ is isomorphic to $P_{\mathcal{A}}$. for For notational purposes it  will be convenient to have them indexed by $P_{\mathcal{A}}$. If $(D_p)_{p\in P_{\mathcal{A}}}$ is such a cell-system and $v$ is a vertex of the complex $\bigcup_p D_p$, then we  denote by $\mathrm{Supp}(v,P_{\mathcal{A}})$ the least $p\in P_{\mathcal{A}}$ with $v\in\mathrm{dom}(D_p)$. It is easy to see that such $p$ always exists  and it coincides with the unique $p\in P_{\mathcal{A}}$ so that $v\in\mathrm{dom}(D_p) \setminus \mathrm{dom}(\partial D_p)$.

Consider the collection $\mathcal{B}:=\{D(\sigma, f_{\mathrm{X}})\mid \mathrm{X}\in \Delta, \sigma\in A_{\mathrm{X}} \}$. By Theorem \ref{Cohen's original theorem}(ii) and since $f$ is hereditarily cellular, we have that $\mathcal{B}$ is a cell-system on $\beta B$. As a consequence, $\beta \mathcal{B}:= \{\beta D(\sigma, f_{\mathrm{X}})\mid \mathrm{X}\in \Delta, \sigma\in A_{\mathrm{X}} \}$ is a tame cell-system  on $\beta\beta B$. 
Notice that the assignment  $\beta D(\sigma, f_{\mathrm{X}})\to D(\sigma, A_{\mathrm{X}})$ is an isomorphism between the posets $P_{\mathcal{B}}$ and $P_{\mathcal{A}}$. We can therefore use $P_A$ to index  $\beta \mathcal{B}$ as $(D_p)_{p\in P_A}$ and we set $D=\beta\beta B$. Notice that by the choice of the map $s$ we have that 
\[\beta D(\sigma, f_{\mathrm{X}})= (s_{\mathrm{X}})^{-1}_{*}\big(D(\sigma, f_{\mathrm{X}})\big).\]
From this it follows that for every $v\in\mathrm{dom}(D)$ we have that 
\begin{equation}\label{eq:mu}
\beta f\circ s(v)=  \mu\big(\mathrm{Supp}(v,P_{\mathcal{A}})\big).
\end{equation}
By Proposition \ref{P:TameCellSystemAdmitsStarring} the system $(D_p)_p$ admits a starring. Let $\overline{\delta}=\delta_1,\ldots,\delta_k$ be a  linearization of this starring. By Theorem~\ref{T:SystemTransformation} we can assume without the loss of generality that  $\overline{\delta} D= \mathcal{C}(P_{\mathcal{A}})$. After possibly removing some moves from $\overline{\delta}$ we can also assume that $\overline{\delta}$ is essential for $D$.
By Proposition \ref{P:TameCellSystemAdmitsStarring} it then follows that for every initial segment $\overline{\varepsilon}\delta$  of $\overline{\delta}$ there is a unique $p\in P_{\mathcal{A}}$ so that $\delta$ is strongly internal for $\overline{\varepsilon} D_p$.  Let $\pi\colon\{1,\ldots,k\}\to P_{\mathcal{A}}$ be the map witnessing this unique assignment and set $D^l=\delta_l\cdots\delta_1 D$ and $D_p^l=\delta_l\cdots\delta_1 D_p$, for all $p\in P_{\mathcal{A}}$ and all $l\leq k$. By Lemma \ref{L:tog} we have that $(D_p^l)_p$ is cell-system on $D^l$, for all $l\leq k$.

We can now define the cellular connection $(\overline{e},\overline{f})$ as follows. If $\delta_l$ is a stellar subdivision $(\sigma_l,a_l)$, then $\delta_l$ is strongly internal for $D_{\pi(l)}^{l-1}$. We can therefore pick some $b_l\in\sigma_l$ which lies in the interior of $D_{\pi(l)}^{l-1}$. Let $e_l\colon D^l\to D^{l-1}$ be the connection map  with $e_l(a_l)=b_l$ and let $f_l\colon D^l\to \beta A$ be the composition $f_{l-1}\circ e_l$.  Similarly, if $\delta_l$ is a stellar weld $(\sigma_l,a_l)^{-1}$ we pick some $b_l\in\sigma_l$ which lies in the interior of $D_{\pi(l)}^{l}$. Let $e_l\colon D^{l-1}\to D^{l}$ be the connection map with $e(a_l)=b_l$, and let $f_l\colon D^l\to \beta A$ be the restriction $f_{l-1}\res \mathrm{dom}(D^{l-1})\setminus \{a_l\}$. It is immediate that for all $l>0$, the triangle formed by $f_{l-1},e_l,f_l$ commutes. By Lemma \ref{L: Cellular has composition,anti-composition,triangulations} we have that $f_l\in\mathcal{C}_{\star}(\Delta)$.
Finally, by a straightforward induction based on  (\ref{eq:mu}) above and the choice of $b_l$, we get that for every $l\leq k$ and every $v\in\mathrm{dom}(D_l)$, 
we have that
\[f_l(v)=  \mu\big(\mathrm{Supp}(v,P_{\mathcal{A}})\big).\] 
It follows that $f_k=\mu$ and therefore $(\overline{e},\overline{f})$ is the desired cellular connection. This finish the proof of the claim and the theorem. 
\end{proof}

We can now establish Theorem \ref{Theorem:CellularAndSelections} and Corollary \ref{Corollary:CellularIsFraisse}.

\begin{proof}[Proof of Theorem~\ref{Theorem:CellularAndSelections}]
(i) To see that they respectively contain $\mathcal{S}_{\star}(\Delta)$ and $\mathcal{S}(\Delta)$, 
notice first that every elementary selection map $s\colon \beta A\to A$ is a composition of connection maps. 
The rest follows from Lemma \ref{L: Cellular has composition,anti-composition,triangulations} and Lemma \ref{L:SimpleCellularMaps}.

(ii) By Proposition~\ref{C:CelluAreCat}, both $\mathcal{C}_{\star}(\Delta)$ and $\mathcal{C}(\Delta)$ are categories. Thus, by point (i), 
it suffices to show the last two sentences of the statement of point (ii), that is, 
we need to see that 
for every $f\in\mathcal{C}_{\star}(\Delta)$, there is $g\in\mathcal{C}_{\star}(\Delta)$ with $f\circ g$ being a composition of elementary selections 
or an identity map, and, similarly, for every $f\in\mathcal{C}(\Delta)$, there is $g\in\mathcal{C}(\Delta)$ with $f\circ g$ 
a composition of elementary selections or an identity map. 
Notice that if $f\in\mathcal{C}(\Delta)$, $g\in \mathcal{C}_{\star}(\Delta)$, and $f \circ g$ is defined, then $g\in \mathcal{C}(\Delta)$. 
As a consequence, it is enough to prove only the statement about $\mathcal{C}_{\star}(\Delta)$.

Let $f\colon B\to A$ be in $\mathcal{C}_{\star}(\Delta)$ and let $s\colon\beta^2 B\to \beta B$ 
be the elementary selection map defined by the assignment (\ref{eq:s,t selection}) in Lemma \ref{L: connection}. Set 
\[
f'= \beta f \circ s.
\]
Since for each elementary selection map $r\colon \beta A \to A$ we can find a elementary selection map $r'\colon \beta B \to B$ 
so that $r \circ \beta f = f \circ r'$, it will suffice to find a $g\in \mathcal{C}_{\star}(\Delta)$ so that $f'\circ g$ is a composition of elementary selections or an identity map. 

\begin{center}
\begin{tikzcd}[sep=large, every label/.append style={font=\normalsize}]
& B \arrow[r, "f" ] & A  \\
\beta\beta B \ar[r, "s" ] \ar[rr, "f'" description, bend right=20  ]  & \beta B  \ar[u, "{r'}", swap] \ar[r,  "{\beta f}"]  &  \beta A \ar[u, "{r}", swap] \\ 
\end{tikzcd}
\end{center}
Let $C_0=\beta \beta A$ and let $f_0\colon C_0 \to \beta A$ be the map   $(\beta(\mathrm{id}_A))\circ t$, where $t\colon \beta\beta A\to \beta A$ is the elementary selection map defined by the assignment (\ref{eq:s,t selection}) in Lemma \ref{L: connection} and $\mathrm{id}_A$ is the identity on $A$. By Lemma \ref{L: connection} there exists a cellular connection $(\overline{e},\overline{f})$, from $f_0$ to $f_k=f'$. 

\begin{center}
\begin{tikzpicture}[baseline= (a).base]
\node[scale=.9] (a) at (0,0){
\begin{tikzcd}[every label/.append style={font=\normalsize}]
& & & & & & & & \beta A \\ 
\\
\\
C_k \ar[r,  "e_k"]\ar[uuurrrrrrrr, "f_k=f'"  description, bend left=15] & \cdots & \cdots   &   C_{l+1} \ar[l]  \ar[uuurrrrr, bend left=15] & C_{l}\ar[l] \ar[r , "e_l", swap] \ar[uuurrrr, "f_l"  description, bend left=15] & C_{l-1} \ar[uuurrr, bend left=15]& \cdots \ar[l] & \cdots \ar[r,  "e_1"] & C_0 \ar[uuu, "f_0"  description ]\\
\end{tikzcd}
};
\end{tikzpicture}
\end{center}

Let $D_0=C_0$, and set $g_0\colon D_0\to C_0$ to be the identity map. By induction, using Lemma~\ref{L: factoring through stellar moves}, 
we can find, for every $l\in \{1,\ldots,k\}$, an elementary selection or an idenity map $s_l\colon D_l\to D_{l-1}$ and a $\mathcal{C}(\Delta)$-map 
$g_l\colon D_l \to C_l$ so that the square diagram formed by $g_{l-1},e_l,s_l,g_l$ commutes.  
\begin{center}
\begin{tikzcd}[every label/.append style={font=\normalsize},  column sep=scriptsize]
C_k \ar[r,  "e_k"] & \cdots & \cdots   &   C_{l+1} \ar[l]   & C_{l}\ar[l] \ar[r , "e_l"] & C_{l-1} & \cdots \ar[l] & \cdots \ar[r,  "e_1"] & C_0 \\
\\
\\
D_k \ar[r, "s_{k}", swap, outer sep=+5pt]\ar[uuu, "g_{k}" description]& \cdots & \cdots  \ar[r]&   D_{l+1} \ar[r] \ar[uuu]& D_{l} \ar[r, "s_{l}", swap, outer sep=+5pt]\ar[uuu, "g_{l}" description] & D_{l-1} \ar[r] \ar[uuu] & \cdots & \cdots \ar[r , "s_{1}", swap, outer sep=+5pt] & D_0 \ar[uuu, "g_{0}" description] \\

\end{tikzcd}
\end{center}
Let  $g=g_k$, and notice that $f'\circ g$ is equal to $f_0\circ g_0\circ s_1\circ \cdots \circ s_k$, which is a composition of elementary selections, since $\beta(\mathrm{id}_A)=\mathrm{id}_{\beta A}$. 
\end{proof}

\begin{proof}[Proof of Corollary \ref{Corollary:CellularIsFraisse}]
By Proposition \ref{P:cos}, Theorem \ref{Theorem:Intro1}, and Theorem \ref{Theorem:CellularAndSelections} 
we have that $\mathcal{C}(\Delta)$ is a projective \Fraisse{} class whose \Fraisse{} limit is isomorphic to $\DD$.

Since  $\mathcal{S}_{\star}(\Delta)$ satisfies the projective amalgamation property, 
by Proposition \ref{P:cos}  and Theorem \ref{Theorem:CellularAndSelections} so does $\mathcal{C}_{\star}(\Delta)$. Moreover, $\mathcal{C}_{\star}(\Delta)$ 
satisfies the joint projection property since $\mathcal{S}_{\star}(\Delta)\subseteq \mathcal{C}_{\star}(\Delta)$ by 
Theorem~\ref{Theorem:JointProjection}. 
Thus,  $\mathcal{C}_{\star}(\Delta)$ is a projective \Fraisse{} class. To see that  the projective \Fraisse{} limit $\DD'$ of $\mathcal{C}_{\star}(\Delta)$ is isomorphic to 
$\DD$ one defines a ``back and forth" system between 
$\DD'$ and $\DD$ using: the projective extension property of $\DD'$,
together with the fact that $\mathcal{C}(\Delta)\subseteq \mathcal{C}_{\star}(\Delta)$;
 and the projective extension property of $\DD$, together with Theorem \ref{Theorem:JointProjection}.
\end{proof}

\subsection{Proof of Theorem~\ref{T:topologicalRealization}}\label{Su:proofend}

We come back now to the proof of the theorem on the topological realization of the projective Fra{\"i}ss{\'e} limit of the category ${\mathcal S}(A)$ for 
a finite complex $A$. 

\begin{proof}[Proof of Theorem~\ref{T:topologicalRealization}]
The reader may consult Section~\ref{SS:Concrete} for notation in this proof.

We first prove the theorem for the special case when $A=\Delta$. Let ${\mathcal S}_0(\Delta)$ be the category whose objects are barycentric subdivisions 
$\beta^k\Delta$ of $\Delta$, $k\in {\mathbb N}$, and whose morphisms 
are identity maps on $\beta^k\Delta$ and compositions of elementary selections. 
By Theorem~\ref{Theorem:CellularAndSelections}(i) it is clear that 
\[
{\mathcal S}_0(\Delta)\subseteq {\mathcal S}(\Delta)\subseteq {\mathcal C}(\Delta), 
\]
and by Theorem~\ref{Theorem:CellularAndSelections}(ii), ${\mathcal S}_0(\Delta)$ and ${\mathcal S}(\Delta)$ are dominating in ${\mathcal C}(\Delta)$. 

Let $(g_i)$ be a generic sequence for ${\mathcal S}(\Delta)$, so $g_i\colon B_{i+1}\to B_i$, where 
each $B_i$ is some barycentric subdivision of $\Delta$. Let ${\mathbb B}$ be the profinite simplicial complex associated with $(g_i)$. Let 
$R^{\mathbb B}$ be the edge relation on ${\mathbb B}$. By Lemma~\ref{L:EquivalenceRelation}, $R^{\mathbb B}$ is an equivalence relation. 
Our aim is to show that ${\mathbb B}/R^{\mathbb B}$ is homeomorphic to $|\Delta|_{\mathbb R}$ by 
a homeomorphism that respects the structure of subcomplexes of $\Delta$ as in the conclusion of Theorem~\ref{T:topologicalRealization}. 
Since ${\mathcal S}(\Delta)$ is dominating in ${\mathcal C}(\Delta)$, by Proposition~\ref{P:cos}(iii),
the sequence $(g_i)$ is generic for ${\mathcal C}(\Delta)$. Since ${\mathcal S}_0(\Delta)$ is dominating in ${\mathcal C}(\Delta)$,
it follows from Proposition~\ref{P:cos}(iv) and Corollary~\ref{Corollary:CellularIsFraisse} that there exists a generic sequence 
$(f_i)$ for ${\mathcal C}(\Delta)$ such that each $f_i$ is in ${\mathcal S}_0(\Delta)$. Let $f_i\colon A_{i+1}\to A_i$, where each $A_i$ is a barycentric subdivision 
of $\Delta$. Let ${\mathbb A}$ be the profinite simplicial complex associated with $(f_i)$, 
and let $R^{\mathbb A}$ be the edge relation on ${\mathbb A}$. 
Since both $(g_i)$ and $(f_i)$ are generic for ${\mathcal C}(\Delta)$, 
they are isomorphic in ${\mathcal C}(\Delta)$ by Theorem~\ref{T:FraisseOriginal}. In particular, 
there is an isomorphism 
\[
\phi\colon {\mathbb A} \to {\mathbb B}. 
\]
Thus, $R^{\mathbb A}$ is an equivalence relation and $\phi$ induces a homeomorphism 
\[
{\mathbb A}/R^{\mathbb A}\to {\mathbb B}/R^{\mathbb B}. 
\]
Since each $f_i$ is a composition of elementary selections or an identity map, 
the inverse sequence $(A_i,f_i)$ can be viewed as an inverse sequence as in Lemma~\ref{L:geomsel}. 
Thus, by Lemma~\ref{L:geomsel}, ${\mathbb A}/R^{\mathbb A}$, and, therefore, also ${\mathbb B}/R^{\mathbb B}$, is homeomorphic to $|\Delta|_{\mathbb R}$ by 
a homeomorphism that respects the structure of subcomplexes of $\Delta$ as in the conclusion of Theorem~\ref{T:topologicalRealization}.

The general case reduces to the special case $A=\Delta$ as follows. Let $A$ be an arbitrary finite complex. 
Then $|A|$ is a compact metric space that can be represented as follows.  For each face $\sigma\in A$, there is a compact non-empty subspace $X_\sigma$ of $|A|$ such that, with letting 
$X_\emptyset = \emptyset$, we have
\begin{equation}\label{E:struc}
X_\sigma\cap X_\tau = X_{\sigma\cap \tau}\;\hbox{ and }\; \bigcup_{\sigma\in A} X_\sigma = |A|. 
\end{equation} 
Using the notation in \eqref{E:geot}, the special case when $A=\Delta$ may be reformulated to the following statement.

\smallskip
\noindent {\em Statement 1.} There exists a homeomorphism from $X_\sigma$ to $|\sigma|_{\mathbb R}$ 
that maps $X_\tau$ for $\emptyset\not=\tau\subsetneq \sigma$ onto $|\tau|_{\mathbb R}$. 
\smallskip

\noindent The following observation is easy to show.

\smallskip
\noindent {\em Statement 2.} Each homeomorphism of the boundary of $|\sigma|_{\mathbb R}$ 
\[
\bigcup \{ |\tau|_{\mathbb R}\mid \emptyset\not= \tau\subsetneq \sigma\}
\]
extends to a homeomorphism of $|\sigma|_{\mathbb R}$.
\smallskip

Using these two statements we can now finish the proof. Let $n={\rm dim}(A)$ and, for $k\leq n$, let $A_{\leq k}$ be the $k$-skeleton of $A$, that is, the subcomplex of $A$ consisting of all faces of 
dimension not exceeding $k$. By induction on $k$, we produce a homeomorphism 
\begin{equation}\notag
\phi_k\colon  \bigcup_{\sigma\in A_{\leq k}} X_\sigma \to |A_{\leq k}|_{\mathbb R},
\end{equation} 
with notation as in \eqref{E:geos}. We define $\phi_0$ in the obvious manner by mapping the unique point in $X_\sigma$ to the point in $|A_{\leq 0}|_{\mathbb R}$ 
corresponding to $\sigma$, for $\sigma\in A_{\leq 0}$. (The space $X_\sigma$ consists of one point by Statement 1.) If $k<n$, we produce $\phi_{k+1}$ out of 
$\phi_k$ as follows. Fix $\sigma\in A$ with ${\rm dim}(\sigma)=k+1$. Let $f_\sigma\colon X_\sigma\to |\sigma|_{\mathbb R}$ be given by Statement 1. Now 
use Statement 2 to find a homeomorphism $g_\sigma\colon |\sigma|_{\mathbb R}\to |\sigma|_{\mathbb R}$ so that 
$g_\sigma\circ f_\sigma\colon X_\sigma\to |\sigma|_{\mathbb R}$ extends $\phi_k\res \bigcup_{\emptyset\not= \tau\subsetneq \sigma} X_\tau$. 
Produce $\phi_{k+1}$ by extending $\phi_k$ to the subspace 
$\bigcup_{\sigma\in A, {\rm dim}(\sigma)=k+1} X_\sigma$ with $g_\sigma\circ f_\sigma$ for $\sigma\in A$, ${\rm dim}(\sigma)=k+1$. 
\end{proof}

\subsection{The categories $\mathcal{H}_{\star}(\Delta)$ and $\mathcal{H}(\Delta)$---upper bounds on  $[{\mathcal S}_\star(\Delta)]$ and 
$[{\mathcal S}(\Delta)]$}\label{S:ResultsOnRestrictedNearHomeo}

The following theorem and its corollary summarizes the relationship of $\mathcal{H}_{\star}(\Delta)$ and $\mathcal{H}_{\star}(\Delta)$ with the rest of the categories. 

\begin{theorem}\label{Theorem:HeredNearHomeoand Selection}
Let $\Delta$ be the $n$-simplex. Then 
\[
[\mathcal{S}_{\star}(\Delta)]   \subseteq \mathcal{H}_{\star}(\Delta)\,\hbox{ and }\, [\mathcal{S}(\Delta)]   \subseteq \mathcal{H}(\Delta).
\]
\end{theorem}

We register the following  slight refinement of the well known fact that near-homeomorphism are closed under composition (see 
\cite[Section 1.7, Exercise 3]{vM}). 

\begin{proposition}\label{P:nhoca} 
The collections $\mathcal{H}_{\star}(\Delta)$ and $\mathcal{H}(\Delta)$ are categories.
\end{proposition}

\begin{proof} 
Only closure under composition of the two classes needs to be checked. 
It is not difficult to see that if $X,Y, Z$ are compact metric spaces and the sequences $(f_n)$ and $(g_n)$ of continuous functions 
from $X$ to $Y$ and from $Y$ to $Z$, respectively, converge uniformly to $f\colon X\to Y$ and $g\colon Y\to Z$, then there exists a sequence $(k_n)$ 
of natural numbers such that 
\[
(g_{k_n}\circ f_n)
\]
converges uniformly to $g\circ f$. The conclusion of the proposition is an immediate consequence of this observation. 
\end{proof}

Let $\phi\colon K\to L$ and $\psi\colon K\to M$ be near-homeomorphisms between compact metric spaces. It is known that if  $\psi=\gamma \circ \phi$  for some   
$\gamma \colon L\to M$, then $\gamma$ is also a near-homeomorphism. The next lemma reproves this fact and gives some additional information that will be 
needed here. 

\begin{lemma}\label{L:nil}
Let $K, L, M$ be compact metric spaces. Let $\phi\colon K\to L$, $\psi\colon K\to M$, and $\gamma \colon L\to M$ be continuous with 
$\psi=\gamma \circ \phi$. Let $(\phi_p)$ and $(\psi_p)$ be sequences of homeomorphisms from $K$ to $L$ and from $K$ to $M$ 
converging uniformly to $\phi$ and $\psi$, respectively. Then there is a subsequence $(\psi_{k_p})$ of $(\psi_{p})$ such that the sequence 
\[
(\psi_{k_p}\circ \phi_{p}^{-1})
\]
converges uniformly to $\gamma$. 
\end{lemma}

\begin{proof} First we show that the sequence $(\psi \circ \phi_{p}^{-1})$ converges uniformly to $\gamma$. 
Since $L$ is compact, it suffices to show (see \cite[Exercise 4.2.E]{En}) that for each $y\in L$ and each sequence $(y_p)$ of points in 
$L$ converging to $y$, we have 
\begin{equation}\label{E:jhv}
\psi\circ \phi_{p}^{-1}(y_p) \to \gamma(y).
\end{equation}
Of course, it is enough to show that each subsequence of $(\psi \circ \phi_{p}^{-1}(y_p))$ has a further subsequence convergent to $\gamma(y)$. For 
simplicity of notation, we assume that the initial subsequence is the whole sequence $(\psi\circ \phi_{p}^{-1}(y_p))$. Consider the sequence $(\phi_{p}^{-1}(y_p))$ 
of elements of $K$. By compactness of $K$, this sequence has a convergent subsequence. Again, to simplify notation, assume that 
the whole sequence $(\phi_{p}^{-1}(y_p))$ converges to $x\in K$. It follows that 
\begin{equation}\label{E:a}
\psi\circ \phi_{p}^{-1}(y_p)\to \psi(x), 
\end{equation}
and, since $(\phi_{p})$ converges uniformly to $\phi$, that 
\[
y_p = \phi_{p}\circ \phi_{p}^{-1}(y_p)\to \phi(x). 
\] 
This last formula gives $\phi(x)=y$, hence 
\begin{equation}\label{E:b}
\psi(x) = \gamma\circ \phi(x) = \gamma(y). 
\end{equation}
Now, \eqref{E:jhv} follows from \eqref{E:a} and \eqref{E:b}. 

Since  $(\psi\circ \phi_{p}^{-1})$ converges uniformly to $\gamma$ and $(\psi_{p})$ converges uniformly to $\psi$, it is now easy 
to pick a subsequence of $(\psi_{p})$ with the desired property. 
\end{proof}

We may now proceed to the proof of Theorem \ref{Theorem:HeredNearHomeoand Selection}.

\begin{proof}[Proof of Theorem \ref{Theorem:HeredNearHomeoand Selection}]
Notice that if $f \circ g$ is in $\mathcal{S}_{\star}(\Delta)$ and $f\in\mathcal{S}(\Delta)$ then $f \circ g$ is in $\mathcal{S}(\Delta)$. 
As a consequence, it suffices to prove that $[\mathcal{S}_{\star}(\Delta)]\subseteq \mathcal{H}_{\star}(\Delta)$.

We show first that $\mathcal{S}_{\star}(\Delta)\subseteq \mathcal{H}_{\star}(\Delta)$. 
It suffices to see that $\mathcal{H}_{\star}(\Delta)$ is closed 
under composition and the operation $f\to \beta(f)$ and that it contains $1_A$ and each elementary selection $s\colon \beta^{k+1}A\to \beta^kA$, 
for each stellar simplex $A$. The class $\mathcal{H}_{\star}(\Delta)$ obviously contains $1_A$ and is closed under composition 
as it is a category by Proposition~\ref{P:nhoca}. 

The rest of the proof of the inclusion 
$\mathcal{S}_{\star}(\Delta)\subseteq \mathcal{H}_{\star}(\Delta)$ involves the barycentric subdivision operator and geometric realizations. It will be helpful to 
define some auxiliary notions and adopt some conventions. Let $C$ be a finite simplicial complex. 
Its geometric realization $|C|_{\mathbb R}$ is given by \eqref{E:geomr}. For a subcomplex $D$ of $C$, by $|D|_{\mathbb R}$ we denote 
the corresponding to $D$ geometric subcomplex of $|C|_{\mathbb R}$ as in \eqref{E:geos}. 
For a face $\sigma$ of $C$, by $|\sigma|_{\mathbb R}$
we denote the associated geometric face of $|C|_{\mathbb R}$ as in \eqref{E:geot}. We also identify vertices in ${\rm dom}(C)$ with the corresponding 
vertices of $|C|_{\mathbb R}$. 
Now let $\bar{x} = (x_\sigma)_{\sigma\in C}$ be a sequence of points in $|C|_{\mathbb R}$ such that $x_\sigma=v$ if 
$\sigma =\{ v\}$ for a $v\in {\rm dom}(C)$, 
and $x_\sigma$ an element of the interior of the face $|\sigma|_{\mathbb R}$ of 
$|C|_{\mathbb R}$
if $\sigma$ contains at least two vertices in ${\rm dom}(C)$. Let 
\[
\phi_{C,\bar{x}} \colon | \beta C|_{\mathbb R}\to |C|_{\mathbb R}
\]
be the function that is the unique affine extension of the function that
maps the vertex $\sigma$ of $\beta C$ (that is, a face of $C$) to $x_\sigma$ and each vertex of $C$ to itself.
The reader will easily check that $\phi_{C,\bar{x}}$ is a homeomorphism from $|\beta C|_{\mathbb R}$ to $|C|_{\mathbb R}$ that 
maps $|\beta D|_{\mathbb R}$ to $|D|_{\mathbb R}$ for each subcomplex $D$ of $C$.

We check that $\mathcal{H}_{\star}(\Delta)$ is closed under the operation $\beta$. 
We make a general observation first. Let $f\colon B\to A$ be a simplicial map for simplicial complexes $A,B$. 
Let $\phi\colon |\beta A|\to |A|$ and $\psi\colon |\beta B|\to |B|$ be given by 
\[
\phi= \phi_{A, \bar{x}} \,\hbox{ and }\, \psi= \psi_{B, \bar{y}}, 
\]
where $\bar{x} = (x_\sigma)_{\sigma\in A}$ is such that $x_\sigma$ is the geometric barycenter of $|\sigma|_{\mathbb R}$ and 
similarly $\bar{y} = (y_\tau)_{\tau\in B}$ is such that $y_\tau$ is the geometric barycenter of $|\tau|_{\mathbb R}$. Then 
one easily checks hat 
\begin{equation}\label{E:fpp}
|f|_{\mathbb R} \circ \phi = \psi \circ |\beta f|_{\mathbb R}. 
\end{equation} 
By the properties of $\phi_{A, \bar{x}}$ and $\psi_{B, \bar{y}}$, 
it follows from \eqref{E:fpp} that $f\in \mathcal{H}_{\star}(\Delta)$ implies $\beta f \in \mathcal{H}_{\star}(\Delta)$. 

It remains to show that $\mathcal{H}_{\star}(\Delta)$ contains elementary selections. 
Let $s\colon \beta^{k+1} A\to \beta^k A$ be an elementary selection, where $A$ is a stellar simplex. 
Set $B=\beta^k A$. Then $s\colon \beta B\to B$ is still an elementary selection. 
Let now, for $n\in {\mathbb N}$, $\bar{x}^n = (x^n_\sigma)_{\sigma\in B}$ be such that,
for each $\sigma\in B$, the sequence $(x_\sigma^n)$ converges to $s(\sigma)$. For each $n$, 
$\phi_{B,\bar{x}^n}$ is a homeomorphism that maps $|\beta C|_{\mathbb R}$ to $|\beta C|_{\mathbb R}$ for each subcomplex $C$ of $B$.
It is easy to check that the sequence 
$(\phi_{B,\bar{x}^n})$ converges uniformly to $|s|_{\mathbb R}$. Thus, $s$ is in $\mathcal{H}_{\star}(\Delta)$.

We move on to the proof of the inclusion $[\mathcal{S}_{\star}(\Delta)]   \subseteq \mathcal{H}_{\star}(\Delta)$. 
Let $g\colon B\to A$ be in $[{\mathcal S}_*(\Delta)]$ with $A$ and $B$ being stellar $n$-simplexes for some $n$. Clearly $g$ is simplicial. 
By Lemma~\ref{L:nil}, it will suffice to find a compact space $Y$ and continuous 
functions $f_1\colon Y\to |A|_{\mathbb R}$ and $f_2\colon Y\to |B|_{\mathbb R}$ such that $f_1= |g|_{\mathbb R}\circ f_2$ and 
for which there exist sequences of homeomorphisms $h_{1,p}\colon Y\to |A|_{\mathbb R}$ and $h_{2,p}\colon Y\to |B|_{\mathbb R}$
convergent uniformly to $f_1$ and $f_2$, respectively, and such that for each $p, q$,  
$h_{1,p}\circ h_{2,q}^{-1}$ maps $|B_X|_{\mathbb R}$ to $|A_X|_{\mathbb R}$ 
for each $X\in\Delta$. 
Obviously, $Y$ must be homeomorphic with $|\Delta|_{\mathbb R}$. 

By Theorem~\ref{T:iid}, fix an iso-sequence $(e_k)$ for ${\mathcal S}_*(\Delta)$ such that 
\[
g=e_0.
\]
and let 
\[
f_{1,k} = e_{2k}\circ e_{2k+1}\;\hbox{ and }\;f_{2,k} = e_{2k+1}\circ e_{2k+2}. 
\]  
By the definition of the notion of iso-sequence, each $f_{1,k}$ and $f_{2,k}$ is in ${\mathcal S}_*(\Delta)$. Let 
\[
A_k = {\rm codom}(f_{1,k})\;\hbox{ and }\; B_k = {\rm codom}(f_{2,k}).  
\]
Note that each $A_k$ and $B_k$ is a stellar $n$-simplex. 
Since $|f_{1,k}|$ and $|f_{2,k}|$ map $|(A_{k+1})_X|$ onto $|(A_k)_X|$, for each $X\in \Delta$, we can define 
\[
\begin{split}
(Y_1)_X &= \varprojlim_k\, (|(A_k)_X|_{\mathbb R},\, |f_{1,k}|_{\mathbb R}\res |(A_k)_X|_{\mathbb R})\\
(Y_2)_X &= \varprojlim_k\, (|(A_k)_X|_{\mathbb R},\, |f_{2,k}|_{\mathbb R}\res |(A_k)_X|_{\mathbb R}).
\end{split} 
\]
Set also 
\[
Y_1 = (Y_1)_{\{ 0, \dots , n\}} \; \hbox{ and }\; Y_2 = (Y_2)_{\{ 0, \dots , n\}}. 
\]
Since $f_{1,k}\in {\mathcal S}_*(\Delta)$, 
$|f_{1,k}|_{\mathbb R}$ can be uniformly approximated by restricted 
homeomorphisms. To see this, note that the class of maps that are uniformly approximated by restricted 
homeomorphisms is easily seen to be closed 
under composition, and to contain $1_{A}$ and all maps of the form $\beta^n(\delta)$ for a selection $\delta$ and a natural number $n$. 

From Theorem~\ref{T:Brown}, we get that the projection map $f^\infty_{1,0}\colon Y_1\to |A_0|_{\mathbb R}$ can be uniformly approximated by 
homeomorphisms mapping $(Y_1)_X$ to 
$|(A_0)_X|_{\mathbb R}$ for each $X\in \Delta$. 

Similarly, the above analysis can be repeated for $Y_2$.
  
Observe further that the sequence $(|e_k|_{\mathbb R})$ induces a homeomorphism $\phi\colon Y_2\to Y_1$ that maps $(Y_2)_X$ to $(Y_1)_X$ 
for each $X\in\Delta$ and is such that 
\[
f^\infty_{1,0}\circ \phi = |e_0|_{\mathbb R}\circ f^\infty_{2,0} = |g|_{\mathbb R}\circ f^\infty_{2,0}. 
\]
Now, it suffices to take $Y=Y_1$, $f_1 = f^\infty_{1,0}\circ \phi$, and $f_2 = f^\infty_{2,0}$ to obtain the desired objects as stated in the beginning of this proof. 
\end{proof}

\subsection{Comparison of the lower and upper bounds for $[{\mathcal S}(\Delta)]$ and $[{\mathcal S}_*(\Delta)]$}\label{Su:close}
By combining Theorems~\ref{Theorem:HeredNearHomeoand Selection} and \ref{Theorem:CellularAndSelections}, we have  that
\[ \mathcal{C}_{\star}(\Delta)\subseteq \mathcal{H}_{\star}(\Delta) \; \text{ and } \;
\mathcal{C}(\Delta)\subseteq \mathcal{H}(\Delta)\]
In Theorem \ref{T:last} below we show that the reverse inclusions also hold for any  $n<4$; or more generally, for all $n\geq 0$ if the PL-Poincar{\'e} conjecture is positively resolved for $n=4$. This establishes the last part of Theorem \ref{Theorem:Intron}. 

In the context of the theorem, due to Alexander \cite{Al}, that two finite simplicial complexes $A$ and $B$ are stellar equivalent if and only if 
$|A|_{\mathbb{R}}$ and $|B|_{\mathbb{R}}$ are PL-homeomorphic, the PL-Poincar{\'e} conjecture in dimension $n$ can be stated as follows:

\smallskip{}
\noindent {\em if $M$ is a combinatorial $n$-manifold which is homotopy equivalent to $\partial \Delta^{n+1}$, then  $M$ is stellar equivalent to $\partial \Delta^{n+1}$.}
\smallskip{}

\noindent While the PL-Poincar{\'e} conjecture remains open in dimension $n=4$, in all other dimensions  it  has been confirmed---for $n=2$, by the classification theorem for compact surfaces \cite{Bra} and the fact that every surface admits a unique PL-structure \cite{Ra}; for $n=3$, by  Perelman's solution to the smooth Poicar\'e conjecture \cite{Per1,Per2} and the fact that  every $3$-dimensional manifold admits a unique PL-structure \cite{Bra}; for $n\geq 5$,  by 
Smale \cite{Smale}. We summarize these results in the following theorem.

\begin{theorem}\label{T:Poincare}
The PL-Poincar{\'e} conjecture holds in every dimension  $n\neq 4$. 
\end{theorem}

The following theorem may be known; for example, our proof of it is close to the proof of \cite[Theorem 11.2]{Co1}. 
However, we could not find it in the existing literature.

\begin{theorem}\label{T:last}
Let $\Delta$ be the $n$-dimensional simplex and assume that the PL-Poincar{\'e} conjecture holds for all $m\leq n$. 
Then $\mathcal{H}_{\star}(\Delta)= \mathcal{C}_{\star}(\Delta)$ and $\mathcal{H}(\Delta)= \mathcal{C}(\Delta)$.
\end{theorem}

\begin{proof}
By Theorems~\ref{Theorem:HeredNearHomeoand Selection} and \ref{Theorem:CellularAndSelections}, 
it suffices to show that $\mathcal{H}_{\star}(\Delta)\subseteq \mathcal{C}_{\star}(\Delta)$. Let $f\colon B\to A$ be in $\mathcal{H}_{\star}(\Delta)$ and let 
$\mathcal{S}_{f}=\{D(\sigma,f_\mathrm{Y})\mid \; \mathrm{Y}\in \Delta, \; \sigma\in A_{\mathrm{Y}}\}$. It follows from Theorem~\ref{Cohen's original theorem} that 
$\mathcal{S}_{f}$ is a system and every $D\in \mathcal{S}_{f}$ is a stellar manifold of dimension $(\#\mathrm{Y}-\#\tau)$ whose boundary is the union of all 
$C\subsetneq D$ with $C\in\mathcal{S}_{f}$. 
We will show by induction on  the height of $D\in\mathcal{S}_{f}$ in the poset 
$(\mathcal{S}_{f},\subseteq)$ that $D$ is a stellar ball.

Note that, by \cite[Theorem~3.2]{Cu} and \cite[Theorem~1.6]{La}, 
$f$ is cell-like. This property and the implication \cite[Theorem~1.2(a)$\Rightarrow$(c)]{La} give that 
for each subcomplex $C$ of $A$, if $|C|_{\mathbb{R}}$ is contractible, then so is $|f^{-1}(C)|_{\mathbb{R}}$. In particular, $D$ is 
contractible. 


It is clear that the minimal elements in $(\mathcal{S}_{f},\subseteq)$ are $0$-balls. Let now $D=D(\tau,f_{\mathrm{Y}})\in \mathcal{S}_{f}$  and assume that 
every $C\in\mathcal{S}_{f}$ with $C\subsetneq D$ is a stellar ball. But then $\mathcal{S}^{<D}_{f}:=\{C\in\mathcal{S}_{f}\mid C\subsetneq D\}$ is a cell-system 
whose union is $\partial D$. As in the proof of Lemma \ref{L: Cellular has composition,anti-composition,triangulations}, it follows then that $\partial D$ is stellar 
equivalent with the complex $\beta \big( \partial \big( [\{\tau\}]\star \mathrm{lk}(\tau, A_{\mathrm{X}})\big)  \big )$, which is a stellar sphere. Let now $v$ be 
a new point and consider the complex $S:=  \big([v]\star(\partial D)\big) \cup D$. Since $D$ and $\partial D$ are combinatorial manifolds, so is $S$.  
Since $D$ is contractible, $S$ is homotopy equivalent to a sphere of dimension $\leq n$. By the PL-Poincar{\'e} conjecture, 
we have that $S$ is a stellar sphere. By the ``excision'' theorem for stellar spheres due to Newman \cite{Ne}, see \cite[Theorem~3.8]{Li}, 
it follows that $D$ is a stellar ball.  
\end{proof}

\noindent {\bf Acknowledgment.} We would like to thank Jan van Mill, Henryk Toru{\'n}czyk, and Jim West for their valuable comments.

\end{document}